\documentclass[11pt,reqno]{amsart}
\usepackage[margin=1in]{geometry}
\usepackage{amsmath,amssymb,amsthm,graphicx,amsxtra, setspace}
\usepackage[utf8]{inputenc}
\usepackage{mathrsfs}
\usepackage{hyperref}
\usepackage{upgreek}
\usepackage{mathtools}
\usepackage{hyperref}
\usepackage{alltt}
\usepackage{fouridx}
\usepackage{dsfont}
\usepackage{dsfont}
\usepackage{cancel}
\usepackage{xcolor}
\usepackage[mathcal]{euscript}
\usepackage{shadethm}
\usepackage{float}
\allowdisplaybreaks

\usepackage[pagewise]{lineno}

\newtheorem{theorem}{Theorem}[section]
\newtheorem{lemma}[theorem]{Lemma}
\newtheorem{proposition}[theorem]{Proposition}

\newtheorem{definition}[theorem]{Definition}

\newtheorem{remark}[theorem]{Remark}
\newtheorem{hypothesis}[theorem]{Hypothesis}

\let\originalleft\left
\let\originalright\right
\renewcommand{\left}{\mathopen{}\mathclose\bgroup\originalleft}
\renewcommand{\right}{\aftergroup\egroup\originalright}

\hypersetup{colorlinks=true,%
	citecolor=red,%
	filecolor=blue,%
	linkcolor=blue,%
}
\usepackage{graphicx}

\newcommand{\Tr}{\mathop{\mathrm{Tr}}}
\renewcommand{\d}{\/\mathrm{d}\/}

\def\w{\textbf{W}^{\varepsilon}_{{\theta}^{\varepsilon}}}

\def\e{\varepsilon}

\def\t{t\wedge\tau_N}
\def\s{t\wedge\tau_N}

\def\T{T\wedge\tau_N}
\def\L{\mathbb{L}}
\def\A{\mathrm{A}}
\def\I{\mathrm{I}}

\def\C{\mathrm{C}}
\def\f{\boldsymbol{f}}

\def\B{\mathrm{B}}
\def\D{\mathrm{D}}
\def\y{\boldsymbol{y}}
\def\boldq{\boldsymbol{q}}
\def\bolde{\boldsymbol{e}}

\def\E{\mathbb{E}}
\def\X{\mathbb{X}}
\def\x{\boldsymbol{x}}

\def\bolde{\boldsymbol{e}}

\def\z{\boldsymbol{z}}
\def\v{\boldsymbol{v}}
\def\V{\mathbb{v}}
\def\w{\boldsymbol{w}}
\def\W{\mathrm{W}}
\def\G{\mathrm{G}}
\def\Q{\mathrm{Q}}

\def\N{\mathbb{N}}

\def\no{\nonumber}
\def\V{\mathbb{V}}
\def\wi{\widetilde}
\def\Q{\mathrm{Q}}
\def\u{\mathrm{U}}
\def\P{\mathrm{P}}
\def\u{\boldsymbol{u}}
\def\H{\mathbb{H}}
\def\n{\boldsymbol{n}}

\newcommand{\R}{\mathbb{R}}

\renewcommand{\d}{\/\mathrm{d}\/}

\newcommand{\Addresses}{{
		\footnote{
			
			\noindent \textsuperscript{1}Center for Mathematics and Applications (NovaMath),  NOVA	SST,	Portugal.\par\nopagebreak
			\noindent  \textit{e-mail:} \texttt{kushkinra@gmail.com, k.kinra@fct.unl.pt.}
			
			\noindent \textsuperscript{2}Department of Mathematics, Indian Institute of Technology Roorkee-IIT Roorkee,
			Haridwar Highway, Roorkee, Uttarakhand 247667, INDIA.\par\nopagebreak
			\noindent  \textit{e-mail:} \texttt{manilfma@iitr.ac.in, maniltmohan@gmail.com.}
			
			\noindent \textsuperscript{*}Corresponding author.

			\textit{Key words:} convective Brinkman-Forchheimer equations, tamed Navier-Stokes equations, strong solution, exponential stability, invariant measure.
			
			Mathematics Subject Classification (2010): Primary 60H15; Secondary 35R60, 35Q30, 76D05.

}}}

\begin{document}
	
	
	\title[Stochastic convective Brinkman-Forchheimer equations]{Stochastic convective Brinkman-Forchheimer equations on general Unbounded Domains
			\Addresses}
		\author[K. Kinra]{Kush Kinra\textsuperscript{1}}
	\author[M. T. Mohan ]{Manil T. Mohan\textsuperscript{2*}}

	\maketitle
	
	\begin{abstract}
	The stochastic convective Brinkman-Forchheimer (SCBF)  equations  in an open connected set $\mathcal{O}\subseteq\R^d$ ($d\in \{2,3,4\}$) or torus are considered in this work.  We show the existence of a pathwise unique strong solution (in the probabilistic sense) satisfying the energy equality (It\^o formula) to SCBF equations perturbed by multiplicative Gaussian noise.  We exploited a monotonicity property of the linear and nonlinear operators as well as a stochastic generalization of  the Minty-Browder technique in the proofs. The energy equality is obtained by approximating the solution using approximate functions constituting the elements of eigenspaces of a compact operator in such a way that the approximations are bounded and converge in both Sobolev and Lebesgue spaces simultaneously.   We further discuss the global in time regularity results of such strong solutions on the torus. The exponential stability results (in mean square and pathwise sense) for the stationary solutions  is also established in this work for large effective viscosity. Moreover, a stabilization result of the stochastic convective Brinkman-Forchheimer  equations by using a multiplicative noise is obtained. Finally, when $\mathcal{O}$ is a bounded domain, we establish the existence of a unique invariant measure for the SCBF equations with multiplicative Gaussian noise, which is both ergodic and strongly mixing, using the exponential stability of strong solutions.
	\end{abstract}

	\section{Introduction}\label{sec1}\setcounter{equation}{0}
	Let $\mathcal{O}\subseteq\R^d$ ($2\leq d\leq 4$) be an open connected set  with a smooth boundary $\partial\mathcal{O}$. The	convective Brinkman-Forchheimer (CBF)  equations  are given by (see \cite{KT2} for Brinkman-Forchheimer equations with fast growing nonlinearities)
\begin{equation}\label{1}
\left\{
	\begin{aligned}
	\frac{\partial \u}{\partial t}-\mu \Delta\u+(\u\cdot\nabla)\u+\alpha\u+\beta|\u|^{r-1}\u+\nabla p&=\boldsymbol{f}, \ \text{ in } \ \mathcal{O}\times(0,T), \\ \nabla\cdot\u&=0, \ \text{ in } \ \mathcal{O}\times(0,T), \\
	\u&=\mathbf{0}\ \text{ on } \ \partial\mathcal{O}\times[0,T], \\
	\u(0)&=\u_0 \ \text{ in } \ \mathcal{O},\\
	\int_{\mathcal{O}}p(x,t)\d x&=0, \ \text{ in } \ (0,T).
	\end{aligned}
	\right.
	\end{equation}
Here, $\u(t , x) \in \R^d$ represents the velocity field at time $t$ and position $x$, $p(t,x)\in\R$ denotes the pressure field, $\f(t,x)\in\R^d$ is an external forcing. The final condition in \eqref{1} is imposed for the uniqueness of the pressure $p$. The constant $\mu$ represents the positive Brinkman coefficient (effective viscosity), the positive constants $\alpha$ and $\beta$ represent the Darcy (permeability of porous medium) and Forchheimer (proportional to the porosity of the material) coefficients, respectively. The absorption exponent $r\in[1,\infty)$ and  $r=3$ is known as the critical exponent.  The CBF equations \eqref{1} describe the motion of incompressible fluid flows in a saturated porous medium (see \cite{PAM}).   The model given in \eqref{1} is recognized to be more accurate when the flow velocity is too large for the Darcy's law to be valid alone, and apart from that, the porosity is not too small, so that we use the term \emph{non-Darcy models} for these types of flows (cf. \cite{PAM}). The nonlineartiy of the form $|\u|^{r-1}\u$ can be found in tidal dynamics as well as non-Newtonian fluid flows (see \cite{CTA,MTM5,MTM3}, etc and the references therein).   For $\alpha=\beta=0$, we obtain the classical $d$-dimensional  Navier-Stokes equations (NSE) (for example, see \cite{PCCF,FMRT,GGP,OAL,JCRWS,Te,Te1}, etc).   It has been established in Proposition 1.1, \cite{KWH}  that the critical homogeneous CBF equations have the same scaling as NSE only when $\alpha=0$ and no scale invariance property for other values of $\alpha$ and $r$. If $\mathcal{O}=\mathbb{R}^d$, then one has to replace   the boundary condition 	$\u=\mathbf{0}\text{ on } \partial\mathcal{O}\times[0,T]$ by the decay at infinity condition $|\u(x,t)|\to 0$ as $|x|\to\infty$ for all $t\in[0,T]$.

Let us now discuss some of the solvability results available in the literature for 3D CBF and related equations in the whole space as well as on the torus. The authors in \cite{ZCQJ} showed that the Cauchy problem for NSE with damping $r|\u|^{r-1}\u$ in the whole space has global weak solutions, for any $r\geq 1$, global strong solutions, for any $r\geq 7/2$ and that the strong solution is unique, for any $7/2\leq r\leq 5$. An improvement to this result was made in \cite{ZZXW} and the authors showed that the above mentioned problem has global strong solutions, for any $r>3$ and the strong solution is unique, when $3<r\leq 5$. Later, the authors in \cite{YZ}  proved that the strong solution exists globally for $r\geq 3$, and they established two regularity criteria, for $1\leq r<3$. For any $r\geq 1$, they proved that the strong solution is unique even among weak solutions.  The existence and uniqueness of a smooth solution to tamed 3D NSE in the whole space is proved in  \cite{MRXZ}.    A simple proof of the existence of global-in-time smooth solutions for the CBF equations \eqref{1} with $r>3$ on a 3D periodic domain is obtained in \cite{KWH}. The authors also proved that  global, regular solutions exist also for the critical value $r=3$, provided that the coefficients satisfy the relation $4\beta\mu\geq 1$. Furthermore, they showed that in the critical case every weak solution verifies the energy equality and hence is continuous into the phase space $\L^2$. The authors in \cite{KWH1} showed that the strong solutions of three dimensional CBF  equations in periodic domains with  the absorption exponent $r\in[1,3]$ remain strong under small enough changes of initial condition and forcing function. 

The authors in \cite{SNA} considered NSE modified by the absorption term $|\u|^{r-2}\u$, for $r>2$ in bounded domains with compact boundary. For this problem, they proved the existence of weak solutions in the Leray-Hopf sense, for any dimension $d\geq 2$ and its uniqueness for $d=2$. But  in three dimensions, they were not able to establish the energy equality satisfied by the weak solutions. The existence of regular dissipative solutions and global attractors for the system \eqref{1} with $r> 3$ (for $\f\in\mathrm{H}^1(0,T;\H)$) is established in \cite{KT2}. As a global smooth solution exists for $r>3$, the energy equality is satisfied in bounded domains. Recently, the authors in \cite{CLF} were able to construct functions that can approximate functions defined on smooth bounded domains by elements of eigenspaces of linear operators (for example, the Laplacian or Stokes operator) in such a way that the approximations are bounded and converge in both Sobolev and Lebesgue spaces simultaneously. As a simple application of this result, they proved that all weak solutions of the critical CBF equations ($r=3$) posed on a bounded domain in $\mathbb{R}^3$ satisfy the energy equality.  The existence and uniqueness of a global weak solution in the Leray-Hopf sense satisfying the energy equality to the system \eqref{1}	(for all $\beta,\mu>0$ for $r>3$ and $2\beta\mu\geq 1$ for $r=3$) is proved in \cite{Gautam+Mohan_2025}. The author exploited global monotonicity property of the linear and nonlinear operators, and Minty-Browder techniques in the proofs.

Let us now discuss some results available in the literature for the stochastic counterpart of the system \eqref{1} and related models in the whole space or on a torus. The existence of a unique strong solution 
\begin{align}\label{2}
\u\in\mathrm{L}^2(\Omega;\mathrm{L}^{\infty}(0,T;\H^1(\mathcal{O})))\cap\mathrm{L}^2(0,T;\H^2(\mathcal{O})),
\end{align} with $\mathbb{P}$-a.s., continuous paths in $\C([0,T];\H^1(\mathcal{O})),$  for $\u_0\in\mathrm{L}^2(\Omega;\H^1(\mathcal{O}))$,  to stochastic tamed 3D NSE in the whole space as well as in the periodic boundary case is obtained in \cite{MRXZ1}. They have proved the existence of a  unique  invariant measure for the corresponding transition semigroup also.  Recently, \cite{ZBGD} improved their results  for a slightly simplified system. The authors in \cite{WLMR} established the local and global existence and uniqueness of solutions for general deterministic and stochastic nonlinear evolution equations with coefficients satisfying some local monotonicity and generalized coercivity conditions. In \cite{WL}, the author showed the existence and uniqueness of strong solutions for a large class of SPDE, where the coefficients satisfy the local monotonicity and Lyapunov condition,  and he provided the stochastic tamed 3D NSE as an example. A large deviation principle of Freidlin-Wentzell type for stochastic tamed 3D NSE driven by multiplicative noise  in the whole space or on a torus  is established in \cite{MRTZ1}. All these works established the existence and uniqueness of strong solutions in the regularity class given in \eqref{2}. The global solvability of 3D NSE in the whole space with a Brinkman-Forchheimer type term subject to an anisotropic viscosity and a random perturbation of multiplicative type is described in the work \cite{HBAM}. 

Unlike whole space or periodic domains, in bounded domains, there is a technical difficulty in obtaining the regularity given in \eqref{2} of the velocity field $\u(\cdot)$ appearing in \eqref{1} with the fording $\f\in\mathrm{L}^2(0,T;\H)$.   It is  mentioned in the paper \cite{KT2} that the major difficulty in working with bounded domains is that $\mathcal{P}(|\u|^{r-1}\u)$ ($\mathcal{P}:\L^2(\mathcal{O})\to\H$ is the Helmholtz-Hodge orthogonal projection) need not be zero on the boundary, and $\mathcal{P}$ and $-\Delta$ are not necessarily commuting (see Example 2.19, \cite{JCR4}). Moreover, $\Delta\u\cdot\n\big|_{\partial\mathcal{O}}\neq 0$ in general and the term with pressure will not disappear (see \cite{KT2}), while taking inner product with $\Delta\u$ to the first equation in \eqref{1}. Therefore, the equality 
\begin{align}\label{3}
&\int_{\mathcal{O}}(-\Delta\u(x))\cdot|\u(x)|^{r-1}\u(x)\d x\nonumber\\&=\int_{\mathcal{O}}|\nabla\u(x)|^2|\u(x)|^{r-1}\d x+4\left[\frac{r-1}{(r+1)^2}\right]\int_{\mathcal{O}}|\nabla|\u(x)|^{\frac{r+1}{2}}|^2\d x\nonumber\\&=\int_{\mathcal{O}}|\nabla\u(x)|^2|\u(x)|^{r-1}\d x+\frac{r-1}{4}\int_{\mathcal{O}}|\u(x)|^{r-3}|\nabla|\u(x)|^2|^2\d x,
\end{align}
may not be useful in the context of bounded domains. The authors in \cite{MRTZ} showed the existence and uniqueness of strong solutions to stochastic 3D tamed NSE on bounded domains with Dirichlet boundary conditions. They have proved a small time large deviation principle for the solution also. The author in \cite{BYo} proved the existence of a random attractor for the 3D damped NSE in bounded domains with additive noise by verifying the pullback flattening property. The existence of a random attractor  ($r>3$, for any $\beta>0$) as well as the existence of a unique invariant measure ($3<r\leq 5$, for any $\beta>0$ and $\beta\geq\frac{1}{2}$ for $r=3$) for stochastic 3D NSE with damping driven by a multiplicative noise is established in the paper \cite{LHGH}. By using the classical Faedo-Galerkin approximation and compactness method, the existence of martingale solutions for stochastic 3D NSE with nonlinear damping is obtained in \cite{LHGH1}.  The works \cite{ZDRZ,LHGH2}, etc. considered various stochastic problems related to the equations similar to stochastic CBF equations in bounded domains with Dirichlet boundary conditions. As far as the strong solutions are concerned, some of these works proved regularity results in the space given in \eqref{2} by using the estimate given in \eqref{3}, which may not hold true in general domains. The well-posedness and asymptotic behavior of the stochastic CBF equations on bounded domains under pure jump noise perturbations are studied in \cite{Mohan_2022}.

In this work, we consider the following stochastic convective Brinkman-Forchheimer (SCBF) equations perturbed by multiplicative Gaussian noise: 
\begin{equation}\label{31}
	\left\{
	\begin{aligned}
		\d\u(t)&= [\mu \Delta\u(t) - (\u(t)\cdot\nabla)\u(t)-\beta|\u(t)|^{r-1}\u(t)-\nabla p(t)]\d t\\&\quad +\Phi(t,\u(t))\d\W(t), \ \text{ in } \ \mathcal{O}\times[0,T], \\ \nabla\cdot\u(t)&=0, \ \text{ in } \ \mathcal{O}\times(0,T), \\
		\u(t)&=\mathbf{0}\ \text{ on } \ \partial\mathcal{O}\times[0,T], \\
		\u(0)&=\u_0 \ \text{ in } \ \mathcal{O},
	\end{aligned}
	\right.
\end{equation} 
where $\W(\cdot)$ is a Hilbert space valued Wiener process and the function $\Phi:[0,T]\times(\H_0^1(\mathcal{O})\cap\L^{r+1}(\mathcal{O}))\to\mathcal{L}_{\Q}(\H)$ (see Section \ref{sec4} for details). We show the existence and uniqueness of strong solutions in the probabilistic sense in a larger space than \eqref{2} and discuss some asymptotic behavior. The main motivation of this work comes from the papers \cite{KWH} and \cite{CLF}. The work \cite{KWH} helped us to construct functions that can approximate functions defined on smooth bounded domains  by elements of eigenspaces of Stokes operator in such a way that the approximations are bounded and converge in both Sobolev and Lebesgue spaces simultaneously.  On the $d$-dimensional torus, one can approximate functions in $\L^p$-spaces using truncated Fourier expansions (see \cite{CLF}).   We also got inspiration from the work \cite{GK2}, where the	It\^o formula for processes taking values in intersection	of finitely many Banach spaces has been established. Because of the technical challenges mentioned above, alternative approaches may be required to establish the regularity of $\u(\cdot)$ in \eqref{2} for domains other than the whole space or torus. The novelties of this work are: 
\begin{itemize}
	\item The existence and uniqueness of strong solutions to SCBF equations defined on general unbounded domains ($r> 3,$ for any $\mu$ and $\beta$, $r=3$ for $2\beta\mu\geq 1$) with the divergence-free $\u_0\in\mathrm{L}^{4+\eta}(\Omega;\L^2(\mathcal{O}))$, for some $\eta>0$ is obtained in the space
	\begin{center} $\mathrm{L}^{4+\eta}(\Omega;\mathrm{L}^{\infty}(0,T;\L^2(\mathcal{O})))\cap\mathrm{L}^{2}(\Omega;\mathrm{L}^2(0,T;\H_0^1(\mathcal{O})))\cap\mathrm{L}^{r+1}(\Omega;\mathrm{L}^{r+1}(0,T;\L^{r+1}(\mathcal{O}))),$\end{center} with $\mathbb{P}$-a.s. paths in $\C([0,T];\L^2(\mathcal{O}))$. 
	\item The energy equality (It\^o's formula) satisfied by the SCBF equations is established by approximating the strong solution using the finite-dimensional space spanned by the first $n$ eigenfunctions of a suitable compact operator.
	\item The exponential stability (in the mean square and almost sure sense) of the stationary solutions is obtained for large $\mu$ and the lower bound on $\mu$ does not depend on the stationary solutions. A stabilization result of SCBF  equations by using a multiplicative noise is also obtained.
	\item  On bounded domains, the existence of a unique ergodic and strongly mixing invariant measure for SCBF  equations perturbed by multiplicative Gaussian noise is established by using the exponential stability of strong solutions. The case of unbounded domains will be addressed in our future work (see Remark \ref{rem-unbounded} below).
\end{itemize}

The organization of the paper is as follows. In the next section, we define the linear and nonlinear operators, and provide the necessary function spaces needed to obtain the global solvability results of the system \eqref{1}. For $r>3$, we show that the sum of linear and nonlinear operators is monotone (Theorem \ref{thm2.2}),  and for $r=3$ and $2\beta\mu\geq 1$, we show that the  sum is globally monotone (Theorem \ref{thm2.3}). The demicontinuity and hence the hemicontinuity property of these operators is also obtained in the same section (Lemma \ref{lem2.8}). Moreover, we also provide a compact operator in the same section which helps us to obtain the energy equality.  An abstract formulation of the SCBF equations \eqref{31} is formulated in Section \ref{sec4}. We establish the existence and uniqueness of global strong solution by making use of the monotonicity property of the linear and nonlinear operators as well as a stochastic generalization of the Minty-Browder technique (see Proposition \ref{prop-EE2} for a-priori energy estimates and Theorem \ref{exis} for global solvability results). We overcame the major difficulty of establishing the energy equality for SCBF equations by approximating the solution by using the finite-dimensional space spanned by the first $n$ eigenfunctions of the compact operator mentioned above. Owing to the technical difficulties discussed earlier, we establish the regularity results for global strong solutions only on the torus, assuming smooth initial data and imposing additional conditions on the noise coefficient (Theorem \ref{thm3.10}). The Section \ref{se5} is devoted for establishing the exponential stability (in the mean square and almost sure sense) of the stationary solutions (Theorems \ref{exp1} and \ref{exp2}) for large $\mu$. In both Theorems, the lower bound of $\mu$ does not depend on the stationary solutions.  A stabilization result for SCBF equations by using a multiplicative noise is also obtained in the same section (Theorem \ref{thm4.7}). In the final section, we prove the existence of a unique ergodic and strongly mixing invariant measure for SCBF  equations defined on bounded domains by using the exponential stability of strong solutions  (Theorem \ref{UEIM}).

\section{Mathematical Formulation}\label{sec3}\setcounter{equation}{0}
The necessary function spaces needed to obtain the global solvability results of the system \eqref{1} is provided in this section. We prove monotonicity as well as hemicontinuity properties of the linear and nonlinear operators in the same section. 

\subsection{Function spaces}\label{Function-spaces} 
Let $\C_0^{\infty}(\mathcal{O};\R^d)$ be the space of all infinitely differentiable functions  ($\R^d$-valued) with compact support in $\mathcal{O}\subset\R^d$.  Let us define 
\begin{align*} 
	\mathcal{V}&:=\{\u\in\C_0^{\infty}(\mathcal{O},\R^d):\nabla\cdot\u=0\},\\
	\mathbb{H}&:=\text{the closure of }\ \mathcal{V} \ \text{ in the Lebesgue space } \L^2(\mathcal{O})=\mathrm{L}^2(\mathcal{O};\R^d),\\
	\mathbb{V}&:=\text{the closure of }\ \mathcal{V} \ \text{ in the Sobolev space } \H_0^1(\mathcal{O})=\mathrm{H}_0^1(\mathcal{O};\R^d),\\
	\widetilde{\L}^{p}&:=\text{the closure of }\ \mathcal{V} \ \text{ in the Lebesgue space } \L^p(\mathcal{O})=\mathrm{L}^p(\mathcal{O};\R^d),
\end{align*}
for $p\in(2,\infty)$. The space of divergence-free test functions in the space-time domain is defined by 
\begin{align}\label{dvtest}
	\mathcal{V}_T:=\left\{\v\in\C_0^{\infty}(\mathcal{O}\times[0,T);\R^d):\nabla\cdot\v(\cdot,t)=0\right\}.
\end{align} 
It should be noted that $\v(x,T)=0$, for all $\v\in\mathcal{V}_T$. Then under some smoothness assumptions on the boundary, we characterize the spaces $\H$, $\V$ and $\widetilde{\L}^p$ as 
$
\H=\{\u\in\L^2(\mathcal{O}):\nabla\cdot\u=0,\u\cdot\mathbf{n}\big|_{\partial\mathcal{O}}=0\}$,  with norm  $\|\u\|_{\H}^2:=\int_{\mathcal{O}}|\u(x)|^2\d x,
$
where $\mathbf{n}$ is the outward normal to $\partial\mathcal{O}$,
$
\V=\{\u\in\H_0^1(\mathcal{O}):\nabla\cdot\u=0\},$  with norm $ \|\u\|_{\V}^2:=\int_{\mathcal{O}}|\nabla\u(x)|^2\d x,
$ and $\widetilde{\L}^p=\{\u\in\L^p(\mathcal{O}):\nabla\cdot\u=0, \u\cdot\mathbf{n}\big|_{\partial\mathcal{O}}=0\},$ with norm $\|\u\|_{\widetilde{\L}^p}^p=\int_{\mathcal{O}}|\u(x)|^p\d x$, respectively.
Let $(\cdot,\cdot)$ denote the inner product in the Hilbert space $\H$ and $\langle \cdot,\cdot\rangle $ denotes the induced duality between the spaces $\V$  and its dual $\V'$ as well as $\widetilde{\L}^p$ and its dual $\widetilde{\L}^{p'}$, where $\frac{1}{p}+\frac{1}{p'}=1$. Note that $\H$ can be identified with its dual $\H'$. From \cite[Subsection 2.1]{FKS}, we have that the sum space $\V'+\widetilde{\L}^{p'}$ is well defined and  is a Banach space with respect to the norm 
\begin{align}\label{22}
	\|\u\|_{\V'+\widetilde{\L}^{p'}}&:=\inf\{\|\u_1\|_{\V'}+\|\u_2\|_{\wi\L^{p'}}:\u=\u_1+\u_2, \y_1\in\V' \ \text{and} \ \y_2\in\wi\L^{p'}\}\nonumber\\&=
	\sup\left\{\frac{|\langle\u_1+\u_2,\f\rangle|}{\|\f\|_{\V\cap\widetilde{\L}^p}}:\boldsymbol{0}\neq\f\in\V\cap\widetilde{\L}^p\right\},
\end{align}
where $\|\cdot\|_{\V\cap\widetilde{\L}^p}:=\max\{\|\cdot\|_{\V}, \|\cdot\|_{\wi\L^p}\}$ is a norm on the Banach space $\V\cap\widetilde{\L}^p$. Also the norm $\max\{\|\u\|_{\V}, \|\u\|_{\wi\L^p}\}$ is equivalent to the norms  $\|\u\|_{\V}+\|\u\|_{\widetilde{\L}^{p}}$ and $\sqrt{\|\u\|_{\V}^2+\|\u\|_{\widetilde{\L}^{p}}^2}$ on the space $\V\cap\widetilde{\L}^p$. Moreover, we have the continuous embeddings $$\V\cap\widetilde{\L}^p\hookrightarrow\V\hookrightarrow\H\cong\H'\hookrightarrow\V'\hookrightarrow\V'+\widetilde{\L}^{p'}.$$ 
We denote $\H^2(\mathcal{O})=\W^{2,2}(\mathcal{O};\R^d)$ for the second order Hilbertian-Sobolev spaces. For the functional set up on the torus, interested readers are referred to see \cite{KWH,Te1}, etc. 

%

 The  following  interpolation inequality is also frequently in the paper. 
 \begin{lemma}
 	Assume $1\leq s\leq r\leq t\leq \infty$, $\theta\in(0,1)$ such that $\frac{1}{r}=\frac{\theta}{s}+\frac{1-\theta}{t}$ and $\u\in\L^s(\mathcal{O})\cap\L^t(\mathcal{O})$, then we have 
 	\begin{align}\label{211}
 		\|\u\|_{\L^r}\leq\|\u\|_{\L^s}^{\theta}\|\u\|_{\L^t}^{1-\theta}. 
 	\end{align}
 \end{lemma}

\subsection{Linear operator}\label{opeA}
It is well-known from \cite{DFHM,HKTY} that every vector field $\u\in\mathbb{L}^p(\mathcal{O})$, for $1<p<\infty$ can be uniquely represented as $\u=\v+\nabla q,$ where $\v\in\mathbb{L}^p(\mathcal{O})$ with $\mathrm{div \ }\v=0$ in the sense of distributions in $\mathcal{O}$ with $\v\cdot\n=0$ on $\partial\mathcal{O}$ (in the sense of trace) and $q\in\mathrm{W}^{1,p}(\mathcal{O})$ (Helmholtz-Weyl or Helmholtz-Hodge decomposition). For smooth vector fields in $\mathcal{O}$, such a decomposition is an orthogonal sum in $\mathbb{L}^2(\mathcal{O})$. Note that $\u=\v+\nabla q$ holds for all $\u\in\mathbb{L}^p(\mathcal{O})$, so that we can define the projection operator $\mathcal{P}_p$  by $\mathcal{P}_p\u = \v$. Let us consider the set $\mathcal{E}_{p}(\mathcal{O}):=\left\{\nabla q:q\in \mathrm{W}^{1,p}({\mathcal{O}})\right\}$ equipped with the norm $\|\nabla q\|_{\L^p}$. Then, from the above discussion, we obtain $\L^p(\mathcal{O})=\wi\L^p(\mathcal{O})\oplus\mathcal{E}_{p}(\mathcal{O})$.   From  \cite[Theorem 1.4]{CSHS}, we further have 
\begin{align*}
	\|\nabla q\|_{\L^p}\leq C\|\u\|_{\L^p}, \ \|\v\|_{\L^q}\leq (C+1)\|\u\|_{\L^p}\ \text{ and }\ \|\nabla q\|_{\L^p}+\|\v\|_{\L^q}\leq (2C+1)\|\u\|_{\L^p},
\end{align*}
where $C=C(\mathcal{O},p)>0$ is a constant such that \begin{align}\label{23}
	\|\nabla q\|_{\L^p}\leq C\sup_{0\neq\nabla\varphi\in \mathcal{E}_{p'}(\mathcal{O})}\frac{|\langle\nabla q,\nabla\varphi\rangle|}{\|\nabla\varphi\|_{\L^{p'}}}, \ \text{ for all }\ \nabla q\in \mathcal{E}_{p}(\mathcal{O}),
\end{align} 
with $\frac{1}{p}+\frac{1}{p'}=1$. Setting $\mathcal{P}_p\u:=\v$, we obtain a bounded linear operator $\mathcal{P}_p:\L^p(\mathcal{O})\to\wi\L^{p}(\mathcal{O})$ such that $\mathcal{P}_p^2=\mathcal{P}_p$ (projection). For $p=2$, $\mathcal{P}:=\mathcal{P}_2:\L^2(\mathcal{O})\to\H$ is an orthogonal projection.   Since $\mathcal{O}$ is of class $\C^2$, from \cite[Remark 1.6, Chapter 1, pp. 18]{Te}, we also infer that $\mathcal{P}$ maps $\H^1(\mathcal{O})$ into itself and is continuous for the norm of $\H^1(\mathcal{O})$.	
We define
\begin{equation*}
	\left\{
	\begin{aligned}
		\A\u:&=-\mathcal{P}\Delta\u,\;\u\in\D(\A),\\ \D(\A):&=\V\cap\H^{2}(\mathcal{O}).
	\end{aligned}
	\right.
\end{equation*}
It can be easily seen that the operator $\A$ is a non-negative self-adjoint operator in $\H$ with $\V=\D(\A^{1/2})$ and 
\begin{align}\label{2.7a}
	\langle \A\u,\u\rangle =\|\nabla\u\|_{\H}^2,\ \textrm{ for all }\ \u\in\V, \ \text{ so that }\ \|\A\u\|_{\V'}\leq \|\u\|_{\V}.
\end{align}

\begin{remark}
	For the bounded domain $\mathcal{O}$, the operator $\A$ is invertible and its inverse $\A^{-1}$ is bounded, self-adjoint and compact in $\H$. Thus, using spectral theorem, the spectrum of $\A$ consists of an infinite sequence $0< \lambda_1\leq \lambda_2\leq\ldots\leq \lambda_k\leq \ldots,$ with $\lambda_k\to\infty$ as $k\to\infty$ of eigenvalues.  Moreover, there exists an orthogonal basis $\{\boldsymbol{w}_k\}_{k=1}^{\infty} $ of $\H$ consisting of eigen functions of $\A$ such that $\A \boldsymbol{w}_k =\lambda_k\boldsymbol{w}_k$,  for all $ k\in\mathbb{N}$.  We know that $\u$ can be expressed as $\u=\sum_{k=1}^{\infty}\langle\u,\boldsymbol{w}_k\rangle \boldsymbol{w}_k$ and $\A\u=\sum_{k=1}^{\infty}\lambda_k\langle\u,\boldsymbol{w}_k\rangle \boldsymbol{w}_k$. Thus, it is immediate that 
	\begin{align}\label{poin}
		\|\nabla\u\|_{\mathbb{H}}^2=\langle \A\u,\u\rangle =\sum_{k=1}^{\infty}\lambda_k|\langle \u,\boldsymbol{w}_k\rangle|^2\geq \lambda_1\sum_{k=1}^{\infty}|\langle\u,\boldsymbol{w}_k\rangle|^2=\lambda_1\|\u\|_{\mathbb{H}}^2.
	\end{align}
\end{remark}

\begin{remark}
	It should be noted that, while proving global weak solutions, we are not using the Gagliardo-Nirenberg, Ladyzhenskaya or Agmon inequalities. The results obtained in this work are true for $2\leq d\leq 4$ in any smooth domains (for more details see step (iii) of the proof of the Theorem \ref{exis}). 
\end{remark}

\subsection{Bilinear operator}\label{opeB}
Let us define the \emph{trilinear form} $b(\cdot,\cdot,\cdot):\V\times\V\times\V\to\R$ by $$b(\u,\v,\w)=\int_{\mathcal{O}}(\u(x)\cdot\nabla)\v(x)\cdot\w(x)\d x=\sum_{i,j=1}^d\int_{\mathcal{O}}u_i(x)\frac{\partial v_j(x)}{\partial x_i}w_j(x)\d x.$$ If $\u, \v$ are such that the linear map $b(\u, \v, \cdot) $ is continuous on $\V$, the corresponding element of $\V'$ is denoted by $\B(\u, \v)$. We also denote (with an abuse of notation) $\B(\u) = \B(\u, \u)=\mathcal{P}(\u\cdot\nabla)\u$.
An integration by parts gives 
\begin{equation}\label{b0}
	\left\{
	\begin{aligned}
		b(\u,\v,\v) &= 0,\text{ for all }\u,\v \in\V,\\
		b(\u,\v,\w) &=  -b(\u,\w,\v),\text{ for all }\u,\v,\w\in \V.
	\end{aligned}
	\right.\end{equation}
In the trilinear form, an application of H\"older's inequality yields
\begin{align*}
	|b(\u,\v,\w)|=|b(\u,\w,\v)|\leq \|\u\|_{\widetilde{\L}^{r+1}}\|\v\|_{\widetilde{\L}^{\frac{2(r+1)}{r-1}}}\|\w\|_{\V},
\end{align*}
for all $\u\in\V\cap\widetilde{\L}^{r+1}$, $\v\in\V\cap\widetilde{\L}^{\frac{2(r+1)}{r-1}}$ and $\w\in\V$, so that we get 
\begin{align}\label{2p9}
	\|\B(\u,\v)\|_{\V'}\leq \|\u\|_{\widetilde{\L}^{r+1}}\|\v\|_{\widetilde{\L}^{\frac{2(r+1)}{r-1}}}.
\end{align}
Hence, the trilinear map $b : \V\times\V\times\V\to \R$ has a unique extension to a bounded trilinear map from $(\V\cap\widetilde{\L}^{r+1})\times(\V\cap\widetilde{\L}^{\frac{2(r+1)}{r-1}})\times\V$ to $\R$. It can also be seen that $\B$ maps $ \V\cap\widetilde{\L}^{r+1}$  into $\V'+\widetilde{\L}^{\frac{r+1}{r}}$ and using interpolation inequality (see \eqref{211}), we get 
\begin{align}\label{212}
	\left|\langle \B(\u,\u),\v\rangle \right|=\left|b(\u,\v,\u)\right|\leq \|\u\|_{\widetilde{\L}^{r+1}}\|\u\|_{\widetilde{\L}^{\frac{2(r+1)}{r-1}}}\|\v\|_{\V}\leq\|\u\|_{\widetilde{\L}^{r+1}}^{\frac{r+1}{r-1}}\|\u\|_{\H}^{\frac{r-3}{r-1}}\|\v\|_{\V},
\end{align}
for all $\v\in\V\cap\widetilde{\L}^{r+1}$. Thus, we have 
\begin{align}\label{2.9a}
	\|\B(\u)\|_{\V'+\widetilde{\L}^{\frac{r+1}{r}}}\leq\|\u\|_{\widetilde{\L}^{r+1}}^{\frac{r+1}{r-1}}\|\u\|_{\H}^{\frac{r-3}{r-1}}.
\end{align}
Using \eqref{2p9}, for $\u,\v\in\V\cap\widetilde{\L}^{r+1}$, we also have 
\begin{align}\label{lip}
	\|\B(\u)-\B(\v)\|_{\V'+\widetilde{\L}^{\frac{r+1}{r}}}&\leq \|\B(\u-\v,\u)\|_{\V'}+\|\B(\v,\u-\v)\|_{\V'}\nonumber\\&\leq \left(\|\u\|_{\widetilde{\L}^{\frac{2(r+1)}{r-1}}}+\|\v\|_{\widetilde{\L}^{\frac{2(r+1)}{r-1}}}\right)\|\u-\v\|_{\widetilde{\L}^{r+1}}\nonumber\\&\leq \left(\|\u\|_{\H}^{\frac{r-3}{r-1}}\|\u\|_{\widetilde{\L}^{r+1}}^{\frac{2}{r-1}}+\|\v\|_{\H}^{\frac{r-3}{r-1}}\|\v\|_{\widetilde{\L}^{r+1}}^{\frac{2}{r-1}}\right)\|\u-\v\|_{\widetilde{\L}^{r+1}},
\end{align}
for $r>3$, by using the interpolation inequality. For $r=3$, a calculation similar to \eqref{lip} yields 
\begin{align}
	\|\B(\u)-\B(\v)\|_{\V'+\widetilde{\L}^{\frac{4}{3}}}&\leq \left(\|\u\|_{\widetilde{\L}^{4}}+\|\v\|_{\widetilde{\L}^{4}}\right)\|\u-\v\|_{\widetilde{\L}^{4}},
\end{align}
hence $\B(\cdot):\V\cap\widetilde{\L}^{4}\to\V'+\widetilde{\L}^{\frac{4}{3}}$ is a locally Lipschitz operator.

\subsection{Nonlinear operator}\label{opeC}
Let us now define an operator $\mathcal{C}_r(\u):=\mathcal{P}(|\u|^{r-1}\u)$, for $\u\in \V\cap\wi\L^{r+1}$. For convenience of notation, we use $\mathcal{C}$ for $\mathcal{C}_r$ in the rest of the paper. Since the projection $\mathcal{P}$ is bounded from $\H^1$ into itself (see \cite[Remark 1.6]{Te}), the map $\mathcal{C}(\cdot):\V\cap\widetilde{\L}^{r+1}\to \V^{\prime} + \widetilde{\L}^{\frac{r+1}{r}}$ is well-defined and we have  $\langle\mathcal{C}(\u),\u\rangle =\|\u\|_{\widetilde{\L}^{r+1}}^{r+1}$. In addition, for all $\u\in \V\cap\wi\L^{r+1}$, the map $\mathcal{C}(\cdot):\V\cap\widetilde{\L}^{r+1}\to \V^{\prime} + \widetilde{\L}^{\frac{r+1}{r}}$ is G\^ateaux differentiable with its G\^ateaux derivative 
\begin{align}\label{Gaetu}
	\mathcal{C}'(\u)\v&=\left\{\begin{array}{cl}\mathcal{P}(\v),&\text{ for }r=1,\\ \left\{\begin{array}{cc}\mathcal{P}(|\u|^{r-1}\v)+(r-1)\mathcal{P}\left(\frac{\u}{|\u|^{3-r}}(\u\cdot\v)\right),&\text{ if }\u\neq \mathbf{0},\\\mathbf{0},&\text{ if }\u=\mathbf{0},\end{array}\right.&\text{ for } 1<r<3,\\ \mathcal{P}(|\u|^{r-1}\v)+(r-1)\mathcal{P}(\u|\u|^{r-3}(\u\cdot\v)), &\text{ for }r\geq 3,\end{array}\right.
\end{align}
for all $\v\in \widetilde{\L}^{r+1}$. An application of the Mean Value Theorem and H\"older's inequality, we estimate (see \cite[Subsection 2.4, pp. 8]{Gautam+Mohan_2025})
\begin{align}\label{213}
	&\langle \mathcal{P}(|\u|^{r-1}\u)-\mathcal{P}(|\v|^{r-1}\v),\u-\v\rangle
	\leq r\left(\|\u\|_{\widetilde{\L}^{r+1}}+\|\v\|_{\widetilde{\L}^{r+1}}\right)^{r-1}\|\u-\v\|_{\widetilde{\L}^{r+1}}^2,
\end{align}
for all $\u,\v\in\widetilde{\L}^{r+1}$.
Thus the operator $\mathcal{C}(\cdot) : \V\cap\widetilde{\L}^{r+1}\to \V^\prime + \widetilde{\L}^{\frac{r+1}{r}}$ is locally Lipschitz. Moreover, for any $r\in[1,\infty)$, we have (see \cite[Subsection 2.4, pp. 9]{Gautam+Mohan_2025})
\begin{align}\label{2.23}
	&\langle\mathcal{P}(\u|\u|^{r-1})-\mathcal{P}(\v|\v|^{r-1}),\u-\v\rangle\geq \frac{1}{2}\||\u|^{\frac{r-1}{2}}(\u-\v)\|_{\H}^2+\frac{1}{2}\||\v|^{\frac{r-1}{2}}(\u-\v)\|_{\H}^2\geq 0,
\end{align}
for $r\geq 1$.

\begin{proposition}
	For $\u,\v\in\V\cap\wi\L^{r+1}$, we have 
	\begin{align}
		\langle\mathcal{C}(\u)-\mathcal{C}(\v),\u-\v\rangle\geq\left\{\begin{array}{ll}\frac{1}{2}\|\u-\v\|_{\wi\L^{r+1}}^{r+1},&\text{ for } 1\leq r\leq 2,\\
		\frac{1}{2^{r-1}}\|\u-\v\|_{\wi\L^{r+1}}^{r+1},&\text{ for } 2\leq r<\infty.\end{array}\right. 
	\end{align}
\end{proposition}
\begin{proof}
	We know for $2\leq r<\infty$ that 
	\begin{align}
		\|\u-\v\|_{\wi\L^{r+1}}^{r+1}&=\int_{\mathcal{O}}|\u(x)-\v(x)|^{2}|\u(x)-\v(x)|^{r-1}\d x\nonumber\\&\leq 2^{r-2}\int_{\mathcal{O}}|\u(x)|^{r-1}|\u(x)-\v(x)|^{2}\d x+2^{r-2}\int_{\mathcal{O}}|\v(x)|^{r-1}|\u(x)-\v(x)|^{2}\d x\nonumber\\&=2^{r-2}\||\u|^{\frac{r-1}{2}}(\u-\v)\|_{\H}^2+2^{r-2}\||\v|^{\frac{r-1}{2}}(\u-\v)\|_{\H}^2\nonumber\\&\leq 2^{r-1}\langle\mathcal{C}(\u)-\mathcal{C}(\v),\u-\v\rangle,
	\end{align}
	where we have used \eqref{2.23} also. The case $1\leq r\leq 2$ follows in the similar lines. 
\end{proof}

\begin{theorem}{\cite[Theorem 2.5]{Gautam+Mohan_2025}}\label{thm2.2}
	Let $\u,\v\in\V\cap\widetilde{\L}^{r+1}$, for $r>3$. Then,	for the operator $\G(\u):=\mu \A\u +\B(\u) + \alpha \u +\beta\mathcal{C}(\u)$, we  have 
	\begin{align}\label{fe}
		\langle \G(\u)-\G(\v),\u-\v\rangle+\varrho\|\u-\v\|_{\H}^2\geq 0,
	\end{align}
	where
	\begin{align}\label{215}
		\varrho=\frac{r-3}{2\mu(r-1)}\left(\frac{2}{\beta\mu (r-1)}\right)^{\frac{2}{r-3}}.
	\end{align} 
	That is, the operator $\G+\varrho\mathrm{I}_d$, where $\I_d$ is an identity operator on $\H$, is a monotone operator from $\V\cap\widetilde{\L}^{r+1}$ to $\V'+\widetilde{\L}^{\frac{r+1}{r}}$. 
\end{theorem}
\begin{proof}
	We estimate $	\langle\A\u-\A\v,\u-\v\rangle $ by	using an integration by parts as
	\begin{align}\label{ae}
		\mu\langle\A\u-\A\v,\u-\v\rangle + \alpha\langle\u- \v,\u-\v\rangle =\mu\|\nabla(\u-\v)\|^2_{\H} + \alpha \|\u-\v\|^2_{\H}.
	\end{align}
	From \eqref{2.23}, we easily have 
	\begin{align}\label{2.27}
		\beta	\langle\mathcal{C}(\u)-\mathcal{C}(\v),\u-\v\rangle \geq \frac{\beta}{2}\||\v|^{\frac{r-1}{2}}(\u-\v)\|_{\H}^2. 
	\end{align}
	From \cite[Eqn. (2.16), pp. 10]{Gautam+Mohan_2025}, we write
	\begin{align}\label{2.30}
		&|\langle\B(\u-\v,\u-\v),\v\rangle| \leq\frac{\mu }{2}\|\nabla(\u-\v)\|_{\H}^2+\frac{\beta}{2}\||\v|^{\frac{r-1}{2}}(\u-\v)\|_{\H}^2+\varrho\|\u-\v\|_{\H}^2.
	\end{align}
	Combining \eqref{ae}-\eqref{2.30}, we get 
	\begin{align}
		\langle\G(\u)-\G(\v),\u-\v\rangle+\varrho\|\u-\v\|_{\H}^2\geq\frac{\mu }{2}\|\nabla(\u-\v)\|_{\H}^2 + \alpha \|\u-\v\|_{\H}^2\geq 0,
	\end{align}
	for $r>3$ and the estimate \eqref{fe} follows.  
\end{proof} 
\begin{theorem}\label{thm2.3}
	For the critical case $r=3$ with $2\beta\mu \geq 1$, the operator $\G(\cdot):\V\cap\widetilde{\L}^{4}\to \V'+\widetilde{\L}^{\frac{4}{3}}$ is globally monotone, that is, for all $\u,\v\in\V$, we have 
	\begin{align}\label{218}\langle\G(\u)-\G(\v),\u-\v\rangle\geq 0.\end{align}
\end{theorem}
\begin{proof}
	From \eqref{2.23}, we have 
	\begin{align}\label{231}
		\beta\langle\mathcal{C}(\u)-\mathcal{C}(\v),\u-\v\rangle\geq\frac{\beta}{2}\||\v|(\u-\v)\|_{\H}^2. 
	\end{align}
	We estimate $|\langle\B(\u-\v,\u-\v),\v\rangle|$ using H\"older's and Young's inequalities as 
	\begin{align}\label{232}
		|\langle\B(\u-\v,\u-\v),\v\rangle|\leq\|\v(\u-\v)\|_{\H}\|\nabla(\u-\v)\|_{\H} \leq\mu \|\nabla(\u-\v)\|_{\H}^2+\frac{1}{4\mu }\||\v|(\u-\v)\|_{\H}^2.
	\end{align}
	Combining \eqref{ae}, \eqref{231} and \eqref{232}, we obtain 
	\begin{align}
		\langle\G(\u)-\G(\v),\u-\v\rangle\geq\frac{1}{2}\left(\beta-\frac{1}{2\mu }\right)\||\v|(\u-\v)\|_{\H}^2 + \alpha \|\u-\v\|_{\H}^2\geq 0,
	\end{align}
	provided $2\beta\mu \geq 1$. 
\end{proof}

\begin{remark}
1. 	As in \cite{YZ}, for $r\geq 3$, One can estimate $|\langle\B(\u-\v,\u-\v),\v\rangle|$ as (see \cite{YZ} and \cite[Remark 2.7]{Gautam+Mohan_2025})
	\begin{align}\label{2.26}
&	|\langle\B(\u-\v,\u-\v),\v\rangle|
 \leq \mu \|\nabla(\u-\v)\|_{\H}^2+\frac{1}{4\mu } \||\v|^{\frac{r-1}{2}}(\u-\v)\|_{\H}^2 + \frac{1}{4\mu }\|\u-\v\|_{\H}^2,
	\end{align}
	which yields 
			\begin{align*}
			\langle\G(\u)-\G(\v),\u-\v\rangle+\frac{1}{4\mu}\|\u-\v\|_{\H}^2\geq 0,
		\end{align*}
		 provided $2\beta\mu \geq 1$. That is, the operator $\G+\frac{1}{4\mu}\mathrm{I}_d$ is a monotone operator from $\V\cap\widetilde{\L}^{r+1}$ to $\V'+\widetilde{\L}^{\frac{r+1}{r}}$.

	2. For $d=2$ and $r= 3$, one can estimate $|\langle\B(\u-\v,\u-\v),\v\rangle|$ using H\"older's, Ladyzhenskaya and Young's inequalities  as  (see \cite[Remark 2.7]{Gautam+Mohan_2025})
	\begin{align}\label{2.21}
	|\langle\B(\u-\v,\u-\v),\v\rangle|
	\leq  \frac{\mu }{2}\|\nabla(\u-\v)\|_{\H}^2+\frac{27}{16\mu ^3}\|\v\|_{\widetilde{\L}^4}^4\|\u-\v\|_{\H}^2.
	\end{align}
	which implies
	\begin{align}\label{fe2}
	\langle(\G(\u)-\G(\v),\u-\v\rangle+ \frac{27}{32\mu ^3}N^4\|\u-\v\|_{\H}^2\geq 0,
	\end{align}
	for all $\v\in\widehat{\mathbb{B}}_N$, where $\widehat{\mathbb{B}}_N$ is an $\widetilde{\L}^4$-ball of radius $N$, that is,
	$
	\widehat{\mathbb{B}}_N:=\big\{\z\in\widetilde{\L}^4:\|\z\|_{\widetilde{\L}^4}\leq N\big\}.
	$ Thus, the operator $\G(\cdot)$ is locally monotone in this case (see \cite{MJSS,MTM}, etc). 
	\end{remark}

Let us now show that the operator $\G:\V\cap\widetilde{\L}^{r+1}\to \V'+\widetilde{\L}^{\frac{r+1}{r}}$ is hemicontinuous, which is useful in proving the existence of weak solutions and strong solutions to the system \eqref{1}.

\begin{lemma}\label{lem2.8}
	The operator $\G:\V\cap\widetilde{\L}^{r+1}\to \V'+\widetilde{\L}^{\frac{r+1}{r}}$ is demicontinuous. 
\end{lemma}
\begin{proof}
	The proof follows from \cite[Lemma 2.8]{Gautam+Mohan_2025}.
\end{proof}

	\subsection{A compact operator}\label{C_O} We follow \cite[Subsection 2.3]{Brzezniak+Motyl_2013} for the contents of this section. 
Consider the natural embedding $j:\V\hookrightarrow\H$ and its adjoint $j^*:\H\hookrightarrow\V$. Since the range of $j$ is dense in $\H$, the map $j^*$ is one-to-one. Let us define
\begin{align}\label{L1}
	\D(\mathcal{A})&:=j^*(\H)\subset\V,\nonumber \\
	\mathcal{A}\u&:=(j^*)^{-1}\u, \ \ \u\in\D(\mathcal{A}).
\end{align}
Note that for all $\u\in\D(\mathcal{A})$ and $\v\in\V$
\begin{align*}
	(\mathcal{A}\u,\v)_{\H}=(\u,\v)_{\V}.
\end{align*}
Let us assume that $s>2$. It is clear that $\V_s$ is dense in $\V$ and the embedding $j_s:\V_s\hookrightarrow\V$ is continuous. Then, there exists a Hilbert space $\mathbb{U}$ (cf. \cite{Holly+Wiciak_1995}, \cite[Lemma C.1]{Brzezniak+Motyl_2013}) such that $\mathbb{U}\subset\V_s$, $\mathbb{U}$ is dense in $\V_s$ and 
\begin{align*}
	\text{  the natural embedding }\iota_s:\mathbb{U}\hookrightarrow\V_s \text{ is compact.}
\end{align*}
It implies that 
\begin{align*}
	\mathbb{U} \xhookrightarrow[\iota_s]{} \V_s\xhookrightarrow[j_s]{}\V\xhookrightarrow[j]{}\H\cong{\H}{'}\xhookrightarrow[j']{}\V'\xhookrightarrow[j'_s]{}  {\V}'_{s}\xhookrightarrow[\iota'_s]{}\mathbb{U}'.
\end{align*}
Consider the composition $$\iota:=j\circ j_s\circ\iota_s:\mathbb{U}\to\H$$ and its adjoint $$\iota^*:=(j\circ j_s\circ\iota_s)^*=\iota_s^*\circ j^*_s\circ j^*:\H\to \mathbb{U}.$$ We have that $\iota$ is compact and since its range is dense in $\H$, $\iota^*:\H\to\mathbb{U}$ is one-one. Let us define 
\begin{align}\label{L2}
	\D(\mathfrak{L})&:=\iota^*(\H)\subset\mathbb{U},\nonumber\\
	\mathfrak{L}\u&:=(\iota^*)^{-1}\u, \ \ \u\in\D(\mathfrak{L}).
\end{align}
Also we have that $\mathfrak{L}:\D(\mathfrak{L})\to\H$ is onto, $\D(\mathfrak{L})$ is dense in $\H$ and 
\begin{align*}
	(\mathfrak{L}\u,\w)_{\H}=(\u,\w)_{\mathbb{U}}, \ \ \ \ \u\in\D(\mathfrak{L}), \ \w\in\mathbb{U}.
\end{align*}
Furthermore, for $\u\in\D(\mathfrak{L})$,
\begin{align*}
	\mathfrak{L}\u=(\iota^*)^{-1}\u=(j^*)^{-1}\circ (j^*_s)^{-1}\circ(\iota^*_s)^{-1}\u=\mathcal{A}\circ (j^*_s)^{-1}\circ(\iota^*_s)^{-1}\u,
\end{align*}
where $\mathcal{A}$ is defined in \eqref{L1}. Since the operator $\mathfrak{L}$ is self-adjoint and $\mathfrak{L}^{-1}$ is compact, there exists an orthonormal basis $\{\boldsymbol{e}_i\}_{i\in\N}$ of $\H$ such that 
\begin{align}\label{L3}
	\mathfrak{L}\boldsymbol{e}_i=\nu_i\boldsymbol{e}_i,\ \ \ \ i\in\N,
\end{align}
that is, $\{\boldsymbol{e}_i\}_{i\in\N}$ are the eigenfunctions and $\{\nu_i\}_{i\in\N}$ are the corresponding eigenvalues of operator $\mathfrak{L}$. Note that $\boldsymbol{e}_i\in\mathbb{U}$, \ $i\in\N$, because $\D(\mathfrak{L})\subset \mathbb{U}$. 

Let us fix $m\in\N$ and let $\P_m$ be the operator from $\mathbb{U}'$ to $\mathrm{span}\{\boldsymbol{e}_1,\ldots,\boldsymbol{e}_m\}=:\H_{m}$ defined by 
\begin{align}\label{eqn-projection}
	\P_m\u^*:=\sum_{i=1}^{m}\langle\u^*,\boldsymbol{e}_i\rangle_{\mathbb{U}'\times\mathbb{U}}\boldsymbol{e}_i, \ \ \ \ \ \ \u^*\in\mathbb{U}'.
\end{align}
We will consider the restriction of operator $\P_m$ to the space $\H$ denoted still by the  same. In particular, we have $\H\hookrightarrow\mathbb{U}'$, that is, every element $\u\in\H$ induces a functional $\u^*\in\mathbb{U}'$ by 
\begin{align}
	\langle\u^*,\v\rangle_{\mathbb{U}'\times\mathbb{U}}:=(\u,\v), \ \ \ \ \v\in\mathbb{U}.
\end{align}
Thus the restriction of $\P_m$ to $\H$ is given by 
\begin{align}
	\P_m\u:=\sum_{i=1}^{m}(\u,\boldsymbol{e}_i)\boldsymbol{e}_i, \ \ \ \ \ \ \u\in\H.
\end{align}
Hence, in particular, $\P_m$ is the orthogonal projection from $\H$ onto $\H_m$. 
\begin{lemma}[{\cite[Lemma 2.4]{Brzezniak+Motyl_2013}}]
	For every $\u\in\mathbb{U}$ and $s>2$, we have 
	\begin{itemize}
		\item [(i)] $\lim\limits_{m\to\infty}\|\P_m\u-\u\|_{\mathbb{U}}=0$,
		\item [(ii)] $\lim\limits_{m\to\infty}\|\P_m\u-\u\|_{\mathbb{V}_s}=0$, 
		\item [(iii)] $\lim\limits_{m\to\infty}\|\P_m\u-\u\|_{\mathbb{V}}=0$.
	\end{itemize}
\end{lemma}

\section{Stochastic  convective Brinkman-Forchheimer  equations}\label{sec4}\setcounter{equation}{0} In this section, we consider the system \eqref{31} and discuss the existence and uniqueness of  strong solutions in probabilistic sense.

\subsection{Noise coefficient} Let $(\Omega,\mathscr{F},\mathbb{P})$ be a complete probability space equipped with an increasing family of sub-sigma fields $\{\mathscr{F}_t\}_{0\leq t\leq T}$ of $\mathscr{F}$ satisfying:
\begin{enumerate}
	\item [(i)] $\mathscr{F}_0$ contains all elements $F\in\mathscr{F}$ with $\mathbb{P}(F)=0$,
	\item [(ii)] $\mathscr{F}_t=\mathscr{F}_{t+}=\bigcap\limits_{s>t}\mathscr{F}_s,$ for $0\leq t\leq T$.
\end{enumerate} 

\begin{definition}
Let $\mathcal{H}$ be a given Hilbert space.	A stochastic process $\{\W(t)\}_{0\leq
		t\leq T}$ is said to be an \emph{$\mathcal{H}$-valued $\mathscr{F}_t$-adapted
		Wiener process} with covariance operator $\Q$ if
	\begin{enumerate}
		\item [$(i)$] for each non-zero $h\in \mathcal{H}$, $\|\Q^{1/2}h\|_{\mathcal{H}}^{-1} (\W(t), h)$ is a standard one-dimensional Wiener process,
		\item [$(ii)$] for any $h\in \mathcal{H}, (\W(t), h)$ is a martingale adapted to $\mathscr{F}_t$.
	\end{enumerate}
\end{definition}
The stochastic process $\{\W(t) \}_{0\leq t\leq T}$ is an $\mathcal{H}$-valued Wiener process with covariance $\Q$ if and only if for arbitrary $t$, the process $\W(t)$ can be expressed as 
\begin{align*}
	\W(t,x) =\sum_{k=1}^{\infty}\sqrt{\mu_k}\boldq_k(x)\upbeta_k(t),
\end{align*}
 where  $\upbeta_{k}(t),k\in\mathbb{N}$ are independent one-dimensional Brownian motions on $(\Omega,\mathscr{F},\mathbb{P})$ and $\{\boldq_k \}_{k=1}^{\infty}$ is the orthonormal basis of $\mathcal{H}$ such that $\Q \boldq_k=\mu_k \boldq_k$.  If $\W(\cdot)$ is an $\mathcal{H}$-valued Wiener process with covariance operator $\Q$ with $\Tr \Q=\sum_{k=1}^{\infty} \mu_k< +\infty$, then $\W(\cdot)$ is a Gaussian process on $\mathcal{H}$ and $ \E[\W(t)] = 0,$ $\textrm{Cov} [\W(t)] = t\Q,$ $t\geq 0.$ The space $\mathcal{H}_0:=\Q^{1/2}\mathcal{H}$ is a Hilbert space equipped with the inner product $(\cdot, \cdot)_0$, $$(\u, \v)_0 =\sum_{k=1}^{\infty}\frac{1}{\mu_k}(\u,\boldq_k)(\v,\boldq_k)= (\Q^{-1/2}\u, \Q^{-1/2}\v),\ \text{ for all } \ \u, \v\in \mathcal{H}_0,$$ where $\Q^{-1/2}$ is the pseudo-inverse of $\Q^{1/2}$.

Let $\mathcal{L}(\H)$ denote the space of all bounded linear operators on $\H$ and $\mathcal{L}_{\Q}:=\mathcal{L}_{\Q}(\H)$ represent the space of all Hilbert-Schmidt operators from $\H_0:=\Q^{1/2}\H$ to $\H$.  Since $\Q$ is a trace class operator, the embedding of $\H_0$ in $\H$ is Hilbert-Schmidt and the space $\mathcal{L}_{\Q}$ is a Hilbert space equipped with the norm $ \left\|\Phi\right\|^2_{\mathcal{L}_{\Q}}=\Tr\left(\Phi {\Q}\Phi^*\right)=\sum\limits_{k=1}^{\infty}\| {\Q}^{1/2}\Phi^*\boldq1_k\|_{\H}^2 =\sum\limits_{k=1}^{\infty}\| {\Q}^{1/2}\Phi^*\bolde_k\|_{\H}^2 $ and inner product $ \left(\Phi,\Psi\right)_{\mathcal{L}_{\Q}}=\Tr\left(\Phi {\Q}\Psi^*\right)=\sum\limits_{k=1}^{\infty}\left({\Q}^{1/2}\Psi^*\bolde_k,{\Q}^{1/2}\Phi^*\bolde_k\right) $. For more details, the interested readers are referred to see \cite{DaZ}.

\begin{hypothesis}\label{hyp}
	The noise coefficient $\Phi(\cdot,\cdot)$ satisfies: 
	\begin{itemize}
		\item [(H.1)] The function $\Phi\in\C([0,T]\times(\V\cap\wi\L^{r+1});\mathcal{L}_{\Q}(\H))$.
		\item[(H.2)]  (Growth condition)
		There exists a positive
		constant $K$ such that for all $t\in[0,T]$ and $\u\in\H$,
		\begin{equation}\label{Phi-Growth-condition}
		\|\Phi(t, \u)\|^{2}_{\mathcal{L}_{\Q}} 	\leq K\left(1 +\|\u\|_{\H}^{2}\right),
		\end{equation}
		
		\item[(H.3)]  (Lipschitz condition)
		There exists a positive constant $L$ such that for any $t\in[0,T]$ and all $\u_1,\u_2\in\H$,
		\begin{align}\label{Phi-Lipschitz-condition}
		\|\Phi(t,\u_1) - \Phi(t,\u_2)\|^2_{\mathcal{L}_{\Q}}\leq L\|\u_1 -	\u_2\|_{\H}^2.
		\end{align}
	\end{itemize}
\end{hypothesis}

\subsection{Abstract formulation of the stochastic system}\label{sec2.4}
On  taking orthogonal projection $\mathcal{P}$ onto the first equation in \eqref{31}, we get 
\begin{equation}\label{32}
\left\{
\begin{aligned}
\d\u(t)+[\mu \A\u(t)+\B(\u(t)) + \alpha\u(t) +\beta\mathcal{C}(\u(t))]\d t&=\Phi(t,\u(t))\d\W(t), \ t\in(0,T),\\
\u(0)&=\u_0.
\end{aligned}
\right.
\end{equation}
 Strictly speaking one should write $\mathcal{P}\Phi$ instead of $\Phi$.

Let us now provide the definition of a unique global strong solution in the probabilistic sense to the system (\ref{32}).
\begin{definition}[Global strong solution]\label{def-PSS}
	Let $\u_0\in\mathrm{L}^{4+\eta}(\Omega;\H)$ for some $\eta>0$ be given. 
	An $\H$-valued $(\mathscr{F}_t)_{t\geq 0}$-adapted stochastic process $\u(\cdot)$ is called a \emph{strong solution} to the system (\ref{32}) if the following conditions are satisfied: 
	\begin{enumerate}
		\item [(i)] the process $\u\in\mathrm{L}^{4+\eta}(\Omega;\mathrm{L}^{\infty}(0,T;\H))\cap\mathrm{L}^{2}(\Omega;\mathrm{L}^2(0,T;\V))\cap\mathrm{L}^{r+1}(\Omega;\mathrm{L}^{r+1}(0,T;\widetilde{\L}^{r+1}))$ and $\u(\cdot)$ has a $\V\cap\widetilde{\L}^{r+1}$-valued  modification, which is progressively measurable with continuous paths in $\H$ and $\u\in\C([0,T];\H)\cap\mathrm{L}^2(0,T;\V)\cap\mathrm{L}^{r+1}(0,T;\widetilde{\L}^{r+1})$, $\mathbb{P}$-a.s.,
		\item [(ii)] the following equality holds for every $t\in [0, T ]$, as an element of $\V'+\wi\L^{\frac{r+1}{r}},$ $\mathbb{P}$-a.s.
		\begin{align}\label{4.4}
		\u(t)&=\u_0-\int_0^t\left[\mu \A\u(s)+\B(\u(s))+ \alpha\u(s)+\beta\mathcal{C}(\u(s))\right]\d s+\int_0^t\Phi(s,\u(s))\d \W(s).
		\end{align}
			\item [(iii)] the following It\^o formula holds true: 
				\begin{align}\label{eqn-Ito-formula}
			&	\|\u(t)\|_{\H}^2+2\mu \int_0^t\|\nabla\u(s)\|_{\H}^2\d s + 2\alpha \int_0^t\|\u(s)\|_{\H}^2\d s +2\beta\int_0^t\|\u(s)\|_{\widetilde{\L}^{r+1}}^{r+1}\d s\nonumber\\&= \|{\u_0}\|_{\H}^2+\int_0^t\|\Phi(s,\u(s))\|_{\mathcal{L}_{\Q}}^2\d s+2\int_0^t(\Phi(s,\u(s))\d\W(s),\u(s)),
			\end{align}
			for all $t\in(0,T)$, $\mathbb{P}$-a.s.
	\end{enumerate}
\end{definition}
An alternative version of condition (\ref{4.4}) is to require that for any  $\v\in\V\cap\widetilde{\L}^{r+1}$:
\begin{align}\label{4.5}
(\u(t),\v)&=(\u_0,\v)-\int_0^t\langle\mu \A\u(s)+\B(\u(s)) + \alpha\u(s) +\beta\mathcal{C}(\u(s)),\v\rangle\d s\no\\&\quad+\int_0^t\left(\Phi(s,\u(s))\d \W(s),\v\right),\ \mathbb{P}\text{-a.s.}
\end{align}	
\begin{definition}
	A strong solution $\u(\cdot)$ to (\ref{32}) is called a
	\emph{pathwise  unique strong solution} if
	$\widetilde{\u}(\cdot)$ is an another strong
	solution, then $$\mathbb{P}\Big\{\omega\in\Omega:\u(t)=\widetilde{\u}(t),\ \text{ for all }\ t\in[0,T]\Big\}=1.$$ 
\end{definition}


\subsection{Energy estimates} In this subsection, we formulate a finite-dimensional system and establish some a-priori energy estimates. 
We define $\B^n(\u^n)=\mathrm{P}_n\B(\u^n)$, $\mathcal{C}^n(\u^n)=\mathrm{P}_n\mathcal{C}(\u^n)$,  $\Phi^n(\cdot,\u^n)=\mathrm{P}_n\Phi(\cdot,\u^n)$, and $\W^n=\mathrm{P}_n\W$, where $\mathrm{P}_n$ is the projection operator from $\mathbb{U}^{\prime}$ to $\H_n$ defined in Subsection \ref{C_O}. For each $n\in\N$, we search for approximate solutions of the form 
\begin{align}
	\u^n(x,t)=\sum_{k=1}^n g^n_k(t)\bolde_k(x), \;\; \bolde_k\in\H_n,
\end{align} 
where $\{\bolde_k\}_{k\in\N}$ is defined in \eqref{L3} and  the coefficients $g^n_1,\ldots,g^n_n$ are solutions of the following $n$ stochastic ordinary differential equations:
\begin{equation}\label{4.7}
\left\{
\begin{aligned}
\d\left(\u^n(t),\bolde_j\right)&=-\langle\mu \A\u^n(t)+\B^n(\u^n(t))+ \alpha\u^n(t)+\beta\mathcal{C}^n(\u^n(t)),\bolde_j\rangle\d t
\\ & \quad +\left(\Phi^n(t,\u^n(t))\d \W^n(t),\bolde_j\right),\\
(\u^n(0),\bolde_j)&=(\u_0^n,\bolde_j),
\end{aligned}
\right.
\end{equation}
for a.e. $t\in[0,T]$ and $j=1,\ldots,n$, where $\u_0^n=\mathrm{P}_n\u_0$. Since $\B^n(\cdot)$ and $\mathcal{C}^n(\cdot)$ are  locally Lipschitz (see \eqref{lip} and \eqref{213}), and  $\Phi^n(\cdot,\cdot)$   is globally Lipschitz (see \eqref{Phi-Lipschitz-condition}), the system (\ref{4.7}) has a unique $\H_n$-valued local strong solution $\u^n(\cdot)$ and $\u^n\in\mathrm{L}^2(\Omega;\mathrm{L}^{\infty}(0,T^*;\H_n))$ with $\mathscr{F}_t$-adapted continuous sample paths (see \cite{chow}).  Now we discuss the a-priori energy estimates satisfied by the system \eqref{4.7} and extend the time $T^*$ to $T$. Note that the energy estimates established in the next proposition is true for $r\geq 1$.

\begin{proposition}[Energy estimates]\label{prop-EE2}
	Let $\u^n(\cdot)$ be the unique solution of the system of stochastic
	ODEs (\ref{4.7}) with $\u_0\in\mathrm{L}^{2p}(\Omega;\H)$ with $p\geq1$. Then, under Hypothesis \ref{hyp}, we have 
	\begin{align}\label{5.5}
		&\E \left[   \sup_{t\in[0,T]} \|\u^n(t)\|^{2p}_{\H}+ 4p\int_{0}^{T}\|\u^n(s)\|^{2p-2}\left(\mu\|\nabla\u^n(s)\|^2_{\H}+\alpha\|\u^n(s)\|^2_{\H}+\beta\|\u^n(s)\|^{r+1}_{\wi\L^{r+1}}\right)\d s\right]\nonumber\\&
		\leq C \left( \E \left[\|\u_0\|_{\H}^{2p}\right]+T\right)e^{CT},
	\end{align}
	where $C$ is a positive constant.
\end{proposition}
\begin{proof}
		\noindent\textbf{Step (1):} Let us first define a sequence of stopping times $\tau_N$ by
	\begin{align}\label{stopm}
		\tau_N:=\inf_{t\geq 0}\left\{t:\|\u^n(t)\|_{\H}\geq N\right\},
	\end{align}
	for $N\in\mathbb{N}$. Applying the finite dimensional It\^{o} formula to the process
	$\|\u^n(\cdot)\|_{\H}^2$, we obtain 
	\begin{align}\label{4.9}
		&	\|\u^n(\t)\|_{\H}^2
		\nonumber\\ &=
		\|\u^n(0)\|_{\H}^2-2\int_0^{\t}\langle\mu \A\u^n(s)+\B^n(\u^n(s))+\alpha\u^n(s)+\beta\mathcal{C}^n(\u^n(s)),\u^n(s)\rangle\d s
		\nonumber\\&\quad+\int_0^{\t}\|\Phi^n(s,\u^n(s))\|^2_{\mathcal{L}_{\Q}}\d
		s 
		+2\int_0^{\t}\left(\Phi^n(s,\u^n(s))\d\W^n(s),\u^n(s)\right).
	\end{align}
	Note that $\langle\B^n(\u^n),\u^n\rangle=\langle\B(\u^n),\u^n\rangle=0$. 
	Taking expectation in (\ref{4.9}), and using the fact that final term in the right hand side of the equality (\ref{4.9}) is a martingale with zero expectation, we find
	\begin{align}\label{4.13}
		&\E\left[\|\u^n(\t)\|_{\H}^2+2\mu \int_0^{\t}\hspace{-3mm}\|\nabla\u^n(s)\|_{\H}^2\d s+2\alpha \int_0^{\t}\hspace{-3mm}\|\u^n(s)\|_{\H}^2\d s+2\beta\int_0^{\t}\hspace{-3mm}\|\u^n(s)\|_{\widetilde{\L}^{r+1}}^{r+1}\d s\right]\nonumber\\&\leq
		\E\left[	\|\u^n(0)\|_{\H}^2\right]+\E\left[\int_0^{\t}\|\Phi^n(s,\u^n(s))\|^2_{\mathcal{L}_{\Q}}\d
		s \right]\nonumber\\&\leq \E\left[	\|\u_0\|_{\H}^2\right]+K\E\left[\int_0^{\t}\left[1+\|\u^n(s)\|^2_{\H}\right]\d
		s \right]
		\nonumber\\&\leq \E\left[	\|\u_0\|_{\H}^2\right]+ KT + K\int_0^{t}\E\left[\|\u^n(s\land \tau_N)\|^2_{\H} \right]\d
		s,
	\end{align}
	where we have used Hypothesis \ref{hyp}  (H.2). Applying Gr\"onwall's inequality in (\ref{4.13}), we get 
	\begin{align}\label{2.7}
		\E\left[\|\u^n(\t)\|_{\H}^2\right]\leq
		\left(\E\left[\|\u_0\|_{\H}^2\right]+KT\right)e^{KT},
	\end{align}
	for all $t\in[0,T]$.  Note that for the indicator function $\chi$, 
	$$\E\left[{\chi}_{\{\tau_N^n<t\}}\right]=\mathbb{P}\Big\{\omega\in\Omega:\tau_N^n(\omega)<t\Big\},$$
	and using \eqref{stopm}, we obtain 
	\begin{align}\label{sq11}
		\E\left[\|\u^n(\t)\|_{\H}^2\right]&=
		\E\left[\|\u^n(\tau_N^n)\|_{\H}^2{\chi}_{\{\tau_N^n<t\}}\right]+\E\left[\|\u^n(t)\|_{\H}^2
		{\chi}_{\{\tau_N^n\geq t\}}\right]\nonumber\\&\geq
		\E\left[\|\u^n(\tau_N^n)\|_{\H}^2{\chi}_{\{\tau_N^n<t\}}\right]\geq
		N^2\mathbb{P}\Big\{\omega\in\Omega:\tau_N^n<t\Big\}.
	\end{align}
	Using the energy estimate (\ref{2.7}), we find 
	\begin{align}\label{sq12}
		\mathbb{P}\Big\{\omega\in\Omega:\tau_N<t\Big\}&\leq
		\frac{1}{N^2}\E\left[\|\u^n(\t)\|_{\H}^2\right]\leq
		\frac{1}{N^2}
		\left(\E\left[\|\u_0\|_{\H}^2\right]+KT\right)e^{KT}.
	\end{align}
	Hence, we have
	\begin{align}\label{sq13}
		\lim_{N\to\infty}\mathbb{P}\Big\{\omega\in\Omega:\tau_N<t\Big\}=0, \ \textrm{
			for all }\ t\in [0,T],
	\end{align}
	and $\t\to t$ as $N\to\infty$. 
	Taking limit $N\to\infty$ in
	(\ref{2.7}) and using the \emph{monotone convergence theorem}, we get 
	\begin{align}\label{4.16}
		\E\left[\|\u^n(t)\|_{\H}^2\right]\leq
		\left(\E\left[\|\u_0\|_{\H}^2\right]+KT\right)e^{KT},
	\end{align}
	for $0\leq t\leq T$. Substituting (\ref{4.16}) in (\ref{4.13}), we arrive at
	\begin{align}\label{4.16a}
		&	\E\left[\|\u^n(t)\|_{\H}^2+2\mu \int_0^{t}\|\nabla\u^n(s)\|_{\H}^2\d s +2\alpha \int_0^{t}\|\u^n(s)\|_{\H}^2\d s +2\beta\int_0^t\|\u^n(s)\|_{\widetilde{\L}^{r+1}}^{r+1}\d s\right]
		\nonumber\\ & \leq
		\left(\E\left[\|\u_0\|_{\H}^2\right]+KT\right)e^{2KT},
	\end{align}
	for all $t\in[0,T]$. 
	\vskip 0.2cm
	\noindent\textbf{Step (2):} Let us now prove \eqref{5.5}.	Applying finite dimensional It\^o's formula to the process $\|\u^n(\cdot)\|^{2p}_{\H}$, we obtain
	\begin{align}\label{S4}
		&\|\u^n(\t)\|^{2p}_{\H}+ 2p\int_{0}^{\t}\|\u^n(s)\|^{2p-2}\left(\mu\|\nabla\u^n(s)\|^2_{\H}+\alpha\|\u^n(s)\|^2_{\H}+\beta\|\u^n(s)\|^{r+1}_{\wi\L^{r+1}}\right)\d s\nonumber\\&
		=\underbrace{ \|\P_n\u_0\|^{2p}_{\H}}_{\leq \|\u_0\|^{2p}_{\H}} +  p \int_{0}^{\t}\|\u^n( s)\|^{2p-2}\|\Phi^n(s,\u^n(s))\|^2_{\mathcal{L}_{\Q}} \d s
		\nonumber\\ & \quad +  2 p\int_{0}^{\t}\|\u^n( s)\|^{2p-2}\left(\Phi^n(s,\u^n(s))\d\W^n(s) , \u^n( s)\right) \nonumber\\&\quad+ 2 p(p-1) \int_{0}^{\t}\|\u^n( s)\|^{2p-4}_{\H} \|(\Phi^n(s,\u^n(s)))^{*}\u^n( s)\|^2_{\H}\d s.
	\end{align}
	 Taking supremum from $0$
	to $\T$ before taking expectation in \eqref{S4}, we obtain
	\begin{align}\label{S44}
		&\E \left[ \sup_{t\in[0,\T]} \|\u^n(t)\|^{2p}_{\H}+ 2p\int_{0}^{\T}\|\u^n(s)\|_{\H}^{2p-2}\left(\mu\|\nabla\u^n(s)\|^2_{\H}+\alpha\|\u^n(s)\|^2_{\H}+\beta\|\u^n(s)\|^{r+1}_{\wi\L^{r+1}}\right)\d s\right]\nonumber\\&
		\leq \E \left[\|\u_0\|^{2p}_{\H}\right] +  p \E \left[\int_{0}^{\T}\|\u^n( s)\|_{\H}^{2p-2}\|\Phi^n(s,\u^n(s))\|^2_{\mathcal{L}_{\Q}} \d s\right]
		 \nonumber\\&\quad+ 2 p(p-1) \E \left[\int_{0}^{\T}\|\u^n( s)\|^{2p-4}_{\H} \|(\Phi^n(s,\u^n(s)))^{*}\u^n( s)\|^2_{\H}\d s\right]
		 \nonumber\\ & \quad +  2 p \E\left[ \sup_{t\in[0,\T]} \left|\int_{0}^{t}\|\u^n( s)\|_{\H}^{2p-2}\left(\Phi^n(s,\u^n(s))\d\W^n(s) , \u^n( s)\right)\right| \right]
		 \nonumber\\ & =   \E \left[\|\u_0\|^{2p}_{\H}\right] +\sum_{i=1}^{3}Q_{i}(t,n).
	\end{align}
	Let us now estimate each term of \eqref{S44} separately. Using Hypothesis \ref{hyp} (H.2), we get 
	\begin{align}\label{S5}
		| Q_1(t,n) + Q_2(t,n) | \leq C \E \left[ \int_{0}^{\T} \left\{ 1+ \|\u^n( s)\|_{\H}^{2p}  \right\}\d s \right].
	\end{align}
	Using Burkholder-Davis-Gundy (see \cite[Theorem 1]{BD} for the Burkholder-Davis-Gundy inequality (BDG) and \cite[Theorem 1.1]{DLB} for the best constant), Hypothesis \ref{hyp} (H.2), and H\"{o}lder's  and Young's inequalities, we have
	\begin{align}\label{S55}
		\left|Q_3(s,n)\right|  
		& \leq C \mathbb{E}\left[\left(\int_{0}^{\T}\|\u^n( s)\|^{4p-4}_{\H} \|(\Phi^n(s,\u^n(s)))^{*}\u^n( s)\|^2_{\H}\d s\right)^{\frac{1}{2}}\right]
		\nonumber\\ & \leq C \mathbb{E}\left[\left(\int_{0}^{\T}\|\u^n( s)\|^{4p-2}_{\H} \|(\Phi^n(s,\u^n(s)))^{*}\|^2_{\mathcal{L}_{\Q}}\d s\right)^{\frac{1}{2}}\right] 
		\nonumber\\ & \leq C \mathbb{E}\left[\left(\int_{0}^{\T}\|\u^n( s)\|^{4p-2}_{\H} \left\{1+ \|\u^n( s)\|^{2}_{\H}\right\}\d s\right)^{\frac{1}{2}}\right] 
		\nonumber\\ & \leq C \mathbb{E}\left[\sup_{t\in[0,\T]}\|\u^n(t)\|^{p}_{\H} \left(\int_{0}^{\T}\|\u^n( s)\|^{2p-2}_{\H} \left\{1+ \|\u^n( s)\|^{2}_{\H}\right\}\d s\right)^{\frac{1}{2}}\right]
		\nonumber\\ & \leq  \mathbb{E}\left[\frac12\sup_{t\in[0,\T]}\|\u^n(t)\|^{2p}_{\H} \right] + C \mathbb{E}\left[ \int_{0}^{\T} \left\{1+ \|\u^n( s)\|^{2p}_{\H}\right\}\d s \right].
	\end{align}
	By combining \eqref{S44}-\eqref{S55}, we reach at
	\begin{align}\label{S66}
		&\E \left[ \frac12 \sup_{t\in[0,\T]} \|\u^n(t)\|^{2p}_{\H}+ 2p\int_{0}^{\T} \hspace{-6mm} \|\u^n(s)\|_{\H}^{2p-2}\left(\mu\|\nabla\u^n(s)\|^2_{\H}+\alpha\|\u^n(s)\|^2_{\H}+\beta\|\u^n(s)\|^{r+1}_{\wi\L^{r+1}}\right)\d s\right]\nonumber\\&
		\leq \E \left[\|\u_0\|^{2p}_{\H}\right] + C \mathbb{E}\left[ \int_{0}^{\T} \left\{1+ \|\u^n( s)\|^{2p}_{\H}\right\}\d s \right].
	\end{align}
Applying Gr\"onwall's inequality in \eqref{S66}, we obtain 
\begin{align}\label{4.25}
	\E\left[\sup_{t\in[0,\T]}\|\u^n(t)\|_{\H}^{2p}\right]\leq\left(2 \E \left[\|\u_0\|_{\H}^{2p}\right]+CT\right)e^{CT}.
\end{align}
Substituting \eqref{4.25} in \eqref{S66} and passing $N\to\infty$, using the monotone convergence theorem, we finally obtain \eqref{5.5}. 
\end{proof}

\begin{lemma}\label{lem3.6}
	For $r>3$ and any functions $$\u,\v\in\mathrm{L}^2(\Omega;\mathrm{L}^{\infty}(0,T;\H)\cap\mathrm{L}^2(0,T;\V))\cap\mathrm{L}^{r+1}(\Omega;\mathrm{L}^{r+1}(0,T;\widetilde{\L}^{r+1})),$$  we have
	\begin{align}\label{3.11y}
	&\int_0^Te^{-\varrho t}\Big[\mu \langle \A(\u(t)-\v(t)),\u(t)-\v(t)\rangle +\langle \B(\u(t))-\B(\v(t)),\u(t)-\v(t)\rangle \nonumber\\&\quad+\beta\langle\mathcal{C}(\u(t))-\mathcal{C}(\v(t)),\u(t)-\v(t)\rangle \Big]\d t +\left(\alpha+\varrho+\frac{L}{2}\right)\int_0^Te^{-\varrho t}\|\u(t)-\v(t)\|_{\H}^2\d t\nonumber\\&\geq \frac{1}{2}\int_0^Te^{-\varrho t}\|\Phi(t, \u(t)) - \Phi(t,
	\v(t))\|^2_{\mathcal{L}_{\Q}}\d t,
	\end{align}
	where $\varrho$ is defined in \eqref{215} and $L$ is same as appearing in Hypothesis \ref{hyp} (H.3). 
\end{lemma}
\begin{proof}
	Multiplying (\ref{fe}) with $e^{-\varrho t}$, integrating over time $t\in(0,T)$, and using  Hypothesis \ref{hyp} (H.3), we obtain (\ref{3.11y}). 
\end{proof}

\subsection{Existence and uniqueness of probabilistic strong solution} Our next aim is to show that the system (\ref{32}) has a unique global strong solution by exploiting the monotonicity property (see (\ref{3.11y})) and a stochastic generalization of the Minty-Browder technique. The local monotonicity property of the linear and nonlinear operators and a stochastic generalization of the Minty-Browder technique has been used to obtain solvability results in the works \cite{ICAM,MJSS,MTM,MTM6,SSSP}, etc.

Let us recall a few things form \cite{Bogachev_2007} which will be used in the sequel.

\begin{definition}[{\cite[Definition 4.5.1]{Bogachev_2007}}]
	A set of functions $\mathcal{F} \subset \mathrm{L}^1(\Omega,\mathscr{F}, \mathbb{P})$ is called uniformly integrable if 
	\begin{align}
		\lim\limits_{M\to + \infty} \sup_{f\in\mathcal{F}} \int_{\{|f|>M\}} |f| \d \mathbb{P} =0,
	\end{align}
	that is, for every $\varepsilon>0$ there exists $M_\varepsilon >0$ such that 
	\begin{align}
		\sup_{f\in\mathcal{F}} \int_{\{|f| \geq M_\varepsilon\}} |f| \d \mathbb{P}  < \varepsilon.
	\end{align}
\end{definition}

\begin{theorem}[Lebesgue-Vitali theorem, {\cite[Theorem 4.5.4]{Bogachev_2007}}]\label{L-Vthm}
	Let $\{f_n\}_{n\in\N} \subset \mathrm{L}^p(\Omega,\mathscr{F}, \mathbb{P})$ and $f$ be an $\mathscr{F}$-measurable function. Then, the following are equivalent:
	\begin{enumerate}
		\item $f \in \mathrm{L}^p(\Omega,\mathscr{F}, \mathbb{P})$ and $\{f_n\}_{n\in\N}$ converges to $f$ in $\mathrm{L}^p(\Omega,\mathscr{F}, \mathbb{P})$;
		\item The sequence of  functions $\{f_n\}_{n\in\N}$ converges in $\mathbb{P}$-measure to $f$ and $\{|f_n|^p\}_{n\in\N}$ is uniformly integrable. 
	\end{enumerate}
\end{theorem}

\begin{remark}\label{rem-uniform-integable}
	From \cite[Page 31]{Billingsley_1999}, a simple condition for uniform integrability is that 
	\begin{align*}
		\sup_n\mathbb{E}\left[|f_n|^{1+\varepsilon}\right]<\infty,
	\end{align*}
	for some $\varepsilon>0$, in this case, an application of H\"older's and Markov's inequalities yield
	\begin{align*}
		 \int_{\{|f_n|>M\}} |f_n| \d \mathbb{P} \leq \frac{1}{M^{\varepsilon}}\mathbb{E}\left[|f_n|^{1+\varepsilon}\right]. 
	\end{align*}
\end{remark}

\begin{theorem}\label{exis}
For   $r> 3$, let $\u_0\in \mathrm{L}^{4+\eta}(\Omega;\H)$, for some $\eta>0$ be given.  Then under Hypothesis \ref{hyp}, there exists a \emph{pathwise unique strong solution}
	$\u(\cdot)$ to the system (\ref{32}) in the sense of Definition \ref{def-PSS} such that 
	\begin{align*}
		\u\in\mathrm{L}^{4+\eta}(\Omega;\mathrm{L}^{\infty}(0,T;\H))\cap\mathrm{L}^{2}(\Omega;\mathrm{L}^2(0,T;\V))\cap\mathrm{L}^{r+1}(\Omega;\mathrm{L}^{r+1}(0,T;\widetilde{\L}^{r+1})),
	\end{align*} with $\mathbb{P}$-a.s., continuous trajectories in $\H$.
\end{theorem}
\begin{proof}
	The global solvability results of the system (\ref{32}) is divided into the following steps.
	
	\vskip 0.2cm
	\noindent\textbf{Step (1):} \emph{Finite-dimensional (Galerkin) approximation of the system (\ref{32}):} Let us first consider the following Galerkin approximated It\^{o} stochastic differential equation satisfied by $\{\u^n(\cdot)\}$ in $\H_n$:
	\begin{equation}\label{4.37}
	\left\{
	\begin{aligned}
	\d\u^n(t)&=-\G(\u^n(t))\d
	t+\Phi^n(t,\u^n(t))\d\W^n(t),\\
	\u^n(0)&=\u_0^n,
	\end{aligned}
	\right.
	\end{equation}
	where
	$\G(\u^n(\cdot))=\mu \A\u^n(\cdot)+\B^n\left(\u^n(\cdot)\right)+\alpha\u^n(\cdot)+\beta\mathcal{C}^n(\u^n(\cdot))$. Applying It\^o's formula to the finite dimensional process $e^{-2\varrho t}\|\u^n(\cdot)\|_{\H}^2$, we obtain 
	\begin{align}\label{4.38}
	e^{-2\varrho t}\|\u^n(t)\|_{\H}^2&=\|\u^n(0)\|_{\H}^2-\int_0^te^{-2\varrho s}\langle 2\G(\u^n(s))+2\varrho \u^n(s),\u^n(s)\rangle\d
	s\nonumber\\&\quad+\int_0^te^{-2\varrho s}\|\Phi^n(s,\u^n(s))\|_{\mathcal{L}_{\Q}}^2\d
	s+2\int_0^te^{-2\varrho s}\left(\Phi^n(s,\u^n(s))\d\W^n(s),\u^n(s)\right),
		\end{align}
	for all $t\in[0,T]$. Note that the final term from the right hand side of the equality (\ref{4.38}) is a martingale and on taking expectation, we get 
	\begin{align}\label{4.39}
	\E\left[e^{-2\varrho t}\|\u^n(t)\|_{\H}^2\right]&=\E\left[\|\u^n(0)\|_{\H}^2\right]-\E\left[\int_0^te^{-2\varrho s}\langle2\G(\u^n(s))+2\varrho\u^n(s),\u^n(s)\rangle\d
	s\right]\nonumber\\&\quad+\E\left[\int_0^t
	e^{-2\varrho s}\|\Phi^n(s,\u^n(s))\|_{\mathcal{L}_{\Q}}^2\d s\right],
	\end{align}
	for all $t\in[0,T]$.
	
	\vskip 0.2cm
	\noindent\textbf{Step (2):} \emph{Weak convergence of the
		sequences $\u^n(\cdot)$, $\G(\u^n(\cdot))$ and
		$\Phi^n(\cdot,\cdot)$.} 
	We point out  that 	$\mathrm{L}^{4+\eta}\left(\Omega;\mathrm{L}^{\infty}(0,T;\H)\right)\cong	(\mathrm{L}^{\frac{4+\eta}{3+\eta}}\left(\Omega;\mathrm{L}^1(0,T;\H)\right)'$ and  the space $\mathrm{L}^{2}\left(\Omega;\mathrm{L}^1(0,T;\H)\right)$ is separable. Moreover,  the spaces $\mathrm{L}^2(\Omega;\mathrm{L}^2(0,T;\V))$ and $\mathrm{L}^{r+1}(\Omega;\mathrm{L}^{r+1}(0,T;\widetilde{\L}^{r+1}))$ are  reflexive ($\X''=\X$).  From the energy estimate \eqref{5.5} given in Proposition \ref{prop-EE2} (for $p=1$), we know that the sequence $\{\u^n(\cdot)\}_{n\in\N}$ is bounded in the spaces $\mathrm{L}^{4+\eta}\left(\Omega;\mathrm{L}^{\infty}(0,T;\H)\right)$, $\mathrm{L}^2(\Omega;\mathrm{L}^2(0,T;\V))$ and $\mathrm{L}^{r+1}(\Omega;\mathrm{L}^{r+1}(0,T;\widetilde{\L}^{r+1}))$. Then applying the	Banach-Alaoglu theorem, we can extract a subsequence	$\{\u^{n_k}\}_{k\in\N}$ of $\{\u^n\}_{n\in\N}$ such that	(for simplicity, we denote by the same index):
	\begin{equation}\label{4.40}
	\left\{
	\begin{aligned}
	\u^n&\xrightarrow{w^{*}} \u\textrm{ in
	}\mathrm{L}^{4+\eta}(\Omega;\mathrm{L}^{\infty}(0,T ;\H)),\\ 
	\u^n&\xrightarrow{w} \u\textrm{ in
	}\mathrm{L}^2(\Omega;\mathrm{L}^{2}(0,T ;\V)),\\  
\u^n&\xrightarrow{w} \u\textrm{ in
}\mathrm{L}^{r+1}(\Omega;\mathrm{L}^{r+1}(0,T ;\widetilde{\L}^{r+1})),\\
	\u^n(T)&\xrightarrow{w}\xi \in\mathrm{L}^2(\Omega;\H),\\
	\G(\u^n)&\xrightarrow{w} \G_0\textrm{ in
	}\mathrm{L}^2(\Omega;\mathrm{L}^2(0,T ;\V'))+\mathrm{L}^{\frac{r+1}{r}}(\Omega;\mathrm{L}^{\frac{r+1}{r}}(0,T;\widetilde{\L}^{\frac{r+1}{r}})).
	\end{aligned}
	\right.
	\end{equation}
Using H\"older's inequality, interpolation inequality (see \eqref{211})  and Proposition \ref{prop-EE2}, we justify the final convergence  in (\ref{4.5})  in the following way: 
	\begin{align}\label{4.41}
	&\E\left[\left|\int_0^T\langle\G(\u^n(t)),\v(t)\rangle\d t\right]\right|\nonumber\\&\leq \mu\E\left[ \left|\int_0^T(\nabla\u^n(t),\nabla\v(t))\d t\right|\right]+\E\left[\left|\int_0^T\langle \B(\u^n(t),\v(t)),\u^n(t)\rangle\d t\right|\right]\nonumber\\&\quad + \alpha\E\left[ \left|\int_0^T(\u^n(t),\v(t))\d t \right|\right] +\beta\E\left[\left|\int_0^T\langle|\u^n(t)|^{r-1}\u^n(t),\v(t)\rangle\d t\right|\right]\nonumber\\&\leq\mu\E\left[ \int_0^T\|\nabla\u^n(t)\|_{\H}\|\nabla\v(t)\|_{\H}\d t\right]+\E\left[\int_0^T\|\u^n(t)\|_{\widetilde{\L}^{r+1}}\|\u^n(t)\|_{\widetilde{\L}^{\frac{2(r+1)}{r-1}}}\|\v(t)\|_{\V}\d t\right]\nonumber\\&\quad + \alpha\E\left[ \int_0^T\|\u^n(t)\|_{\H}\|\v(t)\|_{\H}\d t\right] +\beta\E\left[\int_0^T\|\u^n(t)\|_{\widetilde{\L}^{r+1}}^{r}\|\v(t)\|_{\widetilde{\L}^{r+1}}\d t\right]\nonumber\\&\leq \mu \mathbb{E}\left[\left(\int_0^T\|\nabla\u^n(t)\|_{\H}^2\d t\right)^{1/2}\left(\int_0^T\|\nabla\v(t)\|_{\H}^2\d t\right)^{1/2}\right]\nonumber\\&\quad+\E\left[\left(\int_0^T\|\u^n(t)\|_{\widetilde{\L}^{r+1}}^{r+1}\d t\right)^{\frac{1}{r-1}}\left(\int_0^T\|\u^n(t)\|_{\H}^2\d t\right)^{\frac{r-3}{2(r-1)}}\left(\int_0^T\|\v(t)\|_{\V}^2\d t\right)^{1/2}\right]
	\nonumber\\&\quad + \alpha \mathbb{E}\left[\left(\int_0^T\|\u^n(t)\|_{\H}^2\d t\right)^{1/2}\left(\int_0^T\|\v(t)\|_{\H}^2\d t\right)^{1/2}\right] \nonumber\\&\quad+\beta\E\left[\left(\int_0^T\|\u^n(t)\|_{\widetilde{\L}^{r+1}}^{r+1}\d t\right)^{\frac{r}{r+1}}\left(\int_0^T\|\v(t)\|_{\widetilde{\L}^{r+1}}^{r+1}\d t\right)^{\frac{1}{r+1}}\right]\nonumber\\&\leq\mu \left\{\E\left(\int_0^T\|\nabla\u^n(t)\|_{\H}^2\d t\right)\right\}^{1/2}\left\{\E\left(\int_0^T\|\nabla\v(t)\|_{\H}^2\d t\right)\right\}^{1/2}\nonumber\\&\quad+T^{\frac{r-3}{2(r-1)}}\left\{\E\left(\int_0^T\|\u^n(t)\|_{\widetilde{\L}^{r+1}}^{r+1}\d t\right)\right\}^{\frac{1}{r-1}}\left\{\E\left(\sup_{t\in[0,T]}\|\u^n(t)\|_{\H}^2\right)\right\}^{\frac{r-3}{2(r-1)}}\left\{\E\left(\int_0^T\|\v(t)\|_{\V}^2\d t\right)\right\}^{\frac{1}{2}}\nonumber\\&\quad + \alpha \left\{\E\left(\int_0^T\|\u^n(t)\|_{\H}^2\d t\right)\right\}^{1/2}\left\{\E\left(\int_0^T\|\v(t)\|_{\H}^2\d t\right)\right\}^{1/2}\nonumber\\&\quad +\beta\left\{\E\left(\int_0^T\|\u^n(t)\|_{\widetilde{\L}^{r+1}}^{r+1}\d t\right)\right\}^{\frac{r}{r+1}}\left\{\E\left(\int_0^T\|\v(t)\|_{\widetilde{\L}^{r+1}}^{r+1}\d t\right)\right\}^{\frac{1}{r+1}}\nonumber\\&\leq  C(\E\left[\|\u_0\|_{\H}^2\right],\mu ,T,\alpha,\beta,K)
	\nonumber\\ & \quad \times \left[\left\{\E\left(\int_0^T\|\v(t)\|_{\H}^2\d t\right)\right\}^{\frac12} + \left\{\E\left(\int_0^T\|\nabla\v(t)\|_{\H}^2\d t\right)\right\}^{\frac12}+\E\left\{\left(\int_0^T\|\v(t)\|_{\widetilde{\L}^{r+1}}^{r+1}\d t\right)\right\}^{\frac{1}{r+1}}\right],
	\end{align}
	for all $\v\in\mathrm{L}^2(\Omega;\mathrm{L}^2(0,T;\V))\cap\mathrm{L}^{r+1}(\Omega;\mathrm{L}^{r+1}(0,T;\widetilde{\L}^{r+1}))$. 
	Using the Hypothesis \ref{hyp} (H.2) and energy	estimates  in Proposition \ref{prop-EE2},	we also have
	\begin{align}\label{4.42}
	\E\left[\int_0^{T
	}\|\Phi^n(t,\u^n(t))\|_{\mathcal{L}_{\Q}}^2\d
	t\right]&\leq K\E\left[\int_0^T\left(1+\|\u^n(t)\|_{\H}^2\right)\d t\right]\nonumber\\&\leq
	KT \left(1+ C \left(\E\left[\|\u_0\|_{\H}^2\right] + T\right) e^{C T} \right)<+\infty,
	\end{align}
and
	\begin{align}\label{4.43}
	\E\left[\left(\int_0^{T}\|\Phi^n(t,\u^n(t))\|_{\mathcal{L}_{\Q}}^2\d
	t\right)^2\right]
	&\leq K^2\E\left[\left(\int_0^T\left(1+\|\u^n(t)\|_{\H}^2\right)\d t\right)^2\right]
	\nonumber \\ & \leq T K^2\E\left[   \int_0^T \left(1+\|\u^n(t)\|_{\H}^2\right)^2 \d t \right]
	\nonumber \\ & \leq 2 T K^2\E\left[   \int_0^T \left(1+\|\u^n(t)\|_{\H}^4\right) \d t \right]
	\nonumber \\ & \leq 2 T^2 K^2 \left(1+\E \left[\sup_{t\in[0,T]}\|\u^n(t)\|_{\H}^4\right]\right)
	\nonumber \\ & \leq 2 T^2 K^2 C \left(1+\left(\E\left[\|\u_0\|_{\H}^4\right]+T\right)e^{CT} \right)  <+\infty.
\end{align}
	The right hand side of the estimate \eqref{4.43} is independent of $n$ and thus we can extract a subsequence
	$\{\Phi^{n_k}(\cdot,\u^{n_k})\}_{k\in\N}$ of $\{\Phi^n(\cdot,\u^n)\}_{n\in\N}$ such that (relabeled as $\{\Phi^n(\cdot,\u^n)\}$)
	\begin{equation}\label{4.43z}
	 \Phi^n(\cdot,\u^n)\mathrm{P}_n \xrightarrow{w}
	\Phi(\cdot)\ \textrm{ in	}\ \mathrm{L}^4(\Omega;\mathrm{L}^2(0,T ;\mathcal{L}_{\Q}(\H))).
	\end{equation}

		\vskip 0.2cm
	\noindent\textbf{Step (3):} \emph{It\^o stochastic differential satisfied by $\u(\cdot)$.} 	 We make use of the discussions  in  \cite[Theorem 7.5]{chow}, to obtain the It\^{o} stochastic differential satisfied by $\u(\cdot)$. Note that we cannot apply the results available in  \cite[Theorem 7.5]{chow} here directly as we have the weak convergence of the nonlinear term in  $\mathrm{L}^2(\Omega;\mathrm{L}^2(0,T ;\V'))+\mathrm{L}^{\frac{r+1}{r}}(\Omega;\mathrm{L}^{\frac{r+1}{r}}(0,T;\widetilde{\L}^{\frac{r+1}{r}})).$ 
	
	Due to technical reasons, we extend the time	interval from $[0,T]$ to an open interval $(-\e, T+\e)$ with	$\e>0$, and set the terms in the equation (\ref{4.37}) equal to	zero outside the interval $[0,T]$. Let $\phi\in\mathrm{H}^1(-\e,T+\e)$ be such that $\phi(0)=1$. 
	Also let us take $n\geq m$ so that $\H_m\subset \H_n$ and $\v\in \H_m$ so that $\P_m\v=\v$.
	For $\v\in \H_m,$ we define $\wi\v(t)=\phi(t)\v$.	Applying It\^{o}'s formula to the process	$(\u^n(t),\wi\v(t))$, one obtains
	\begin{align}\label{sq63}
	(\u^n(t),\wi\v(t))&=(\u^n(0),\wi\v(0))+\int_0^{T}(\u^n(t),\dot{\wi\v}_m(t))\d	t-\int_0^{T	}(\G(\u^n(t)),\wi\v(t))\d t \nonumber\\&\quad+\int_0^{T	}\left(\Phi^n(t,\u^n(t))\d\W^n(t),\wi\v(t)\right),
	\end{align}
where $\dot{\wi\v}_m(t)=\frac{\d\phi(t)}{\d t}\v_m$.	Note that, we can take the term by term limit $n\to\infty$ in (\ref{sq63}) by	making use of the weak convergences given in (\ref{4.40}) and (\ref{4.43z}).	For instance, we consider the stochastic integral present in  the equality (\ref{sq63}),	with $m$ fixed. Let $\mathcal{P}_{T}$ denote the class of	predictable processes with values in	$\mathrm{L}^2(\Omega;\mathrm{L}^2(0,T;\mathcal{L}_{\Q}(\H)))$ with the inner	product defined by	$$(\Phi,\Psi)_{\mathcal{P}_{T}}=\E\left[\int_0^{T}\Tr(\Phi(t)\Q\Psi^*(t))\d	t\right], \ \textrm{ for all }\ \Phi,\Psi\in\mathcal{P}_{T}.$$	Let us define a map $\upsilon:\mathcal{P}_{T}\to	\mathrm{L}^2(\Omega;\mathrm{L}^2(0,T))$ by	$$\upsilon(\Phi)=\displaystyle{\int_0^t}(\Phi(s)\d\W(s),\wi\v(s)),$$ for	all $t\in[0,T]$. Note that the map $\upsilon$ is linear and	continuous. The weak convergence  of	$\Phi^n(\cdot,\u^n)\mathrm{P}_n\xrightarrow{w} \Phi(\cdot)$ in	$\mathrm{L}^2(\Omega;\mathrm{L}^2(0,T ;\mathcal{L}_Q(\H)))$ (see	(\ref{4.43z})) implies that	
$$\left(\Phi^n(t,\u^n(t))\P_{n}\d\W(t),\zeta\right)_{\mathcal{P}_{T}}\to	\left(\Phi(t)\d\W(t),\zeta\right)_{\mathcal{P}_{T}},$$
 for all	$\zeta\in\mathcal{P}_{T}$ as $n\to\infty$. From this, as	$n\to\infty$, we conclude that
	\begin{align*}
		\upsilon(\Phi^n(t,\u^n(t)))&=\int_0^t(\Phi^n(t,\u^n(t))\P_{n}\d\W(s),\wi\v(s))
	\to \int_0^t(\Phi(t)\d\W(s),\wi\v(s)),
\end{align*} for all $t\in[0,T]$	and for each $m$. 
   Passing to limits term wise in the equation (\ref{sq63}), we get
	\begin{align}\label{sq64}
	&(\xi,\v)\phi(T)
	\nonumber\\& =(\u_0,\v)+\int_0^{T
	}(\u(t),\frac{\d\phi(t)}{\d t}{\v})\d
	t-\int_0^{T
	}\phi(t)\langle\G_0(t),\v\rangle\d t +\int_0^{T
	}\phi(t)\left(\Phi(t)\d \W(t),\v\right).
	\end{align}
	Let us now choose a subsequence $\{\phi_k\}\in\mathrm{H}^1(-\e,T+\e)$	with $\phi_k(0)=1$, for $k\in\mathbb{N}$, such that	$\phi_k\to\displaystyle{\chi}_t$ and the time derivative of $\phi_k$	converges to $\delta_t$, where $\displaystyle{\chi}_t(s)=1$, for	$s\leq t$ and $0$ otherwise, and $\delta_t(s)=\delta(t-s)$ is the	Dirac $\delta$-distribution. Using $\phi_k$ in place of $\phi$ in	(\ref{sq64}) and then letting $k\to\infty$, we obtain
	\begin{align}\label{sq666}
	(\u(t),\v)=(\u_0,\v)-\int_0^t\langle\G_0(s),\v\rangle\d s +\int_0^t(\Phi(s)\d
	\W(s),\v),\ \mathbb{P}\text{-a.s.,}
	\end{align}
for all $0<t<T$	with $(\u(T),\v)=(\xi,\v)$, for all $\v\in\H_m$. Since \eqref{sq666} holds for any $\v\in \bigcup\limits_{n=1}^{\infty}\H_m$ and the set $\bigcup\limits_{n=1}^{\infty}\H_m$ is a sense subspace of $\V\cap\wi\L^{r+1}$, by using a standard continuity argument we can show that the following holds
\begin{align}\label{sq66}
	(\u(t),\v)=(\u_0,\v)-\int_0^t\langle\G_0(s),\v\rangle\d s +\int_0^t(\Phi(s)\d
	\W(s),\v),\ \mathbb{P}\text{-a.s.,}
\end{align}
 for any $\v\in \V\cap\wi\L^{r+1}.$ Hence, $\u(\cdot)$ satisfies the following stochastic differential equation: 
	\begin{equation}\label{4.44}
	\left\{
	\begin{aligned}
	\d\u(t)&=-\G_0(t)\d
	t+\Phi(t)\d\W(t),\\
	\u(0)&=\u_0,
	\end{aligned}
	\right.
	\end{equation}
 in $\V^{\prime}+\widetilde{\L}^{\frac{r+1}{r}}$. Moreover,  for  $\u_0\in \mathrm{L}^{4+\eta}(\Omega;\H)$ (for some $\eta>0$), by using the weakly lowersemicontinuity property of the norms, the following energy estimate holds true: 
 \begin{align} \label{eng1}
 	&\E\left[\sup_{t\in[0,T]}\|\u(t)\|_{\H}^{4+\eta}+4\mu \int_0^{T}\|\u(t)\|_{\V}^2\d t+4\beta\int_0^T\|\u(t)\|_{\widetilde{\L}^{r+1}}^{r+1}\d t\right]\nonumber\\&\leq\liminf_{n\to\infty}\E\left[\sup_{t\in[0,T]}\|\u^n(t)\|_{\H}^{4+\eta}+4\mu \int_0^{T}\|\u^n(t)\|_{\V}^2\d t+4\beta\int_0^T\|\u^n(t)\|_{\widetilde{\L}^{r+1}}^{r+1}\d t\right]\nonumber\\&\leq C\left(\E\left[\|\u_0\|_{\H}^{4+\eta}\right]+T\right)e^{CT}.
 \end{align}  

		\vskip 0.2cm
	\noindent\textbf{Step (4):} \emph{Energy equality satisfied by $\u(\cdot)$.} 
Let us now establish the energy equality satisfied by $\u(\cdot)$. Note that such an energy equality is not immediate due to the final convergence in \eqref{4.40} and we cannot apply the infinite dimensional It\^o formula available in the literature for semimartingales (see \cite[Theorem 1]{GK1} and  \cite[Theorem 6.1]{Me}). In \cite{CLF}, the authors established an approximation of $\u(\cdot)$ in bounded domains	such that the approximations are bounded and converge in both Sobolev and Lebesgue spaces simultaneously (cf. \cite{KWH} for such an approximation in periodic domains). In particular, they approximate $\u(t)$, for each $t\in [0,T],$ by using the finite-dimensional space spanned by the first $n$ eigenfunctions of the Stokes operator $\A$. Since we are working on unbounded domains, we do not have the existence of eigenfunctions of the Stokes operator. Therefore, we use the eigenfunctions of the operator $\mathfrak{L}$ (cf. \eqref{L2}-\eqref{L3}) to obtain a sequence which approximates $\u(\cdot)$. 
 Set 
\begin{align}\label{3.32}
	\u_n(t):=\mathrm{P}_{1/n}\u(t)=\sum_{\nu_j<n^2}e^{-\nu_j/n}\langle\u(t),\boldsymbol{e}_j\rangle_{\mathbb{U}'\times\mathbb{U}} \boldsymbol{e}_j.
\end{align}
Note that $\mathrm{P}_{1/n}$ is a self-adjoint operator, but not a projection. For notational convenience, we use $\u_n$ for approximations of the type \eqref{3.32} and $\u^n$ for Galerkin approximations in Steps \textbf{(1), (2), (3)}. Note first that 
\begin{align}\label{3p34}
\|\u_n\|_{\H}^2=\|\mathrm{P}_{1/n}\u\|_{\H}^2&=\sum_{\nu_j<n^2}e^{-2\nu_j/n}|\langle\u,\bolde_j\rangle_{\mathbb{U}'\times\mathbb{U}}|^2 \leq\sum_{j=1}^{\infty}|\langle\u,\bolde_j\rangle_{\mathbb{U}'\times\mathbb{U}}|^2=\|\u\|_{\H}^2<+\infty,
\end{align}
for all $\u\in\H$. It can also be seen that
\begin{align}\label{3.36}
&\|(\mathrm{I}_d-\mathrm{P}_{1/n})\u\|_{\H}^2
\nonumber\\&=\|\u\|_{\H}^2-2\langle\u,\mathrm{P}_{1/n}\u\rangle_{\mathbb{U}'\times\mathbb{U}}+\|\mathrm{P}_{1/n}\u\|_{\H}^2\nonumber\\&=\sum_{j=1}^{\infty}|\langle\u,\bolde_j\rangle_{\mathbb{U}'\times\mathbb{U}}|^2-2\sum_{\nu_j<n^2}e^{-\nu_j/n}|\langle\u,\bolde_j\rangle_{\mathbb{U}'\times\mathbb{U}}|^2+\sum_{\nu_j<n^2}e^{-2\nu_j/n}|\langle\u,\bolde_j\rangle_{\mathbb{U}'\times\mathbb{U}}|^2\nonumber\\&=\sum_{\nu_j<n^2}(1-e^{-\nu_j/n})^2|\langle\u,\bolde_j\rangle_{\mathbb{U}'\times\mathbb{U}}|^2+\sum_{\nu_j\geq n^2}|\langle\u,\bolde_j\rangle_{\mathbb{U}'\times\mathbb{U}}|^2,
\end{align}
for all $\u\in\H$. Note that the final term on the right hand side of the equality \eqref{3.36} tends to zero as $n\to\infty$, since the series $\sum_{j=1}^{\infty}|\langle\u,e_j\rangle_{\mathbb{U}'\times\mathbb{U}}|^2$ is convergent. The first term on the right hand side of the equality can be bounded from above by $$\sum_{j=1}^{\infty}(1-e^{-\nu_j/n})^2|\langle\u,\bolde_j\rangle_{\mathbb{U}'\times\mathbb{U}}|^2\leq 4\sum_{j=1}^{\infty}|\langle\u,\bolde_j\rangle_{\mathbb{U}'\times\mathbb{U}}|^2=4\|\u\|_{\H}^2<+\infty.$$ By using the Dominated Convergence Theorem, one can interchange the limit and sum and hence we get 
$$\lim_{n\to\infty}\sum_{j=1}^{\infty}(1-e^{-\nu_j/n})^2|\langle\u,\bolde_j\rangle_{\mathbb{U}'\times\mathbb{U}}|^2=\sum_{j=1}^{\infty}\lim_{n\to\infty}(1-e^{-\nu_j/n})^2|\langle\u,\bolde_j\rangle_{\mathbb{U}'\times\mathbb{U}}|^2=0.$$
  Hence, from \eqref{3.36}, we have 
  \begin{align}\label{3.37}
  	\|(\mathrm{I}_d-\mathrm{P}_{1/n})\u\|_{\H}\to 0 \text{ as }n\to\infty.
  \end{align}
 Let us now discuss the properties of the approximation given in \eqref{3.32}. The authors in \cite{CLF}  showed that such an approximation  satisfies the following: 
\begin{equation}\label{approx-properties}
	\left\{
	\begin{aligned}
		(1) & 	\mbox{   $\u_n(t)\to\u(t)$ in $\H_0^1(\mathcal{O})$ with $\|\u_n(t)\|_{\H^1}\leq \mathfrak{C} \|\u(t)\|_{\H^1}$, for a.e. $t\in[0,T]$ } 
		\\ & \mbox{ and $\mathbb{P}$-a.s.,} 
		 \\
		(2) & 	\mbox{   $\u_n(t)\to\u(t)$ in $\L^{p}(\mathcal{O})$ with $\|\u_n(t)\|_{{\L}^{p}}\leq \mathfrak{C} \|\u(t)\|_{{\L}^{p}}$, for any $p\in(1,\infty)$, } \\ & \mbox{ for a.e. $t\in[0,T]$ and $\mathbb{P}$-a.s.,} \\
		(3) & 	  \mbox{   $\u_n(t)$ is divergence-free and zero on $\partial\mathcal{O}$,  for a.e. $t\in[0,T]$ and $\mathbb{P}$-a.s.}
	\end{aligned}
\right.
\end{equation}
In (1) and (2) of \eqref{approx-properties}, $\mathfrak{C}$ is an absolute constant.   Note that for $2\leq d\leq 4$ and $s>2$, $\D(\mathfrak{L})\subset \V_s \subset\H^2\subset\L^p,$ for all $p\in(1,\infty)$ (cf. Subsection \ref{C_O}). Since $\bolde_j$'s are the eigenfunctions of the  operator $\mathfrak{L}$, we get $\bolde_j\in\D(\mathfrak{L})\subset\V$ and $\bolde_j\in\D(\mathfrak{L})\subset\widetilde{\L}^{r+1}$.  We also need the fact 
\begin{align}\label{335}
	\|\u_n-\u\|_{\mathrm{L}^{r+1}(0,T;\widetilde{\L}^{r+1})} \to 0\ \text{ as }\ n\to\infty,\;\; \mathbb{P}\text{-a.s.}
\end{align}
which  follows from (3). Since $\u\in\mathrm{L}^{r+1}(\Omega;\mathrm{L}^{r+1}(0,T;\widetilde{\L}^{r+1}))$ and the fact that $\|\u^n(t,\omega)-\u(t,\omega)\|_{\widetilde{\L}^{r+1}}\to 0$, for a.e. $t\in[0,T]$ and a.e. $\omega\in\Omega$, one can obtain the above convergence by an application of the Dominated Convergence Theorem (with the dominating function $(1+\mathfrak{C})\|\u(t,\omega)\|_{\widetilde{\L}^{r+1}}$).  Moreover, using the fact that $\u\in\mathrm{L}^{2}(\Omega;\mathrm{L}^{2}(0,T;\V))$, we also have 
\begin{align}\label{335a}
	\|\u_n-\u\|_{\mathrm{L}^{2}(0,T;\V)}\to 0 \ \text{ as }\ n\to\infty, \;\; \mathbb{P}\text{-a.s.}
\end{align}

In order to prove the energy equality, we follow the ideas from the work \cite{Krylov_2010}. Note that from \eqref{sq66}, we have that for each $x$ for all $t\in[0,\infty)$
\begin{align}\label{sq67}
	\P_{1/n}\u(t,x) = \P_{1/n} \u_0(x) - \int_0^t \P_{1/n}\G_0(s,x) \d s +\int_0^t \P_{1/n} \Phi(s,x)\d
	\W(s),\ \mathbb{P}\text{-a.s.},
\end{align}
so that $\P_{1/n}\u(\cdot,\cdot)$ is a smooth function in the second variable. By using the finite dimensional It\^o formula, we have for each $x\in\mathcal{O}$
\begin{align}\label{eqn-pointwise-ito}
	|\P_{1/n}\u(t,x)|^2 & = |\P_{1/n}\u_0(x)|^2 - 2  \int_{0}^{t} \P_{1/n}\G_0(s,x) \cdot \P_{1/n} \u(s,x) \d s 
	\nonumber\\ & \quad  +  2 \int_0^t \P_{1/n} \u(s,x)\cdot\P_{1/n} \Phi(s,x) \d
	\W(s) 
	\nonumber\\ & \quad + \sum_{k=1}^{\infty} \int_{0}^{t} \left[\sum_{\nu_j < n^2 } e^{- \nu_j/ n} \sqrt{\mu_k} \left\langle\Phi(s)\boldq_k, \bolde_j\right\rangle_{\mathbb{U}^\prime \times \mathbb{U}} \bolde_j(x)\right]^2  \d s   ,  \ \mathbb{P}\text{-a.s.}
\end{align}
Now we integrate \eqref{eqn-pointwise-ito} over $\mathcal{O}$ and obtain
\begin{align}\label{eqn-int-O}
	\|\P_{1/n}\u(t)\|^2_{\H} & = \|\P_{1/n}\u_0\|^2_{\H} - 2 \int_{\mathcal{O}} \int_{0}^{t} \P_{1/n}\G_0(s,x) \cdot \P_{1/n} \u(s,x) \d s  \d x
	\nonumber\\ & \quad  +  2 \int_{\mathcal{O}} \int_0^t  \P_{1/n} \u(s,x)\cdot \P_{1/n} \Phi(s,x)\d
	\W(s)   \d x
	\nonumber\\ & \quad +  \int_{\mathcal{O}} \sum_{k=1}^{\infty} \int_{0}^{t} \left[\sum_{\nu_j < n^2 } e^{- \nu_j/ n} \sqrt{\mu_k} \left\langle\Phi(s)\boldq_k, \bolde_j\right\rangle_{\mathbb{U}^\prime \times \mathbb{U}} \bolde_j(x)\right]^2  \d s \d x ,  \ \mathbb{P}\text{-a.s.}
\end{align}
 We aim to apply both the deterministic and stochastic versions of Fubini’s theorem. There are no issues concerning the integral with respect to $\d s$. However, to use the stochastic Fubini theorem, we must ensure that the resulting stochastic integral is well-defined. In particular, we require that the following integral is finite $\mathbb{P}$ a.s.:
\begin{align}
	\int_{0}^{t} \sum_{k=1}^{\infty} \left[ \int_{\mathcal{O}} \sqrt{\mu_k} \P_{1/n} \Phi(s,x)\boldq_k (x) \cdot \P_{1/n} \u(s,x) \d x  \right]^2 \d s < \infty.
\end{align}
The computations below show that
\begin{align}
\mathbb{E}\left[	\int_{0}^{t} \sum_{k=1}^{\infty} \left( \int_{\mathcal{O}} \sqrt{\mu_k} \P_{1/n} \Phi(s,x)\boldq_k (x) \cdot \P_{1/n} \u(s,x) \d x  \right)^2 \d s \right] < \infty,
\end{align}
  and this is known to be sufficient to apply the stochastic version of Fubini’s theorem (see, for instance, Lemma 3.3 of \cite{Kallianpur+Striebel_1969} or \cite{Neerven+Veraar} and the references therein for more sophisticated results).   Also note that $\P_{1/n}\u(t,x)$  is continuous (infinitely differentiable)
in $x$ for any $(\omega,t)$. Therefore, it is  $\mathcal{F}_t\otimes \mathcal{B}(\R^+)$-measurable. Since it is also continuous in $t$ for each $(\omega,x)$, the function $\P_{1/n}\u(t,x)$ is  $\mathcal{F}_t\otimes \mathcal{B}(\R^d)$-measurable and there is no measurability obstructions in applying Fubini’s theorems.

For $\G_0=\G_0^1+\G_0^2$, where $\G_0^1\in\mathrm{L}^2(\Omega;\mathrm{L}^{2}(0,T;\V'))$ and $\G_0^2\in\mathrm{L}^{\frac{r+1}{r}}(\Omega;\mathrm{L}^{\frac{r+1}{r}}(0,T;{\wi\L}^{\frac{r+1}{r}}))$,  we consider 
\begin{align}
	& \int_{0}^{t}  \int_{\mathcal{O}} \P_{1/n}\G_0(s,x) \cdot \P_{1/n} \u(s,x) \d s 
	\nonumber\\ & = \int_{0}^{t}   \left\langle \G_0(s) , \P_{1/n}\P_{1/n} \u(s) \right\rangle \d s 
	\nonumber\\ & = \int_{0}^{t}   \left\langle \G_0^1 (s) , \P_{1/n}\P_{1/n} \u(s) \right\rangle \d s  + \int_{0}^{t}   \left\langle \G_0^2 (s) , \P_{1/n}\P_{1/n} \u(s) \right\rangle \d s 
	\nonumber\\ &  \leq \int_{0}^{t} \|\G_0^1(s)\|_{\V^{\prime}} \|\P_{1/n}\P_{1/n} \u(s)\|_{\V}   \d s +  \int_{0}^{t} \|\G_0^2(s)\|_{\wi\L^{\frac{r+1}{r}}} \|\P_{1/n}\P_{1/n} \u(s)\|_{\wi\L^{r+1}} \d s
	\nonumber\\ &  \leq \mathfrak{C} \int_{0}^{t} \|\G_0^1(s)\|_{\V^{\prime}} \|\P_{1/n} \u(s)\|_{\V} \d s  + \mathfrak{C} \int_{0}^{t} \|\G_0^2(s)\|_{\wi\L^{\frac{r+1}{r}}} \|\P_{1/n} \u(s)\|_{\wi\L^{r+1}} \d s
	\nonumber\\ &  \leq \mathfrak{C}^2 \int_{0}^{t} \|\G_0^1(s)\|_{\V^{\prime}} \| \u(s)\|_{\V} \d s  + \mathfrak{C}^2 \int_{0}^{t} \|\G_0^2(s)\|_{\wi\L^{\frac{r+1}{r}}} \|  \u(s)\|_{\wi\L^{r+1}} \d s
	\nonumber\\ &  \leq \mathfrak{C}^2 \left(\int_{0}^{t} \|\G_0^1(s)\|_{\V^{\prime}}^2 \d s\right)^{\frac12}   \left(\int_{0}^{t}  \| \u(s)\|_{\V}^2 \d s \right)^{\frac12}
	\nonumber \\ & \quad + \mathfrak{C}^2 \left(\int_{0}^{t} \|\G_0^2(s)\|_{\wi\L^{\frac{r+1}{r}}}^{\frac{r+1}{r}} \d s\right)^{\frac{r}{r+1}}  \left(\int_{0}^{t}  \|  \u(s)\|_{\wi\L^{r+1}}^{r+1} \d s\right)^{\frac{1}{r+1}},
\end{align}
$\mathbb{P}$-a.s., where we have used properties $(1)$ and $(2)$ from \eqref{approx-properties}, which, due to the fact that $\G_0^1\in\mathrm{L}^2(\Omega;\mathrm{L}^{2}(0,T;\V'))$, $\G_0^2\in\mathrm{L}^{\frac{r+1}{r}}(\Omega;\mathrm{L}^{\frac{r+1}{r}}(0,T;{\wi\L}^{\frac{r+1}{r}}))$ and   $\u\in \mathrm{L}^2(\Omega;\mathrm{L}^2(0,T;\V))\cap\mathrm{L}^{r+1}(\Omega;\mathrm{L}^{r+1}(0,T;\widetilde{\L}^{r+1}))$, implies
\begin{align}
	\E\left[\int_{0}^{t}  \int_{\mathcal{O}} \P_{1/n}\G_0(s,x) \cdot \P_{1/n} \u(s,x) \d s \right] < +\infty.
 \end{align}

Now we consider
\begin{align}
	  & \sum_{k=1}^{\infty} \int_{0}^{t} \int_{\mathcal{O}} \left[\sum_{\nu_j < n^2 } e^{- \nu_j/ n} \sqrt{\mu_k} \left\langle\Phi(s)\boldq_k, \bolde_j\right\rangle_{\mathbb{U}^\prime \times \mathbb{U}} \bolde_j(x)\right]^2 \d x \d s
	  \nonumber\\ & = \sum_{k=1}^{\infty} \int_{0}^{t}   \left[ \sum_{\nu_j < n^2 } e^{- 2\nu_j/ n}    |\left\langle\Phi(s)\sqrt{\mu_k}\boldq_k, \bolde_j\right\rangle_{\mathbb{U}^\prime \times \mathbb{U}} |^2  \right]  \d s
	  \nonumber\\ & \leq  \sum_{k=1}^{\infty} \int_{0}^{t}     \sum_{j=1 }^{+\infty}      |\left\langle\Phi(s)\sqrt{\mu_k}\boldq_k, \bolde_j\right\rangle_{\mathbb{U}^\prime \times \mathbb{U}} |^2     \d s
	  \nonumber\\ & =  \sum_{k=1}^{\infty} \int_{0}^{t}   \|\Phi(s)\sqrt{\mu_k}\boldq_k\|_{\H}^2   \d s
	    =   \int_{0}^{t}   \|\Phi(s)\|_{\mathcal{L}_{Q}}^2   \d s,
\end{align}
which, due to the fact that $\Phi  \in \mathrm{L}^2(\Omega;\mathrm{L}^2(0,T ;\mathcal{L}_{\Q}(\H)))$ implies
\begin{align}
	\E\left[\sum_{k=1}^{\infty} \int_{0}^{t} \int_{\mathcal{O}} \left[\sum_{\nu_j < n^2 } e^{- \nu_j/ n} \sqrt{\mu_k} \left\langle\Phi(s)\boldq_k, \bolde_j\right\rangle_{\mathbb{U}^\prime \times \mathbb{U}} \bolde_j(x)\right]^2\d x  \d s\right]  < +\infty.
\end{align}
By the Cauchy-Schwarz and H\"older's inequalities, we have 
\begin{align}
	& \E \left[ \int_{0}^{t} \sum_{k=1}^{\infty} \left( \int_{\mathcal{O}} \sqrt{\mu_k} \P_{1/n} \Phi(s,x)\boldq_k (x) \cdot \P_{1/n} \u(s,x) \d x  \right)^2\d s \right] 
	\nonumber\\ & = \E \left[ \int_{0}^{t} \sum_{k=1}^{\infty} \left[ ( \Phi(s)\sqrt{\mu_k} \boldq_k, \P_{1/n}\P_{1/n}\u(s))    \right]^2  \d s \right] 
	\nonumber\\  & \leq  \E \left[ \int_{0}^{t} \sum_{k=1}^{\infty}  \| \Phi(s)\sqrt{\mu_k} \boldq_k\|_{\H}^2 \|\P_{1/n}\P_{1/n}\u(s)\|^2_{\H}    \d s \right]  
	\nonumber\\  & \leq  \E \left[ \int_{0}^{t} \|\u(s)\|^2_{\H}   \|\Phi(s)\|_{\mathcal{L}_{Q}}^2    \d s \right] 
	\nonumber\\  & \leq  \E \left[ \sup_{s\in[0,t]} \|\u(s)\|^2_{\H} \int_{0}^{t}    \|\Phi(s)\|_{\mathcal{L}_{Q}}^2    \d s \right] 
	\nonumber\\  & \leq  \left(\E \left[ \sup_{s\in[0,t]} \|\u(s)\|^4_{\H} \right]\right)^{\frac12}  \left(\E \left[  \left( \int_{0}^{t}    \|\Phi(s)\|_{\mathcal{L}_{Q}}^2    \d s\right)^2 \right] \right)^{\frac{1}{2}} 
	< +\infty,
\end{align}
where we have used that $\u \in  \mathrm{L}^4\left(\Omega;\mathrm{L}^{\infty}(0,T;\H)\right)$ and $\Phi  \in \mathrm{L}^4(\Omega;\mathrm{L}^2(0,T ;\mathcal{L}_{\Q}(\H)))$.

Hence, by an application of deterministic and stochastic versions of Fubini’s theorem (see, for instance, Lemma 3.3 of \cite{Kallianpur+Striebel_1969} or \cite{Neerven+Veraar}), we deduce from \eqref{eqn-int-O} that 
\begin{align}\label{eqn-in-n}
	\|\P_{1/n}\u(t)\|^2_{\H} 
& = \|\P_{1/n}\u_0\|^2_{\H} - 2  \int_{0}^{t}\int_{\mathcal{O}} \P_{1/n}\G_0(s,x) \cdot \P_{1/n} \u(s,x) \d s 
	\nonumber\\ & \quad  +  2 \int_0^t  \int_{\mathcal{O}} \P_{1/n} \u(s,x)  \cdot  \P_{1/n} \Phi(s,x)\d
	\W(s) 
	\nonumber\\ & \quad +   \sum_{k=1}^{\infty} \int_{0}^{t}\int_{\mathcal{O}}  \left[\sum_{\nu_j < n^2 } e^{- \nu_j/ n} \sqrt{\mu_k} \left\langle\Phi(s)\boldq_k, \bolde_j\right\rangle_{\mathbb{U}^\prime \times \mathbb{U}} \bolde_j(x)\right]^2  \d s 
	\nonumber\\ & = \|\P_{1/n}\u_0\|^2_{\H} - 2  \int_{0}^{t}   \left(\P_{1/n}\G_0(s) , \P_{1/n} \u(s)\right) \d s 
	\nonumber\\ & \quad  +  2 \int_0^t  \left( \P_{1/n} \Phi(s)\d
	\W(s) ,  \P_{1/n} \u(s)  \right)
	\nonumber\\ & \quad +   \sum_{k=1}^{\infty} \int_{0}^{t}   \sum_{\nu_j < n^2 } e^{- 2\nu_j/ n}  |\left\langle\Phi(s)\sqrt{\mu_k}\boldq_k, \bolde_j\right\rangle_{\mathbb{U}^\prime \times \mathbb{U}}|^2   \d s
	\nonumber\\ & = \|\P_{1/n}\u_0\|^2_{\H} - 2  \int_{0}^{t}   \left\langle \G_0(s) , \P_{1/n}\P_{1/n} \u(s)\right\rangle \d s 
   +  2 \int_0^t  \left(  \Phi(s)\d
	\W(s) ,  \P_{1/n}\P_{1/n} \u(s)  \right)
	\nonumber\\ & \quad +   \int_{0}^{t}   \sum_{k=1}^{\infty}  \sum_{\nu_j < n^2 } e^{- 2\nu_j/ n}  |\left\langle\Phi(s)\sqrt{\mu_k}\boldq_k, \bolde_j\right\rangle_{\mathbb{U}^\prime \times \mathbb{U}}|^2   \d s,  \ \mathbb{P}\text{-a.s.}
\end{align}
From \eqref{3.37}, we infer that  for all $t\in[0,T]$
\begin{align}\label{eqn-lim-1}
		\lim_{n\to\infty}\|\P_{1/n}\u(t)\|_{\H}^2= \|\u(t)\|^2_{\H}, \ \mathbb{P}\text{-a.s.,}  \ \text{  and } \ \lim_{n\to\infty}	\|\P_{1/n}\u_0\|_{\H}^2= \|\u_0\|^2_{\H}, \ \mathbb{P}\text{-a.s.}
 \end{align}
For $\G_0=\G_0^1+\G_0^2$, where $\G_0^1\in\mathrm{L}^2(\Omega;\mathrm{L}^{2}(0,T;\V'))$ and $\G_0^2\in\mathrm{L}^{\frac{r+1}{r}}(\Omega;\mathrm{L}^{\frac{r+1}{r}}(0,T;{\wi\L}^{\frac{r+1}{r}}))$,  we consider 
\begin{align}\label{eqn-lim-2}
	& \lim_{n\to \infty} \left| \int_{0}^{t}   \left\langle \G_0(s) , \P_{1/n}\P_{1/n} \u(s)\right\rangle \d s  - \int_{0}^{t}   \left\langle \G_0(s) ,   \u(s)\right\rangle \d s  \right|
	\nonumber\\ & = \lim_{n\to \infty} \left| \int_{0}^{t}   \left\langle \G_0^1(s) + \G_0^2(s) , \P_{1/n}(\P_{1/n} \u(s) -  \u(s)) + \P_{1/n} \u(s) -   \u(s)\right\rangle \d s  \right|
	\nonumber\\ & \leq 
	 \lim_{n\to \infty} \left| \int_{0}^{t}   \left\langle \G_0^1(s)  , \P_{1/n}(\P_{1/n} \u(s) -  \u(s)) + \P_{1/n} \u(s) -   \u(s)\right\rangle \d s  \right| 
	 \nonumber\\ & \quad + \lim_{n\to \infty} \left| \int_{0}^{t}   \left\langle   \G_0^2(s) , \P_{1/n}(\P_{1/n} \u(s) -  \u(s)) + \P_{1/n} \u(s) -   \u(s)\right\rangle \d s  \right|
	 \nonumber\\ & \leq 
	 \lim_{n\to \infty} \|\G_0^1\|_{\mathrm{L}^{2}(0,T;\V')} (\|\P_{1/n}(\P_{1/n} \u -  \u)\|_{\mathrm{L}^{2}(0,T;\V)} + \|\P_{1/n} \u -   \u\|_{\mathrm{L}^{2}(0,T;\V)}   )
	 \nonumber\\ & \quad + \lim_{n\to \infty} \|\G_0^2\|_{\mathrm{L}^{\frac{r+1}{r}}(0,T;{\wi\L}^{\frac{r+1}{r}})} (\|\P_{1/n}(\P_{1/n} \u -  \u)\|_{\mathrm{L}^{r+1}(0,T;{\wi\L}^{r+1})} + \|\P_{1/n} \u -   \u\|_{\mathrm{L}^{r+1}(0,T;{\wi\L}^{r+1})}   )
	 \nonumber\\ & \leq 
	 \lim_{n\to \infty} \|\G_0^1\|_{\mathrm{L}^{2}(0,T;\V')} (\mathfrak{C}\|\P_{1/n} \u -  \u\|_{\mathrm{L}^{2}(0,T;\V)} + \|\P_{1/n} \u -   \u\|_{\mathrm{L}^{2}(0,T;\V)}   )
	 \nonumber\\ & \quad + \lim_{n\to \infty} \|\G_0^2\|_{\mathrm{L}^{\frac{r+1}{r}}(0,T;{\wi\L}^{\frac{r+1}{r}})} (\mathfrak{C}\|\P_{1/n} \u -  \u\|_{\mathrm{L}^{r+1}(0,T;{\wi\L}^{r+1})} + \|\P_{1/n} \u -   \u\|_{\mathrm{L}^{r+1}(0,T;{\wi\L}^{r+1})}   )
	 \nonumber \\ & =0, \;\;\; \mathbb{P}\text{-a.s.},
\end{align}
where we have used \eqref{335} and \eqref{335a}. Let us now consider
{\small{\begin{align}
	&    \int_0^t\|\Phi(s)\|_{\mathcal{L}_{\Q}}^2\d s      -   \int_{0}^{t}  \sum_{k=1}^{\infty}  \sum_{\nu_j < n^2 } e^{- 2\nu_j/ n}  |\left\langle\Phi(s)\sqrt{\mu_k}\boldq_k, \bolde_j\right\rangle_{\mathbb{U}^\prime \times \mathbb{U}}|^2   \d s 
	\nonumber\\ & =       \int_{0}^{t}  \sum_{k=1}^{\infty}  \sum_{j =1 }^{\infty}   |\left\langle\Phi(s)\sqrt{\mu_k}\boldq_k, \bolde_j\right\rangle_{\mathbb{U}^\prime \times \mathbb{U}}|^2   \d s     -   \int_{0}^{t}  \sum_{k=1}^{\infty}  \sum_{\nu_j < n^2 } e^{- 2\nu_j/ n}  |\left\langle\Phi(s)\sqrt{\mu_k}\boldq_k, \bolde_j\right\rangle_{\mathbb{U}^\prime \times \mathbb{U}}|^2   \d s
	\nonumber\\ & =       \int_{0}^{t}  \sum_{k=1}^{\infty}  \sum_{\nu_j <  n^2 } (1 - e^{- 2\nu_j/ n} )  |\left\langle\Phi(s)\sqrt{\mu_k}\boldq_k, \bolde_j\right\rangle_{\mathbb{U}^\prime \times \mathbb{U}}|^2   \d s + \int_{0}^{t}  \sum_{k=1}^{\infty}  \sum_{\nu_j \geq   n^2 }   |\left\langle\Phi(s)\sqrt{\mu_k}\boldq_k, \bolde_j\right\rangle_{\mathbb{U}^\prime \times \mathbb{U}}|^2   \d s
	\nonumber\\ & \leq        \int_{0}^{t}  \sum_{k=1}^{\infty}  \sum_{j=1}^{\infty} (1 - e^{- 2\nu_j/ n} )  |\left\langle\Phi(s)\sqrt{\mu_k}\boldq_k, \bolde_j\right\rangle_{\mathbb{U}^\prime \times \mathbb{U}}|^2   \d s + \int_{0}^{t}  \sum_{k=1}^{\infty}  \sum_{\nu_j \geq   n^2 }   |\left\langle\Phi(s)\sqrt{\mu_k}\boldq_k, \bolde_j\right\rangle_{\mathbb{U}^\prime \times \mathbb{U}}|^2   \d s, \   \mbox{$\mathbb{P}$-a.s.}
\end{align}}}
 Note that 
\begin{align}
	& \int_{0}^{t}  \sum_{k=1}^{\infty}  \sum_{j=1}^{\infty} (1 - e^{- 2\nu_j/ n} )  |\left\langle\Phi(s)\sqrt{\mu_k}\boldq_k, \bolde_j\right\rangle_{\mathbb{U}^\prime \times \mathbb{U}}|^2   \d s + \int_{0}^{t}  \sum_{k=1}^{\infty}  \sum_{\nu_j \geq   n^2 }   |\left\langle\Phi(s)\sqrt{\mu_k}\boldq_k, \bolde_j\right\rangle_{\mathbb{U}^\prime \times \mathbb{U}}|^2   \d s \nonumber \\ & \leq  3 \int_{0}^{t}  \sum_{k=1}^{\infty}  \sum_{j=1}^{\infty}  |\left\langle\Phi(s)\sqrt{\mu_k}\boldq_k, \bolde_j\right\rangle_{\mathbb{U}^\prime \times \mathbb{U}}|^2   \d s = 3\int_0^t\|\Phi(s)\|_{\mathcal{L}_{\Q}}^2\d s < + \infty , \;\;\; \mathbb{P}\text{-a.s.}
\end{align}
Therefore, by an application of the Dominated Convergence Theorem, we arrive at
\begin{align}\label{eqn-lim-3}
    \lim_{n\to\infty}   \int_{0}^{t}  \sum_{k=1}^{\infty}  \sum_{\nu_j < n^2 } e^{- 2\nu_j/ n}  |\left\langle\Phi(s)\sqrt{\mu_k}\boldq_k, \bolde_j\right\rangle_{\mathbb{U}^\prime \times \mathbb{U}}|^2   \d s = 	\int_0^t\|\Phi(s)\|_{\mathcal{L}_{\Q}}^2\d s , \;\;\;  \mathbb{P}\text{-a.s.}
\end{align}

Finally, we consider 
\begin{align}
	& \E\left[ \left(\int_0^t  \left(  \Phi(s)\d
	\W(s) ,  \P_{1/n}\P_{1/n} \u(s) - \u(s)  \right)\right)^2 \right]
	\nonumber\\ & = \E \left[ \int_{0}^{t} \sum_{k=1}^{\infty} \left[ ( \Phi(s)\sqrt{\mu_k} \boldq_k, \P_{1/n}\P_{1/n}\u(s) - \u(s))    \right]^2  \d s \right] 
	\nonumber\\  & \leq  \E \left[ \int_{0}^{t} \sum_{k=1}^{\infty}  \| \Phi(s)\sqrt{\mu_k} \boldq_k\|_{\H}^2 \|\P_{1/n}\P_{1/n}\u(s) - \u(s) \|^2_{\H}    \d s \right]  
	\nonumber\\  & =  \E \left[ \int_{0}^{t} \sum_{k=1}^{\infty}  \| \Phi(s)\sqrt{\mu_k} \boldq_k\|_{\H}^2 \|\P_{1/n}\P_{1/n}\u(s) - \P_{1/n}\u(s) + \P_{1/n}\u(s) - \u(s) \|^2_{\H}    \d s \right]  
	\nonumber\\  & \leq 2 \E \left[ \int_{0}^{t} \sum_{k=1}^{\infty}  \| \Phi(s)\sqrt{\mu_k} \boldq_k\|_{\H}^2 \| \P_{1/n}\u(s) - \u(s) \|^2_{\H}    \d s \right]  
	\nonumber\\  & \leq 2 \E \left[ \int_{0}^{t} \|\P_{1/n}\u(s) -  \u(s)\|^2_{\H}   \|\Phi(s)\|_{\mathcal{L}_{Q}}^2    \d s \right] 
	\nonumber\\  & \leq 2 \E \left[ \sup_{s\in[0,t]} \|\P_{1/n}\u(s) -  \u(s)\|^2_{\H} \int_{0}^{t}    \|\Phi(s)\|_{\mathcal{L}_{Q}}^2    \d s \right] 
	\nonumber\\  & \leq  2 \left(\E \left[ \sup_{s\in[0,t]} \|\P_{1/n}\u(s) -  \u(s)\|^4_{\H} \right]\right)^{\frac12}  \underbrace{\left(\E \left[  \left( \int_{0}^{t}    \|\Phi(s)\|_{\mathcal{L}_{Q}}^2    \d s\right)^2 \right] \right)^{\frac{1}{2}}}_{< +\infty}.
\end{align}
We aim next to show that $\E \left[ \sup\limits_{s\in[0,t]} \|\P_{1/n}\u(s) -  \u(s)\|^4_{\H} \right]\to 0$ and $n\to\infty$. In order to achieve this goal, we make use of the Lebesgue-Vitali theorem (Theorem \ref{L-Vthm}). 

From property $(1)$ of \eqref{approx-properties}, we have that $$\sup\limits_{s\in[0,t]} \|\P_{1/n}\u(s) -  \u(s)\|^4_{\H}\to 0$$ as $n\to\infty$. From Proposition \ref{prop-EE2}, we have that for $\u_0\in\mathrm{L}^{4+\eta}(\Omega;\H)$ (for some $\eta>0$)
\begin{align}
	\E \left[ \sup\limits_{s\in[0,t]} \|\P_{1/n}\u(s) -  \u(s)\|^{4 + \eta}_{\H} \right] &  \leq C \E \left[ \sup\limits_{s\in[0,t]} [\|\P_{1/n}\u(s) \|^{4 + \eta}_{\H} + \|\u(s)\|^{4(1+\eta)}_{\H}] \right]
	\nonumber\\ & \leq C \E \left[ \sup\limits_{s\in[0,t]}   \|\u(s)\|^{4+\eta}_{\H} \right] < + \infty.
\end{align}
From Remark \ref{rem-uniform-integable}, it is immediate that the uniform integrability condition holds. Therefore, we apply Theorem \ref{L-Vthm} and obtain $\E \left[ \sup\limits_{s\in[0,t]} \|\P_{1/n}\u(s) -  \u(s)\|^4_{\H} \right]\to 0$ as  $n\to \infty$. Further, we deduce 
\begin{align}
	\E\left[ \left(\int_0^t  \left(  \Phi(s)\d
	\W(s) ,  \P_{1/n}\P_{1/n} \u(s) - \u(s)  \right)\right)^2 \right] \to 0 \;\;\; \text{ as } \;\;\; n\to \infty,
\end{align}
which implies
\begin{align}\label{eqn-lim-4}
\lim_{n\to \infty}	\int_0^t  \left( \P_{1/n} \Phi(s)\d
	\W(s) ,  \P_{1/n} \u(s)  \right) = \int_0^t  \left(  \Phi(s)\d
	\W(s) ,  \u(s)  \right), \;\;\; \mathbb{P}\text{-a.s.}
\end{align}

Using \eqref{eqn-lim-1}-\eqref{eqn-lim-2}, \eqref{eqn-lim-3}  and \eqref{eqn-lim-4} in \eqref{eqn-in-n}, we immediately obtain 
\begin{align}\label{345}
	 \|\u(t)\|^2_{\H} 
 & = \|\u_0\|^2_{\H} - 2  \int_{0}^{t}   \left\langle \G_0(s) ,   \u(s)\right\rangle \d s 
	+  2 \int_0^t  \left(  \Phi(s)\d
	\W(s) ,    \u(s)  \right)
	  +   \int_0^t\|\Phi(s)\|_{\mathcal{L}_{\Q}}^2\d s,   \mathbb{P}\text{-a.s.}
\end{align}
 Taking expectation and noting the fact that the final term in the right hand side of the equality \eqref{345} is a martingale, we find 
\begin{align}
\E\left[\|\u(t)\|_{\H}^2\right]&=\E\left[\|{\u_0}\|_{\H}^2\right]-2\E\left[\int_0^t\langle{\mathrm{G}_0}(s),\u(s)\rangle\d s\right]+\E\left[\int_0^t\|\Phi(s)\|_{\mathcal{L}_{\Q}}^2\d s\right].
\end{align}
Thus an application of It\^o's formula to the process $e^{-2\varrho t}\|\u(\cdot)\|_{\H}^2$ yields 
	\begin{align}\label{4.45}
	\E\left[e^{-2\varrho t}\|\u(t)\|_{\H}^2\right]&=\E\left[\|\u_0\|_{\H}^2\right]-\E\left[\int_0^te^{-2\varrho s}\langle 2\G_0(s)+2\varrho \u(s),\u(s)\rangle\d
	s\right]\nonumber\\&\quad+\E\left[\int_0^t
	e^{-2\varrho s}\|\Phi(s)\|_{\mathcal{L}_{\Q}}^2\d s\right],
	\end{align}
	for all $t\in[0,T]$. Note that the initial value
	$\u^n(0)$ converges to $\u_0$ strongly in $\mathrm{L}^2(\Omega;\H)$, that is,
	\begin{align}\label{4.46}
	\lim_{n\to\infty}\E\left[\|\u^n(0)-\u_0\|_{\H}^2\right]=0.
	\end{align}

	\noindent\textbf{Step (5):} \emph{Minty-Browder
		technique and global strong solution.} It is now left to show that $$\G(\u(\cdot))=\G_0(\cdot)\ \text{ and }\ 
	\Phi(\cdot,\u(\cdot))=\Phi(\cdot).$$ In order to do this, we make use of Lemma \ref{lem3.6}.  For
	$\v\in\mathrm{L}^2(\Omega;\mathrm{L}^{\infty}(0,T;\H_m))$ with $m<n$, using the local monotonicity result (see (\ref{3.11y})),	we get 
	\begin{align}\label{4.48}
	&\E\bigg[\int_0^{T}e^{-2\varrho t}(2\langle\G(\v(t))-\G(\u^n(t)),\v(t)-\u^n(t)\rangle
	+2\varrho \left(\v(t)-\u^n(t),\v(t)-\u^n(t)\right))\d
	t\bigg]\nonumber\\&\geq \E\left[\int_0^{T}e^{-2\varrho t}\|\Phi^n(t,
	\v(t)) - \Phi^n(t,\u^n(t))\|^2_{\mathcal{L}_{\Q}}\d
	t\right].
	\end{align}
Rearranging the terms in \eqref{4.48} and then using the energy	equality (\ref{4.39}), we obtain 
	\begin{align}\label{4.49}
	&\E\left[\int_0^{T}e^{-2\varrho t}\langle 2\G(\v(t))+2\varrho \v(t),\v(t)-\u^n(t)\rangle\d
	t\right]\nonumber\\&\quad-\E\left[\int_0^{T}e^{-2\varrho t}\|\Phi^n(t,
	\v(t))\|^2_{\mathcal{L}_{\Q}}
	\d
	t\right]+2\E\left[\int_0^{T}e^{-2\varrho t}\left(\Phi^n(t,
	\v(t)),
	\Phi^n(t,\u^n(t))\right)_{\mathcal{L}_{\Q}}\d
	t\right]\nonumber\\&\geq
	\E\left[\int_0^{T}e^{-2\varrho t}\langle 2\G(\u^n(t))+2\varrho \u^n(t),\v(t)\rangle\d
	t\right]\nonumber\\&\quad-\E\left[\int_0^{T}e^{-2\varrho t}\langle 2\G(\u^n(t))+2\varrho \u^n(t),\u^n(t)\rangle\d
	t\right]+\E\left[\int_0^{T}e^{-2\varrho t}\|
	\Phi^n(t,\u^n(t))\|^2_{\mathcal{L}_{\Q}}\d
	t\right]\nonumber\\&=\E\left[\int_0^{T}e^{-2\varrho t}\langle 2\G(\u^n(t))+2\varrho \u^n(t),\v(t)\rangle\d
	t\right]
	+\E\left[e^{-2\varrho T}\|\u^n(T)\|_{\H}^2-\|\u^n(0)\|_{\H}^2\right].
	\end{align}
	Let us now discuss the convergence of the terms involving noise co-efficient. Note that \begin{align}\label{eqn49z}
	&\E\Bigg[\int_0^{T}e^{-2\varrho t}(2\left(\Phi^n(t, \v(t)),
	\Phi^n(t,\u^n(t))\right)_{\mathcal{L}_{\Q}}-\|\Phi^n(t,
	\v(t))\|^2_{\mathcal{L}_{\Q}})\d
	t\Bigg]\nonumber\\& = \E\left[\int_0^{T}e^{-2\varrho t}2\left(\Phi(t,
	\v(t)),
	\Phi^n(t,\u^n(t))\right)_{\mathcal{L}_{\Q}}\d
	t\right]\nonumber\\&\quad+
	\E\left[\int_0^{T}e^{-2\varrho t}2\left(\Phi^n(t, \v(t))-\Phi(t,
	\v(t)),
	\Phi^n(t,\u^n(t))\right)_{\mathcal{L}_{\Q}}\d
	t\right]\nonumber\\&\quad
	-\E\left[\int_0^{T}e^{-2\varrho t}\|\Phi^n(t,
	\v(t))\|^2_{\mathcal{L}_{\Q}}\d t\right]\nonumber\\&\leq
	\E\left[\int_0^{T}e^{-2\varrho t}2\left(\Phi(t, \v(t)),
	\Phi^n(t,\u^n(t))\right)_{\mathcal{L}_{\Q}}\d
	t\right]\nonumber\\&\quad+2C
	\left(\E\left[\int_0^{T}e^{-4\varrho t}\left\|\Phi^n(t,
	\v(t))-\Phi(t, \v(t))\right\|^2_{\mathcal{L}_{\Q}}\d
	t\right]\right)^{1/2} 
	-\E\left[\int_0^{T}e^{-2\varrho t}\|\Phi^n(t,
	\v(t))\|^2_{\mathcal{L}_{\Q}}\d t\right],
	\end{align}
	where $C=
	{\left(\E\left[\int_0^{T}e^{-4\varrho t}\left\|\Phi^n(t,
		\u^n(t))\right\|^2_{\mathcal{L}_{\Q}}\d
		t\right]\right)^{1/2}}$. Then, applying the weak convergence of $\{\Phi^n(\cdot,\u^n(\cdot)):n\in\mathbb{N}\}$  given in 
	(\ref{4.43z}) 
	to the first term and using Lebesgue's dominated convergence theorem to
	the second and final terms on the right hand side of the inequality
	(\ref{eqn49z}), we deduce that 
	\begin{align}\label{4.50}
	&\E\Bigg[\int_0^{T}e^{-2\varrho t}(2\left(\Phi^n(t, \v(t)),
	\Phi^n(t,\u^n(t))\right)_{\mathcal{L}_{\Q}}-\|\Phi^n(t,
	\v(t))\|^2_{\mathcal{L}_{\Q}})\d
	t\Bigg]\nonumber\\& \to
	\E\left[\int_0^{T}e^{-2\varrho t}(2(\Phi(t, \v(t)),
	\Phi(t))_{\mathcal{L}_{\Q}}-\|\Phi(t,
	\v(t))\|^2_{\mathcal{L}_{\Q}})\d t\right],
	\end{align}
	as $n\to\infty$. Taking liminf on both sides of (\ref{4.49}),
	and using (\ref{4.50}), we obtain
	\begin{align}\label{4.51}
	&\E\left[\int_0^{T}e^{-2\varrho t}\langle 2\G(\v(t))+2\varrho \v(t),\v(t)-\u(t)\rangle\d
	t\right]\nonumber\\&\quad-\E\left[\int_0^{T}e^{-2\varrho t}\|\Phi(t,
	\v(t))\|^2_{\mathcal{L}_{\Q}} \d
	t\right]+2\E\left[\int_0^{T}e^{-2\varrho t}\left(\Phi(t,
	\v(t)), \Phi(t)\right)_{\mathcal{L}_{\Q}}\d
	t\right]\nonumber\\&\geq\E\left[\int_0^{T}e^{-2\varrho t}\langle 2\G_0(t)+2\varrho \u(t),\v(t)\rangle\d
	t\right]
	+\liminf_{n\to\infty}\E\left[e^{-2\varrho T}\|\u^n(T)\|_{\H}^2-\|\u^n(0)\|_{\H}^2\right].
	\end{align}
	Making use of the lower semicontinuity property of the $\H$-norm and the strong convergence given in
	(\ref{4.46}), the second term on the right hand side of the
	inequality (\ref{4.51}) satisfies:
	\begin{align}\label{4.52}
	&\liminf_{n\to\infty}\E\left[e^{-2\varrho T}\|\u^n(T)\|_{\H}^2-\|\u^n(0)\|_{\H}^2\right]\geq
	\E\left[e^{-2\varrho T}\|\u(T)\|^2_{\H}-\|\u_0\|^2_{\H}\right].
	\end{align}
	We use   the energy equality (\ref{4.45}) and (\ref{4.52}) in
	(\ref{4.51}) to have  
	\begin{align}\label{4.53}
	&\E\left[\int_0^{T}e^{-2\varrho t}\langle 2\G(\v(t))+2\varrho \v(t),\v(t)-\u(t)\rangle\d
	t\right]\nonumber\\&\geq\E\left[\int_0^{T}e^{-2\varrho t}\|\Phi(t,
	\v(t))\|^2_{\mathcal{L}_{\Q}} \d
	t\right]-2\E\left[\int_0^{T}e^{-2\varrho t}\left(\Phi(t,
	\v(t)), \Phi(t)\right)_{\mathcal{L}_{\Q}}\d
	t\right]\nonumber\\&\quad+\E\left[\int_0^{T}e^{-2\varrho t}\|\Phi(t)\|^2_{\mathcal{L}_{\Q}}\d t\right]
	+\E\left[\int_0^{T}e^{-2\varrho t}\langle2\G_0(t)+\varrho \u(t),\v(t)-\u(t)\rangle\d
	t\right].
	\end{align}
	Rearranging the terms in (\ref{4.53}), we obtain
	\begin{align}\label{4.54}
	&\E\left[\int_0^{T}e^{-2\varrho t}\langle 2\G(\v(t))-2\G_0(t)+2\varrho (\v(t)-\u(t)),\v(t)-\u(t)\rangle\d
	t\right]\nonumber\\&\geq
	\E\Bigg[\int_0^{T}e^{-2\varrho t}\|\Phi(t,
	\v(t))-\Phi(t)\|^2_{\mathcal{L}_{\Q}}
	\d t\Bigg]\geq 0.
	\end{align}
	Note that the estimate (\ref{4.54}) holds true for any	$\v\in\mathrm{L}^2(\Omega;\mathrm{L}^{\infty}(0,T;\H_m))$ and for any
	$m\in\mathbb{N}$, since the estimate is independent of
	$m$ and $n$. Using a density argument, the
	inequality (\ref{4.54}) remains true for any
\begin{align*}\v&\in\mathrm{L}^2(\Omega;\mathrm{L}^{\infty}(0,T;\H)\cap\mathrm{L}^2(0,T;\V))\cap\mathrm{L}^{r+1}(\Omega;\mathrm{L}^{r+1}(0,T;\widetilde{\L}^{r+1}))=:\mathcal{G}.\end{align*} 
	Indeed, for any $\v\in\mathcal{G}$,
	there exists a strongly convergent subsequence
	$\v_m\in\mathcal{G}$, which satisfies the
	inequality (\ref{4.54}).
	Taking $\v(\cdot)=\u(\cdot)$ in (\ref{4.54}) immediately gives
	$\Phi(\cdot,\v(\cdot))=\Phi(\cdot)$. Next, we take
	$\v(\cdot)=\u(\cdot)+\lambda\w(\cdot)$, $\lambda>0$, where
	$\w\in\mathcal{G},$ and substitute for
	$\v$ in (\ref{4.54}) to find
	\begin{align}\label{4.55}
	\E\left[\int_0^{T}e^{-2\varrho t}\langle \G(\u(t)+\lambda\w(t))-\G_0(t)+\varrho \lambda\w(t),\lambda\w(t)\rangle\d
	t\right]\geq 0.
	\end{align}
	Dividing the above inequality by $\lambda$,  using the hemicontinuity property of
	$\G(\cdot)$ (see Lemma \ref{lem2.8}), and passing $\lambda\to 0$, we obtain
	\begin{align}\label{4.56}
	\E\left[\int_0^{T}e^{-2\varrho t}\langle\G(\u(t))-\G_0(t),\w(t)\rangle\d
	t\right]\geq 0,
	\end{align}
	since the final term in (\ref{4.55}) tends to $0$ as $\lambda\to0$.
	Thus from (\ref{4.56}), we get  $\G(\u(t))=\G_0(t)$ and hence $\u(\cdot)$ is a strong solution of the
	system (\ref{32}) and
	$\u\in\mathcal{G}$. From \eqref{345}, it is immediate that $\u(\cdot)$ satisfy the following energy equality (It\^o's formula): 
	\begin{align}\label{3.63}
&	\|\u(t)\|_{\H}^2+2\mu \int_0^t\|\nabla\u(s)\|_{\H}^2\d s +2\alpha \int_0^t\|\u(s)\|_{\H}^2\d s +2\beta\int_0^t\|\u(s)\|_{\widetilde{\L}^{r+1}}^{r+1}\d s\nonumber\\&=\|{\u_0}\|_{\H}^2+\int_0^t\|\Phi(s,\u(s))\|_{\mathcal{L}_{\Q}}^2\d s+2\int_0^t(\Phi(s,\u(s))\d\W(s),\u(s)),
	\end{align}
for all $t\in(0,T)$, $\mathbb{P}$-a.s. 	
	Furthermore, since $\u(\cdot)$ satisfies the energy estimate \eqref{eng1} and the energy equality \eqref{3.63}, one can show  that the $\mathscr{F}_t$-adapted paths of $\u(\cdot)$ are continuous with trajectories in  $\C([0,T];\H)\cap\mathrm{L}^2(0,T;\V))\cap\mathrm{L}^{r+1}(0,T;\widetilde{\L}^{r+1})$, $\mathbb{P}$-a.s. (see \cite{GK1,Me,GK2}, etc).

	\vskip 0.2cm 
	\noindent\textbf{Step (6):} \emph{Uniqueness.} Finally, we show that the strong solution established in step (5) is unique. Let $\u_1(\cdot)$ and $\u_2(\cdot)$ be two strong solutions of the system (\ref{32}). For $N>0$, let us define 
	\begin{align*}
	\tau_N^1&=\inf_{0\leq t\leq T}\Big\{t:\|\u_1(t)\|_{\H}\geq N\Big\},\ \tau_N^2=\inf_{0\leq t\leq T}\Big\{t:\|\u_2(t)\|_{\H}\geq N\Big\}\text{ and }\tau_N:=\tau_N^1\wedge\tau_N^2.
	\end{align*}
	Using the energy estimate \eqref{eng1}, it can  be shown in a similar way as in step (1), Proposition \ref{prop-EE2} that $\tau_N\to T$ as $N\to\infty$, $\mathbb{P}$-a.s. Let us define $\w(\cdot):=\u_1(\cdot)-\u_2(\cdot)$ and $\widetilde{\Phi}(\cdot):=\Phi(\cdot,\u_1(\cdot))-\Phi(\cdot,\u_2(\cdot))$. Then, $\w(\cdot)$ satisfies the following system:
	\begin{equation}
	\left\{
	\begin{aligned}
	\d\w(t)&=-\left[\mu \A\w(t) +\B(\u_1(t))-\B(\u_2(t)) + \alpha\w(t) +\beta(\mathcal{C}(\u_1(t))-\mathcal{C}(\u_2(t)))\right]\d t
	\\  & \quad +\widetilde{\Phi}(t)\d\W(t),\\
	\w(0)&=\w_0.
	\end{aligned}
	\right.
	\end{equation}
Then, $\w(\cdot)$ satisfies the following energy equality: 
	\begin{align}\label{4.59}
	&\|\w(\s)\|_{\H}^2 +2\mu \int_0^{\t}\|\nabla\w(s)\|_{\H}^2\d s +2\alpha \int_0^{\t}\|\w(s)\|_{\H}^2\d s\nonumber\\&=\|\w(0)\|_{\H}^2 -2\int_0^{\t}\langle\B(\u_1(s))-\B(\u_2(s)),\w(s)\rangle\d s\nonumber\\&\quad-2\int_0^{\t}\langle\mathcal{C}(\u_1(s))-\mathcal{C}(\u_2(s)),\u_1(s)-\u_2(s)\rangle\d s+\int_0^{\t}\|\wi\Phi(s)\|_{\mathcal{L}_{\Q}}^2\d s\nonumber\\&\quad+2\int_0^{\t}(\widetilde{\Phi}(s)\d\W(s),\w(s)).
	\end{align}
	From (\ref{2.30}), we obtain 
	\begin{align*}
	|\langle \B(\u_1)-\B(\u_2),\w\rangle|\leq\frac{\mu }{2}\|\nabla\w\|_{\H}^2+\frac{\beta}{2}\||\u_2|^{\frac{r-1}{2}}\w\|_{\H}^2+\varrho\|\w\|_{\H}^2.
	\end{align*}
	and from \eqref{2.27}, we get 
	\begin{align*}
\beta	\langle\mathcal{C}(\u_1)-\mathcal{C}(\u_2),\w\rangle\geq \frac{\beta}{2}\||\u_2|^{\frac{r-1}{2}}\w\|_{\H}^2.
	\end{align*}
	Thus, using the above two estimates in   (\ref{4.59}), we infer that 
	\begin{align}\label{4.60}
	&\|\w(\s)\|_{\H}^2+\mu \int_0^{\t}\|\nabla\w(s)\|_{\H}^2\d s +2\alpha \int_0^{\t}\|\w(s)\|_{\H}^2\d s \nonumber\\&\leq\|\w(0)\|_{\H}^2 +2\varrho\int_0^{\t}\|\w(s)\|_{\H}^2\d s +\int_0^{\t}\|\widetilde{\Phi}(s)\|^2_{\mathcal{L}_{\Q}}\d
	s +2\int_0^{\t}(\widetilde{\Phi}(s)\d\W(s),\w(s)).
	\end{align}
	It should be noted that the final term in the right hand side of the inequality (\ref{4.60}) is a martingale. Taking expectation in (\ref{4.60}), and  then using  Hypothesis \ref{hyp} (H.2), we obtain  
	\begin{align}\label{4.62}
	\E\left[\|\w(\s)\|_{\H}^2\right]
	&\leq \E\left[\|\w(0)\|_{\H}^2\right]+(L+2\varrho)\E\left[\int_0^{\s}\|\w(s)\|_{\H}^2\d s\right]\\
	&\leq  \E\left[\|\w(0)\|_{\H}^2\right]+(L+2\varrho)\int_0^{t}\E\left[\|\w(s\land \tau_N)\|_{\H}^2\right]\d s.
	\end{align}
	Applying  Gr\"onwall's inequality in (\ref{4.62}),  we arrive at 
	\begin{align}\label{4.63}
	&\E\left[\|\w(\s)\|_{\H}^2\right]\leq \E\left[\|\w_0\|_{\H}^2\right]e^{(L+2\varrho)T}.
	\end{align}
	Thus the initial data  $\u_1(0)=\u_2(0)=\u_0$ leads to $\w(\s)=0$, $\mathbb{P}$-a.s. But using the fact that $\tau_N\to T$, $\mathbb{P}$-a.s., implies $\w(t)=0$ and hence $\u_1(t) = \u_2(t)$, $\mathbb{P}$-a.s., for all $t \in[0, T ]$,  and hence the uniqueness follows.
\end{proof}

\begin{remark}
	One might wonder whether the operator $\mathrm{P}_{1/n}$ defined in \eqref{3.32} could also be employed for the existence results. However, although $\mathrm{P}_{1/n}$ is self-adjoint, it is not a projection. Therefore, it cannot replace $\mathrm{P}_{n}$ from \eqref{eqn-projection}, which is crucially used in Step $\mathbf{(1)}$ of the proof of Theorem \ref{exis}.
	\end{remark}

\begin{remark}
	Recently authors in \cite{GK2} (Theorem 1) obtained It\^o's formula for processes taking values in intersection of finitely many Banach spaces. But it seems to us that this result may not be applicable in our context for establishing the energy equality \eqref{3.63}, as our operators $\B(\cdot),\mathcal{C}(\cdot):\V\cap\widetilde{\L}^{r+1}\to\V'+\widetilde{\L}^{\frac{r+1}{r}}$ (see \eqref{212}) and one can show the local integrability in the sum of Banach spaces only. 
\end{remark}
\begin{theorem}\label{exis1}
For $r=3$ and $2\beta\mu \geq 1$,	let $\u_0\in \mathrm{L}^{4+\eta}(\Omega;\H)$ (for some $\eta>0$) be given.  Then under Hypothesis \ref{hyp}, there exists a \emph{pathwise unique strong solution}
	$\u(\cdot)$ to the system (\ref{32}) such that \begin{align*}\u&\in\mathrm{L}^{4+\eta}(\Omega;\mathrm{L}^{\infty}(0,T;\H))\cap\mathrm{L}^{2}(\Omega;\mathrm{L}^2(0,T;\V))\cap\mathrm{L}^{4}(\Omega;\mathrm{L}^{4}(0,T;\widetilde{\L}^{4})),\end{align*} with $\mathbb{P}$-a.s., continuous trajectories in $\H$.
\end{theorem}
\begin{proof}
	A proof of the Theorem \ref{exis1} follows similarly as in the Theorem \ref{exis}, by using the global monotonicity result \eqref{218} and the fact that 
	\begin{align}
	&\int_0^T\langle\mathrm{G}(\u(t))-\mathrm{G}(\v(t)),\u(t)-\v(t)\rangle +\frac{L}{2}\int_0^T\|\u(t)-\v(t)\|_{\H}^2\d t\nonumber\\&\geq \frac{1}{2}\int_0^T\|\Phi(t, \u(t)) - \Phi(t,	\v(t))\|^2_{\mathcal{L}_{\Q}}\d t,
	\end{align}
	for $2\beta\mu \geq 1$. Uniqueness also follows easily by using the estimate \eqref{232}.
\end{proof}

\begin{remark}\label{rem3.7}
	1.	For $d=2$, $r\in[1,3]$ and $\u_0\in\mathrm{L}^4(\Omega;\H)$, one can obtain the existence an uniqueness of pathwise strong solution $$\u\in\mathrm{L}^4\left(\Omega;\mathrm{L}^{\infty}\left(0,T;\H\right)\right)\cap\mathrm{L}^2\left(\Omega;\mathrm{L}^2\left(0,T;\V\right)\right)$$ with $\mathbb{P}$-a.s.  paths in $\C([0,T];\H)\cap\mathrm{L}^2(0,T;\V)$ to the system \eqref{32} can be obtained by using local monotonicity result given in \eqref{fe2} (see \cite{MTM,MTM6}, etc for similar techniques). 
	
	2. If the  domain is an $d$-dimensional torus, then one can approximate functions in $\L^p$-spaces using the truncated Fourier expansions in the following way (see \cite[Theorem 1.6, Chapter 1, pp. 27]{JCR4} and \cite[Theorem 5.2]{KWH}). Let $\mathcal{O}=[0,2\pi]^d$ and $\mathcal{Q}_k:=[-k,k]^d\cap\mathbb{Z}^d$. For every $\w\in\L^1(\mathcal{O})$ and every $k\in\mathbb{N}$, we define $$R_k(\w):=\sum_{m\in\mathcal{Q}_k}\widehat{\w}_je^{\iota m\cdot x},$$ where the Fourier coefficients $\widehat{\w}_j$ are given by $\widehat{\w}_j:=\frac{1}{|\mathcal{O}|}\int_{\mathcal{O}}\w_j(x)e^{-\iota m\cdot x}\d x$. Then, for every $1<p<\infty$, there exists a constant $C_p$, independent of $k$, such that $$\|R_k\w\|_{\L^p(\mathcal{O})}\leq C_p\|\w\|_{\L^p(\mathcal{O})}, \ \text{ for all } \ \w\in\L^p(\mathcal{O}),$$ and $$\|R_k\w-\w\|_{\L^p(\mathcal{O})}\to 0\ \text{ as }\ k\to\infty.$$
\end{remark}
\section{Strong solution}

In this section, we establish the existence of a global strong solution to system \eqref{1} on either the torus or the whole space.
Owing to certain technical challenges, the regularity of strong solutions is derived separately for the torus and the whole space, within different function spaces and under varying regularity assumptions on the initial data and forcing terms.

\subsection{Functional setting for periodic domain}
For $\mathrm{L}>0$, let us consider a $d$-dimensional torus $\mathbb{T}^d=\left(\frac{\R}{\mathrm{L}\mathbb{Z}}\right)^d$, $2\leq d\leq4$.
Let $\C_{\mathrm{p}}^{\infty}(\mathbb{T}^d;\R^d)$ denote the space of all infinitely differentiable  functions $\u$ satisfying periodic boundary conditions $\u(x+\mathrm{L}e_{i},\cdot) = \u(x,\cdot)$, for $x\in \R^d$. \emph{We are not assuming the zero mean condition for the velocity field unlike the case of NSE, since the absorption term $\beta|\u|^{r-1}\u$ does not preserve this property (see \cite{PAM}). Therefore, we cannot use the well-known Poincar\'e inequality and we have to deal with the  full $\H^1$-norm.} The Sobolev space  $\H_{\mathrm{p}}^s(\mathbb{T}^d):=\mathrm{H}_{\mathrm{p}}^s(\mathbb{T}^d;\mathbb{R}^d)$ is the completion of $\C_{\mathrm{p}}^{\infty}(\mathbb{T}^d;\R^d)$  with respect to the $\H^s$-norm and the norm on the space $\H_{\mathrm{p}}^s(\mathbb{T}^d)$ is given by $$\|\u\|_{{\H}^s_{\mathrm{p}}}:=\left(\sum_{0\leq|\boldsymbol\alpha|\leq s}\|\D^{\boldsymbol\alpha}\u\|_{\mathbb{L}^2(\mathbb{T}^d)}^2\right)^{1/2}.$$ 	
It is known from \cite{JCR1} that the Sobolev space of periodic functions $\H_{\mathrm{p}}^s(\mathbb{T}^d)$, for $s\geq0$, can be defined as
$$\H_{\mathrm{f}}^s(\mathbb{T}^d)=\left\{\u:\u=\sum_{k\in\mathbb{Z}^d}\u_{k}\mathrm{e}^{2\pi i k\cdot x /  \mathrm{L}},\ \overline{\u}_{k}=\u_{-k}, \  \|\u\|_{{\H}^s_\mathrm{f}}:=\left(\sum_{k\in\mathbb{Z}^d}(1+|k|^{2s})|\u_{k}|^2\right)^{1/2}<\infty\right\}.$$ We infer from \cite[Proposition 5.38]{JCR1} that the norms $\|\cdot\|_{{\H}^s_{\mathrm{p}}}$ and $\|\cdot\|_{{\H}^s_f}$ are equivalent. Let us define 
\begin{align*} 
	\mathcal{V}:=\{\u\in\C_{\mathrm{p}}^{\infty}(\mathbb{T}^d;\R^d):\nabla\cdot\u=0\}.
\end{align*}
We define the spaces $\H$ and $\widetilde{\L}^{p}$ as the closure of $\mathcal{V}$ in the Lebesgue spaces $\mathrm{L}^2(\mathbb{T}^d;\R^d)$ and $\mathrm{L}^p(\mathbb{T}^d;\R^d)$ for $p\in(2,\infty)$, respectively. We also define the space $\V$ as the closure of $\mathcal{V}$ in the Sobolev space $\mathrm{H}^1(\mathbb{T}^d;\R^d)$. Then, we characterize the spaces $\H$, $\widetilde{\L}^p$ and $\V$ with the norms  $$\|\u\|_{\H}^2:=\int_{\mathbb{T}^d}|\u(x)|^2\d x,\quad \|\u\|_{\widetilde{\L}^p}^p:=\int_{\mathbb{T}^d}|\u(x)|^p\d x\ \text{ and }\ \|\u\|_{\V}^2:=\int_{\mathbb{T}^d}(|\u(x)|^2+|\nabla\u(x)|^2)\d x,$$ respectively.  We define the Stokes operator $\A\u:=-\mathcal{P}\Delta\u=-\Delta\u,\;\u\in\D(\A)=\V\cap{\H}^{2}_\mathrm{p}(\mathbb{T}^d). $

Functional setting for whole space has be given in Subsection \ref{Function-spaces}. For this section now onward, we set $\mathcal{O}:=\mathbb{T}^d$ or $\mathbb{R}^d$.

In  periodic domains or on the whole space $\R^d,$ the operators $\mathcal{P}$ and $-\Delta$ commute, so that we can use \eqref{3} and we have the following result  also (see \cite[Lemma 2.1]{KWH}): 
\begin{align}\label{370}
	0&\leq\int_{\mathcal{O}}|\nabla\u(x)|^2|\u(x)|^{r-1}\d x\leq\int_{\mathcal{O}}|\u(x)|^{r-1}\u(x)\cdot\A\u(x)\d x\nonumber\\&\leq r\int_{\mathcal{O}}|\nabla\u(x)|^2|\u(x)|^{r-1}\d x.
\end{align}
Note that the estimate \eqref{370} is true even in bounded domains (with Dirichlet boundary conditions) if one replaces $\A\u$ with $-\Delta\u$ and \eqref{371}-\eqref{3a71} (see below) holds true in bounded domains as well as on the whole space $\R^d$.
Since we are not assuming zero mean condition, using the Sobolev embedding and  \cite[Lemma 2]{JCRWS}, we infer
\begin{align}
	\|\u\|^{p+q}_{\L^{\frac{(p+q)d}{d-p}}(\mathcal{O})}&=\||\u|^{\frac{p+q}{p}}\|_{\L^{\frac{pd}{p-d}}(\mathcal{O})}^p\leq C\||\u|^{\frac{p+q}{p}}\|_{\mathrm{W}^{1,p}(\mathcal{O})}^p\\&=C\left(\|\nabla(|\u|^{\frac{p+q}{p}})\|_{\L^p(\mathcal{O})}^p+\|\u\|_{\L^{p+q}(\mathcal{O})}^{p+q}\right)\nonumber\\&\leq C\left(\||\u|^{\frac{q}{p}}|\nabla\u|\|_{\L^p(\mathcal{O})}^p+\|\u\|_{\L^{p+q}(\mathcal{O})}^{p+q}\right)\nonumber\\&= C\bigg(\int_{\mathcal{O}}|\nabla\u(x)|^{p}|\u(x)|^q\d x+
	\int_{\mathcal{O}}|\u(x)|^{p+q}\d x\bigg),
\end{align}
for all $\u\in\W^{1,m}_0(\mathcal{O})$  with $m=d(p+q)/(d+q),p<d$. One can handle the case $d=2$ in the following way: Let us take $\u\in\C_{\mathrm{p}}^{\infty}(\mathcal{O})$. Then $|\u|^{p/2}\in\H_{\mathrm{p}}^1(\mathcal{O})$, for all $p\in[2,\infty)$, and from the Sobolev embedding, $\H_{\mathrm{p}}^1(\mathcal{O})\hookrightarrow\L^p(\mathcal{O})$, for all $p\in[2,\infty)$, we find 
\begin{align*}
	\|\u\|^{r+1}_{\L^{p(r+1)}(\mathcal{O})}&=\||\u|^{\frac{r+1}{2}}\|_{\L^{2p}(\mathcal{O})}^2\leq C\||\u|^{\frac{r+1}{2}}\|_{\H_{\mathrm{p}}^1(\mathcal{O})}^2\nonumber\\&\leq C\bigg(\int_{\mathcal{O}}|\nabla|\u(x)|^{\frac{r+1}{2}}|^2\d x+
	\int_{\mathcal{O}}||\u(x)|^{\frac{r+1}{2}}|^2\d x\bigg).
\end{align*}
Buy using elementary calculus identity $\partial_k(|\u|^{r+1})=(r+1)u_k(\partial u_k)|\u|^{r-1}$, where $\partial_k$ denotes the $k$-th partial derivative, we infer that $|\nabla|\u|^{\frac{r+1}{2}}|^2\leq C_r|\u|^{r-1}|\nabla\u|^2$ (see the proof of \cite[Lemma 1]{JCRWS}). Thus, we further have 
\begin{align}\label{3a71}
	\|\u\|^{r+1}_{\L^{p(r+1)}(\mathcal{O})}\leq C\bigg(\int_{\mathcal{O}}|\nabla\u(x)|^2|\u(x)|^{r-1}\d x+ \int_{\mathcal{O}}|\u(x)|^{r+1}\d x\bigg),
\end{align}
for any $p\in[2,\infty)$ and $r\geq1$. Similarly, for $d=3$, by the Sobolev embedding $\H_{\mathrm{p}}^1(\mathcal{O})\hookrightarrow\L^6(\mathcal{O})$, we find
\begin{align}\label{371}
	\|\u\|^{r+1}_{\L^{3(r+1)}(\mathcal{O})}&=\||\u|^{\frac{r+1}{2}}\|_{\L^{6}(\mathcal{O})}^2\leq C\||\u|^{\frac{r+1}{2}}\|_{\H_{\mathrm{p}}^1(\mathcal{O})}^2
	\nonumber\\&\leq C\bigg(\int_{\mathcal{O}}|\nabla\u(x)|^2|\u(x)|^{r-1}\d x+ \int_{\mathcal{O}}|\u(x)|^{r+1}\d x\bigg), 
\end{align}
for all $\u\in\D(\A)$ and for $r\geq1$. For $d=4$, we obtain 
\begin{align}\label{3p71}
	\|\u\|_{\widetilde{\L}^{2(r+1)}}^{r+1}=\||\u|^{\frac{r+1}{2}}\|_{\L^{4}(\mathcal{O})}^2\leq C\bigg(\int_{\mathcal{O}}|\nabla\u(x)|^2|\u(x)|^{r-1}\d x+\int_{\mathcal{O}}|\u(x)|^{r+1}\d x\bigg), 
\end{align}
for all $\u\in\D(\A)$ and $r\geq1$.   We also assume that the noise coefficient satisfies the following: 
\begin{hypothesis}\label{hyp2}
	There exist a positive
	constant $\widetilde{K}$ such that for all $t\in[0,T]$ and $\u\in\V$, 
	\begin{equation*}
	\|\nabla\Phi(t, \u)\|^{2}_{\mathcal{L}_{\Q}} 
	\leq \widetilde{K}\left(1 +\|\u\|_{\V}^{2}\right).
	\end{equation*}
\end{hypothesis}
\begin{theorem}\label{thm3.10}
Let $2\leq d\leq 4$ and $\u_0\in \mathrm{L}^2(\Omega;\V)$ be given.  Then {under Hypotheses \ref{hyp} and \ref{hyp2}}, for $r\geq 3$ ($2\beta\mu\geq 1$ for $r=3$), the pathwise unique strong solution	$\u(\cdot)$ to the system (\ref{32}) satisfies the following regularity:   \begin{align}\label{reg}
	\u&\in\mathrm{L}^2\left(\Omega;\mathrm{L}^{\infty}\left(0,T;\V\right)\cap\mathrm{L}^2\left(0,T;\D(\A)\right)\right)\cap\mathrm{L}^{r+1}\big(\Omega;\mathrm{L}^{r+1}\big(0,T;\widetilde{\L}^{p(r+1)}\big)\big),
\end{align} where $p\in[2,\infty)$ for $d=2$, $p=3$ for $d=3$, and $p=2$ for $d=4$, with $\mathscr{F}_t$-adapted paths of $\u\in\C([0,T];\V)\cap\mathrm{L}^2(0,T;\D(\A))\cap\mathrm{L}^{r+1}(0,T;\widetilde{\L}^{p(r+1)})$, $\mathbb{P}$-a.s.
\end{theorem}
\begin{proof}
Let us take $\u_0\in\mathrm{L}^2(\Omega;\V)$  to obtain the further regularity results of the strong solution to \eqref{32} with $r\geq 3$. In order to do this, we first consider the finite dimensional system given in \eqref{4.7}. Applying the finite dimensional It\^o formula to the process $\|\nabla\u^n(\cdot)\|_{\H}^2$, we obtain 
\begin{align}\label{362}
	&	\|\nabla\u^n(t)\|_{\H}^2+2\mu \int_0^t\|\A\u^n(s)\|_{\H}^2\d s + 2\alpha \int_0^t\|\nabla\u^n(s)\|_{\H}^2\d s \nonumber\\&=\|\nabla\u^n_0\|_{\H}^2-2\int_0^t(\B(\u^n(s)),\A\u^n(s))\d s	-2\beta\int_0^t(\mathcal{C}(\u^n(s)),\A\u^n(s))\d s\nonumber\\&\quad+\int_0^{t}\|\nabla\Phi(s,\u^n(s))\|^2_{\mathcal{L}_{\Q}}\d
	s +2\int_0^{t}\left(\nabla\Phi^n(s,\u^n(s))\d\W^n(s),\nabla\u^n(s)\right).
\end{align}
We estimate $|(\B(\u^n(s)),\A\u^n(s))|$ using H\"older's, and Young's inequalities as 
\begin{align}\label{3.73}
	|(\B(\u^n),\A\u^n)|&\leq\||\u^n||\nabla\u^n|\|_{\H}\|\A\u^n\|_{\H}\leq\frac{\mu}{2}\|\A\u^n\|_{\H}^2+\frac{1}{2\mu }\||\u^n||\nabla\u^n|\|_{\H}^2. 
\end{align}
For $r>3$, we  estimate the final term from \eqref{3.73} using H\"older's and Young's inequalities as 
\begin{align*}
	&	\int_{\mathcal{O}}|\u^n(x)|^2|\nabla\u^n(x)|^2\d x\nonumber\\&=\int_{\mathcal{O}}|\u^n(x)|^2|\nabla\u^n(x)|^{\frac{4}{r-1}}|\nabla\u^n(x)|^{\frac{2(r-3)}{r-1}}\d x\nonumber\\&\leq\left(\int_{\mathcal{O}}|\u^n(x)|^{r-1}|\nabla\u^n(x)|^2\d x\right)^{\frac{2}{r-1}}\left(\int_{\mathcal{O}}|\nabla\u^n(x)|^2\d x\right)^{\frac{r-3}{r-1}}\nonumber\\&\leq{\beta\mu }\left(\int_{\mathcal{O}}|\u^n(x)|^{r-1}|\nabla\u^n(x)|^2\d x\right)+\frac{r-3}{r-1}\left(\frac{2}{\beta\mu (r-1)}\right)^{\frac{2}{r-3}}\left(\int_{\mathcal{O}}|\nabla\u^n(x)|^2\d x\right).
\end{align*}
Making use of the estimate \eqref{370} in \eqref{3.73}, taking supremum over time from $0$ to $t$ and then taking  expectation, we have 
\begin{align}\label{363}
	&\E\bigg[\sup_{t\in[0,T]}	\|\nabla\u^n(t)\|_{\H}^2+\mu \int_0^T\|\A\u^n(t)\|_{\H}^2\d t+  2\alpha \int_0^t\|\nabla\u^n(s)\|_{\H}^2\d s 
	\nonumber\\ & \quad +\beta\int_0^T\||\u^n(t)|^{\frac{r-1}{2}}|\nabla\u^n(t)|\|_{\H}^2\d s
	\bigg]\nonumber\\&\leq\E\left[\|\nabla\u^n_0\|_{\H}^2\right]+2\varrho\E\left[\int_0^T\|\nabla\u^n(t)\|_{\H}^2\d t\right]+\E\left[\int_0^{T}\|\nabla\Phi(t,\u^n(t))\|^2_{\mathcal{L}_{\Q}}\d
	t\right] \nonumber\\&\quad+2\E\left[\sup_{t\in[0,T]}\left|\int_0^{t}\left(\nabla\Phi(s,\u^n(s))\d\W(s),\nabla\u^n(s)\right)\right|\right],
\end{align}
where $\varrho$ is defined in \eqref{215}.
Applying Burkholder-Davis-Gundy inequality to the final term appearing in the inequality \eqref{363}, we obtain 
\begin{align}\label{364}
	&2\E\left[\sup_{t\in[0,T]}\left|\int_0^{t}\left(\nabla\Phi(s,\u^n(s))\d\W(s),\nabla\u^n(s)\right)\right|\right]\nonumber\\&\leq 2\sqrt{3}\E\left[\int_0^T\|\nabla\Phi(t,\u^n(t))\|_{\mathcal{L}_{\Q}}^2\|\nabla\u^n(t)\|_{\H}^2\d t\right]^{1/2}\nonumber\\&\leq 2\sqrt{3}\E\left[\sup_{t\in[0,T]}\|\nabla\u^n(t)\|_{\H}\left(\int_0^T\|\nabla\Phi(t,\u^n(t))\|_{\mathcal{L}_{\Q}}^2\d t\right)^{1/2}\right]\nonumber\\&\leq\frac{1}{2}\E\left[\sup_{t\in[0,T]}\|\nabla\u^n(t)\|_{\H}^2\right]+6\E\left[\int_0^T\|\nabla\Phi(t,\u^n(t))\|_{\mathcal{L}_{\Q}}^2\d t\right].
\end{align}
Substituting \eqref{364} in \eqref{363}, we find
\begin{align}\label{365}
	&\E\bigg[\sup_{t\in[0,T]}	\|\nabla\u^n(t)\|_{\H}^2+2\mu \int_0^T\|\A\u^n(t)\|_{\H}^2\d t + 2\alpha \int_0^t\|\nabla\u^n(s)\|_{\H}^2\d s 
	\nonumber\\ & \quad +2\beta\int_0^T\||\u^n(t)|^{\frac{r-1}{2}}|\nabla\u^n(t)|\|_{\H}^2\d t
	\bigg]\nonumber\\&\leq 2\E\left[\|\nabla\u^n_0\|_{\H}^2\right]+4\varrho\E\left[\int_0^T\|\nabla\u^n(t)\|_{\H}^2\d t\right]+14\E\left[\int_0^{T}\|\nabla\Phi(t,\u^n(t))\|^2_{\mathcal{L}_{\Q}}\d
	t\right]\nonumber\\&\leq 2\E\left[\|\nabla\u_0\|_{\H}^2\right]+4\varrho\E\left[\int_0^T\|\nabla\u^n(t)\|_{\H}^2\d t\right]+14\widetilde{K}\E\left[\int_0^{T}(1+\|\u^n(t)\|_{\H}^2+\|\nabla\u^n(t)\|_{\H}^2)\d
	t\right],
\end{align}
where we have used Hypothesis \ref{hyp}. 
Applying Gr\"onwall's inequality in \eqref{365}, we obtain 
\begin{align}\label{366}
	&\E\left[\sup_{t\in[0,T]}	\|\nabla\u^n(t)\|_{\H}^2\right]
	\nonumber\\ & \leq \left\{2\E\left[\|\nabla\u_0\|_{\H}^2\right]+14\widetilde{K}T +14\widetilde{K}T \left(2\E\left[\|\u_0\|_{\H}^2\right]+14KT\right)e^{28KT}\right\}e^{(4\varrho+14\widetilde{K})T},
\end{align}
where we have used \eqref{eng1}. Using \eqref{366} in \eqref{365}, we finally get 
\begin{align}\label{367}
	&\E\left[\sup_{t\in[0,T]}	\|\nabla\u^n(t)\|_{\H}^2+2\mu \int_0^T\|\A\u^n(t)\|_{\H}^2\d t + 4\alpha \int_0^T\|\nabla\u^n(t)\|_{\H}^2\d t +2\beta\int_0^T\|\u^n(t)\|_{\wi\L^{p(r+1)}}^{r+1}\d t  \right]\nonumber\\&  \leq \left\{2\E\left[\|\nabla\u_0\|_{\H}^2\right]+14\widetilde{K}T +14\widetilde{K}T \left(2\E\left[\|\u_0\|_{\H}^2\right]+14KT\right)e^{28KT}\right\}e^{(4\varrho+14\widetilde{K})T},
\end{align}
where $p\in[2,\infty)$ for $d=2$, $p=3$ for $d=3$, and $p=2$ for $d=4$. Note that the right hand side of the inequality \eqref{367} is independent of $n$. Thus, making use of the Banach-Alaoglu theorem, we infer from \eqref{367} that 
\begin{equation}\label{4p40}
	\left\{
	\begin{aligned}
		\u^n&\xrightarrow{w^{*}} \wi\u\ \textrm{ in
		}\ \mathrm{L}^2(\Omega;\mathrm{L}^{\infty}(0,T ;\V)),\\ 
		\u^n&\xrightarrow{w} \wi\u\ \textrm{ in
		}\ \mathrm{L}^2(\Omega;\mathrm{L}^{2}(0,T ;\D(\A))),\\  
		\u^n&\xrightarrow{w} \wi\u\ \textrm{ in
		} \ \mathrm{L}^{r+1}(\Omega;\mathrm{L}^{r+1}(0,T ;\widetilde{\L}^{p(r+1)})).
	\end{aligned}\right.\end{equation}
Because the sequence $\{\u^n\}$ also satisfies the convergence \eqref{4.40} and  the weak limit is unique,  we obtain $\u=\wi\u$, where $\u(\cdot)$ satisfies \eqref{4.5}. Hence, we get the regularity given in \eqref{reg}.

For $r=3$, we estimate $|(\B(\u^n),\A\u^n)|$ as 
\begin{align}\label{2.91}
	|(\B(\u^n),\A\u^n)|&\leq\|(\u^n\cdot\nabla)\u^n\|_{\H}\|\A\u^n\|_{\H}\leq\frac{1}{4\theta}\|\A\u^n\|_{\H}^2+\theta\||\u^n|\nabla\u^n\|_{\H}^2.
\end{align}
A calculation similar to \eqref{365} gives 
\begin{align}\label{2.92}
	&\E\bigg[\sup_{t\in[0,T]}	\|\nabla\u^n(t)\|_{\H}^2+2\left(\mu-\frac{1}{2\theta}\right) \int_0^T\|\A\u^n(t)\|_{\H}^2\d t + 4\alpha \int_0^T\|\nabla\u^n(t)\|_{\H}^2\d t \nonumber\\ & \ \quad +4(\beta-\theta)\int_0^T\||\u^n(t)||\nabla\u^n(t)|\|_{\H}^2\d t\bigg]\nonumber\\& \leq 2\E\left[\|\nabla\u_0\|_{\H}^2\right] +14\widetilde{K}\E\left[\int_0^{T}(1+\|\u^n(t)\|_{\H}^2+\|\nabla\u^n(t)\|_{\H}^2)\d
	t\right],
\end{align} 
For $2\beta\mu \geq 1$,  it is immediate that 
\begin{align}\label{2p94}
	&\E\bigg[\sup_{t\in[0,T]}	\|\nabla\u^n(t)\|_{\H}^2+2\left(\mu-\frac{1}{2\theta}\right)\int_0^T\|\A\u^n(t)\|_{\H}^2\d t+ 4 \alpha \int_0^T\|\nabla\u^n(t)\|_{\H}^2\d t 
	\nonumber\\ & \quad +4(\beta-\theta)\int_0^T\||\u^n(t)||\nabla\u^n(t)|\|_{\H}^2\d t \bigg]\nonumber\\& \leq \left\{2\E\left[\|\nabla\u_0\|_{\H}^2\right]+14\widetilde{K}T +14\widetilde{K}T \left(2\E\left[\|\u_0\|_{\H}^2\right]+14KT\right)e^{28KT}\right\}e^{14\widetilde{K}T}.
\end{align}
Hence $\u^n\in\mathrm{L}^2(\Omega;\mathrm{L}^{\infty}(0,T;\V)\cap\mathrm{L}^2(0,T;\D(\A)))$ and using the estimate \eqref{371}, we also get $\u^n\in\mathrm{L}^{4}(\Omega;\mathrm{L}^{4}(0,T;\widetilde{\L}^{4p}))$, where $p\in[2,\infty)$ for $d=2$, $p=3$ for $d=3$, and $p=2$ for $d=4$. Moreover, the $\mathscr{F}_t$-adapted paths of $\u^n(\cdot)$ are continuous with trajectories in  $\C([0,T];\V)\cap\mathrm{L}^2(0,T;\D(\A))$, $\mathbb{P}$-a.s.
	\end{proof}

\section{Stationary solutions and stability}\label{se5}\setcounter{equation}{0}
In this section, we consider the deterministic stationary system  corresponding to CBF equations. We show the existence and uniqueness of weak solutions to the steady state equations and discuss exponential stability as well as stabilization by noise results. 
\subsection{Existence and uniqueness of weak solutions to stationary system} 
Let us consider the following stationary system: 
\begin{equation}\label{3pp1}
\left\{
\begin{aligned}
-\mu \Delta\u_{\infty}+(\u_{\infty}\cdot\nabla)\u_{\infty}+ \alpha \u_{\infty}+\beta|\u_{\infty}|^{r-1}\u_{\infty}+\nabla p_{\infty}&=\boldsymbol{f}, \ \text{ in } \ \mathcal{O}, \\ \nabla\cdot\u_{\infty}&=0, \ \text{ in } \ \mathcal{O}, \\
\u_{\infty}&=\mathbf{0}\ \text{ on } \ \partial\mathcal{O}.
\end{aligned}
\right.
\end{equation}
Taking the Helmholtz-Hodge orthogonal projection onto the system \eqref{3pp1}, one can write down the abstract formulation of the system \eqref{3pp1} as 
\begin{align}\label{3pp2}
\mu \A\u_{\infty}+\B(\u_{\infty}) + \alpha\u_{\infty} +\beta\mathcal{C}(\u_{\infty})=\f \ \text{ in }\ \V'.
\end{align}
We show that there exists a unique weak solution to the system \eqref{3pp2} in  $\V\cap\widetilde{\L}^{r+1}$, for $r\geq 3$. 
Given any $\f\in\V'$, our problem is to find $\u_{\infty}\in\V\cap\widetilde{\L}^{r+1}$ such that 
\begin{align}\label{3pp3}
\mu (\nabla\u_{\infty},\nabla\v)+\langle\B(\u_{\infty}),\v\rangle + \alpha (\u_{\infty}, \v)+\beta\langle \mathcal{C}(\u_{\infty}),\v\rangle=\langle \f,\v\rangle, \ \text{ for all }\ \v\in\V\cap\widetilde{\L}^{r+1}, 
\end{align}
is satisfied. Our next aim is to discuss the existence and uniqueness of weak solutions of the system \eqref{3pp1}. 
\begin{theorem}\label{thm6.1}
	For every $\f\in\V'$ and $r> 3$, there exists at least one weak solution to the system \eqref{3pp1}, and if 
 	\begin{align}\label{8pp5}
	\alpha>\varrho \ \text{ or }\ \mu>\max\left\{\frac{1}{2\beta},\frac{1}{4\alpha}\right\},
	\end{align}
where $\varrho$ is defined in \eqref{215}, 	then the solution of \eqref{3pp1} is unique. For $r=3$ with $\mu\geq \frac{1}{2\beta}$, there exists a unique weak solution to the system \eqref{3pp1}. 
\end{theorem}
\begin{proof}
	(i) We show the existence of weak solution to the system \eqref{3pp1} (equivalently the existence of \eqref{3pp2}) by applying  the  Faedo-Galerkin approximation technique. Let $\{\bolde_k\}_{k\in\N}$ be eigenfunctions of the compact  operator defined in \eqref{L3}. 
	For a fixed positive integer $m$, we look for a function $\u^m_{\infty}\in\V$ of the form \begin{align}\label{8p4}
	\u_{\infty}^m=\sum_{k=1}^m\xi_m^k\bolde_k,\ \xi_m^k\in\mathbb{R},
\end{align} and 
	\begin{align}\label{8p5}
	\mu (\nabla\u_{\infty}^m,\nabla \bolde_k)+(\B(\u_{\infty}^m),\bolde_k)+ \alpha (\u_{\infty}^m, \bolde_k) +\beta( \mathcal{C}(\u_{\infty}^m),\bolde_k)=\langle\f,\bolde_k\rangle,
	\end{align}
	for   $k=1,\ldots,m$. Equivalently, the equation \eqref{8p5} can also be written as 
	\begin{align}\label{eqn-stationary}
\mu \A\u_{\infty}^m+ \mathrm{P}_m\B(\u_{\infty}^m)+ \alpha \u_{\infty}^m +\beta \mathrm{P}_m\mathcal{C}(\u_{\infty}^m)=\mathrm{P}_m\f. 
	\end{align}
	Note that the equations \eqref{8p4}-\eqref{8p5} are a system of nonlinear equations for $\xi_m^1,\ldots,\xi_m^m$ and the existence of solutions is proved in the following way. 
	We use \cite[Lemma 1.4, Chapter 2]{Te} to obtain the existence of a solution to the system of equations \eqref{8p4}-\eqref{8p5}. We 
	take $\W=\mathrm{span}\left\{\bolde_1,\ldots,\bolde_m\right\}$ and the scalar product on $\W$ is the scalar product $[\cdot,\cdot]=(\cdot,\cdot)+(\nabla\cdot,\nabla\cdot)$ induced by $\V$ and $\mathrm{P}=\mathrm{P}_m$ is defined by 
	\begin{align}
	[\mathrm{P}_m(\u),\v]& =(\mathrm{P}_m(\u),\v) + (\nabla\mathrm{P}_m(\u),\nabla\v)
	\nonumber\\ & := \mu (\nabla\u,\nabla\v)+ b(\u,\u,\v) + \alpha (\u, \v)+\beta (\mathcal{C}(\u),\v)-\left\langle\f,\v\right\rangle,
	\end{align}
	for all $\u,\v\in \W$. The continuity of the mapping $\mathrm{P}_m:\W\to\W$ is easy to verify, as in finite-dimensional space all norms are equivalent.  In order to apply Lemma 1.4, Chapter 2, \cite{Te}, we need to establish that $$[\mathrm{P}_m(\u),\u]>0, \ \text{ for } \ [\u]=k>0,$$ where $[\cdot]$ denotes the norm on $\W$. Note that it is the norm induced by $\V$. Next, we consider 
	\begin{align}
	[\mathrm{P}_m(\u),\u]&=\mu \|\nabla\u\|_{\H}^2+ \alpha \|\u\|_{\H}^2 +\beta\|\u\|_{\widetilde{\L}^{r+1}}^{r+1}-\left\langle\f,\u\right\rangle
	\nonumber\\ & \geq\frac{\min\{\mu,\alpha\} }{2}[\|\u\|_{\H}^2 + \|\nabla\u\|_{\H}^2]-\frac{1}{2\min\{\mu,\alpha\}  }\|\f\|_{\V'}^2,
	\end{align}
	where we have used the Cauchy-Schwarz inequality  and Young's inequality. 
	It follows that $[\mathrm{P}_m(\u),\u]>0$ for $\|\u\|_{\V}=k$ and $k$ is sufficiently large, more precisely $k>\frac{1}{\min\{\mu,\alpha\} }\|\f\|_{\V'}.$ Hence, the hypotheses of \cite[Lemma 1.4, Chapter 2]{Te} are satisfied and a solution $\u^m_{\infty}$  of \eqref{8p5} exists. 

	Multiplying \eqref{8p5} by $\xi_m^k$ and then adding from $k=1,\ldots,m$, we find 
	\begin{align}\label{8.9}
	\mu \|\nabla\u^{m}_{\infty}\|_{\H}^2+ \alpha \|\u^{m}_{\infty}\|_{\H}^2 +\beta\|\u^{m}_{\infty}\|_{\widetilde{\L}^{r+1}}^{r+1}&=(\mathrm{P}_m\f,\u^{m}_{\infty})\leq\|\f\|_{\V'}\|\u^{m}_{\infty}\|_{\V}
	\nonumber\\&\leq\frac{\min\{\mu,\alpha\}}{2}\|\u^{m}_{\infty}\|_{\V}^2+\frac{1}{2\min\{\mu,\alpha\} }\|\f\|_{\V'}^2,
	\end{align}
	where we have used H\"older's and Young's inequalities. From \eqref{8.9}, we deduce that 
	\begin{align}
	\mu \|\nabla\u^{m}_{\infty}\|_{\H}^2+ \alpha \|\u^{m}_{\infty}\|_{\H}^2 +2\beta\|\u^{m}_{\infty}\|_{\widetilde{\L}^{r+1}}^{r+1}\leq\frac{1}{\min\{\mu,\alpha\} }\|\f\|_{\V'}^2.
	\end{align}
	Thus, we get  $\|\u^{m}_{\infty}\|_{\V}^2$ and $\|\u^{m}_{\infty}\|_{\widetilde{\L}^{r+1}}^{r+1}$ are  bounded uniformly and independent of $m$. Since $\V$ and $\widetilde{\L}^{r+1}$ are reflexive, using the Banach-Alaoglu theorem, we can extract a subsequence $\{\u^{m_k}_{\infty}\}$ of $\{\u^{m}_{\infty}\}$ such that 
	\begin{align}
		\u^{m_k}_{\infty}&\xrightarrow{w} \u_{\infty}, \ \text{ in }\ \V, \label{5121}\\
		\u^{m_k}_{\infty}&\xrightarrow{w} \u_{\infty}, \ \text{ in }\ \widetilde{\L}^{r+1}, \label{5122}
	\end{align}
	as $k\to\infty$.
Since the embedding $\H^1_{\mathrm{loc}}(\mathcal{O})\subset \L^2_{\mathrm{loc}}(\mathcal{O})$ is compact, one can extract a subsequence $\{\u^{m_{k_j}}_{\infty}\}$ of $\{\u^{m_k}_{\infty}\}$ such that 
	\begin{align}\label{5123}
	\u^{m_{k_j}}_{\infty}\to \u_{\infty}, \ \text{ in }\ \L^2_{\mathrm{loc}}(\mathcal{O}), 
	\end{align}
as $j\to\infty$. Next we aim to pass limit in 
\begin{align}\label{3pp4}
	\mu (\nabla\u_{\infty}^{m_{k_j}},\nabla\v) + \langle\B(\u_{\infty}^{m_{k_j}}),\v\rangle + \alpha (\u_{\infty}^{m_{k_j}}, \v)+\beta\langle \mathcal{C}(\u_{\infty}^{m_{k_j}}),\v\rangle=\langle \P_{m_{k_j}}\f,\v\rangle, 
\end{align}
for all $\v\in\V\cap\widetilde{\L}^{r+1}$, along the subsequence $\{m_{k_j}\}$. In view of the weak convergence in \eqref{5121} and the fact that $\P_{m_{k_j}}\f \to \f$ in $\V^{\prime}$, we get
\begin{align}
	\mu(\nabla\u_{\infty}^{m_{k_j}},\nabla\v) + \alpha(\u_{\infty}^{m_{k_j}}, \v) - \langle \P_{m_{k_j}}\f,\v\rangle & \to \mu(\nabla\u_{\infty},\nabla\v) +  \alpha(\u_{\infty}, \v) -  \langle \f,\v\rangle,
\end{align}
for all $\v\in \V\cap\wi\L^{r+1}$. Further, using the strong convergence in \eqref{5123}, one can obtain (see e.g. \cite[page 114]{Te})
\begin{align*}
	\langle\B(\u_{\infty}^{m_{k_j}}),\phi\rangle  +\beta\langle \mathcal{C}(\u_{\infty}^{m_{k_j}}),\phi\rangle \to \langle\B(\u_{\infty}),\phi\rangle  +\beta\langle \mathcal{C}(\u_{\infty}),\phi\rangle
\end{align*}
for all $\phi\in C_c^{\infty}(\mathcal{O})$, and using the density of $C_c^{\infty}(\mathcal{O})$ in $\V\cap\wi\L^{r+1}$, we reach at
\begin{align*}
	\langle\B(\u_{\infty}^{m_{k_j}}),\v\rangle  +\beta\langle \mathcal{C}(\u_{\infty}^{m_{k_j}}),\v\rangle \to \langle\B(\u_{\infty}),\v\rangle  +\beta\langle \mathcal{C}(\u_{\infty}),\v\rangle
\end{align*}
for all $\v\in \V\cap\wi\L^{r+1}$.

Hence, we find that $\u_{\infty}$ is a solution to \eqref{3pp3} and $\u_{\infty}$ satisfies 
	\begin{align}\label{8.14}
	\mu 	\|\nabla\u_{\infty}\|_{\H}^2 + \alpha 	\|\u_{\infty}\|_{\H}^2 +2\beta\|\u_{\infty}\|_{\widetilde{\L}^{r+1}}^{r+1}\leq\frac{1}{\mu }\|\f\|_{\V'}^2,
	\end{align}
	for $r\geq 1$.

	(iii) Let us now establish the uniqueness for $r>3$.  We take $\u_{\infty}$ and $\v_{\infty}$ as two weak solutions of the system \eqref{3pp3}. We define $\w_{\infty}:=\u_{\infty}-\v_{\infty}$. Then $\w_{\infty}$ satisfies:
	\begin{align}\label{8p17}
	\mu (\nabla\w_{\infty},\nabla\v)+\langle\B(\u_{\infty})-\B(\v_{\infty}),\v\rangle + \alpha (\w_{\infty},\v) +\beta\langle\mathcal{C}(\u_{\infty})-\mathcal{C}(\v_{\infty}),\v\rangle=0,
	\end{align}
	for all $\v\in\V$. Taking $\v=\w_{\infty}$ in \eqref{8p17}, we obtain 
	\begin{align}\label{8p18}
	\mu \|\nabla\w_{\infty}\|_{\H}^2 + \alpha \|\w_{\infty}\|_{\H}^2 &=-\langle\B(\u_{\infty})-\B(\v_{\infty}),\w_{\infty}\rangle-\beta\langle\mathcal{C}(\u_{\infty})-\mathcal{C}(\v_{\infty}),\w_{\infty}\rangle\nonumber\\&\leq\frac{\mu}{2}\|\nabla\w_{\infty}\|_{\H}^2+\varrho\|\w_{\infty}\|_{\H}^2,
	\end{align}
	where we have used \eqref{2.27} and \eqref{2.30}. 
	Therefore, we  get 
	\begin{align}\label{8.19}
	\frac{\mu }{2}\|\nabla\w_{\infty}\|_{\H}^2+ (\alpha-\varrho)\|\w_{\infty}\|_{\H}^2\leq 0,
	\end{align}
and for $\alpha>\varrho$, where $\varrho$ is defined in \eqref{215}. On the other hand, by using the estimate \eqref{2.26}, the uniqueness also follows  for $\mu>\max\left\{\frac{1}{2\beta},\frac{1}{4\alpha}\right\}$. Thus, we have $\u_{\infty}=\v_{\infty}$, for a.e. $x\in\mathcal{O}$. For $r=3$ and $2\beta\mu\geq 1$, one can use the estimate \eqref{232} to obtain the uniqueness. 
	\end{proof} 
\begin{remark}
On bounded domains, due to Poincar\'e inequality \eqref{poin}, the condition given in \eqref{8pp5} becomes $\frac{\mu\lambda_1}{2}+\alpha>\varrho$. 
\end{remark}

\begin{remark}
One can get uniqueness for $r\geq 1$ also. For $2\leq d\leq 4$, we know that $\|\u\|_{\wi\L^4}\leq C\|\u\|_{\V}$, by using the Sobolev inequality.	We estimate $-\langle\B(\w_{\infty},\v_{\infty}),\w_{\infty}\rangle$ using H\"older's and Sobolev's inequalities as 
	\begin{align}\label{6.49}
	-\langle\B(\w_{\infty},\v_{\infty}),\w_{\infty}\rangle&\leq\|\w_{\infty}\|_{\wi\L^4}^2\|\nabla\v_{\infty}\|_{\H}\leq C\|\v_{\infty}\|_{\V}\|\nabla\w_{\infty}\|_{\H}^2.
	\end{align}
	Using \eqref{6.49} and \eqref{2.23} in \eqref{8p18}, we  get 
	\begin{align}\label{819}
	\left(\mu-C\|\v_{\infty}\|_{\V}\right)\|\nabla\w_{\infty}\|_{\H}^2\leq 0.
	\end{align}
	Since $\v_{\infty}$ satisfies \eqref{8.14}, from \eqref{819}, we obtain 
	\begin{align}\label{8.22}
	\left(\mu -\frac{C}{\mu\min\{\mu,\alpha\}}\|\f\|_{\V'}\right)\|\nabla\w_{\infty}\|_{\H}^2\leq 0.
	\end{align}
	If the smallness condition $\mu^2\min\{\mu,\alpha\}>C{\|\f\|_{\V'}}$ is satisfied, then we have $\u_{\infty}=\v_{\infty}$. 
\end{remark}

\subsection{Exponential stability} Let us now discuss the exponential stability of the stationary solution obtained in Theorem \ref{thm6.1}. We first describe the deterministic case. 
\begin{definition}
	A weak solution\footnote{For $r\geq 3,$ the existence and uniqueness of weak solutions can be obtained from \cite{SNA,CLF,KT2,Kinra+Mohan_2024_DCDS,Gautam+Mohan_2025}, etc} $\u^d(t)$ of the system  of the deterministic system 
	\begin{equation}\label{4p22}
	\left\{
	\begin{aligned}
	\frac{\d\u^d(t)}{\d t}+\mu \A\u^d(t)+\B(\u^d(t))+\alpha\u^d(t)+\beta\mathcal{C}(\u^d(t))&=\f,\\
	\u^d(0)&=\u_0,
	\end{aligned}
	\right.
	\end{equation}converges to $\u_{\infty}$ is \emph{exponentially stable in $\H$} if there exist a positive number $\kappa > 0$, such that
	\begin{align*}
	\|\u^d(t)-\u_{\infty}\|_{\H}\leq \|\u_0-\u_{\infty}\|_{\H}e^{-\kappa t},\ t\geq 0.
	\end{align*}
	In particular, if $\u_{\infty}$ is a stationary solution to the system (\ref{3pp2}), then $\u_{\infty}$ is called \emph{exponentially stable in $\H$} provided that any weak solution to the system (\ref{4p22}) converges to $\u_{\infty}$ at the same exponential rate $\kappa> 0$.
\end{definition}
\begin{theorem}
	Let $\u_{\infty}$ be the unique weak solution to the system (\ref{3pp2}). If $\u^d(\cdot)$ is any weak  solution to the system \eqref{4p22}
	with $\u_0\in\H$ and $\f\in\V'$ arbitrary, then we have $\u_{\infty}$ is exponentially stable in $\H$ and $
	\u^d(t)\to \u_{\infty}\ \text{ in }\ \H\text{ as }\  t\to\infty,
	$
	for
	$\alpha>\varrho$\footnote{For a bounded domain $\mathcal{O}$, this condition can be replaced with $\frac{\mu\lambda_1}{2}+\alpha>\varrho$.}, for $r>3$ and $\mu> \frac{1}{2\beta},$ for $r=3$. 
\end{theorem}
\begin{proof} 
	Let us define $\w^d=\u^d-\u_{\infty}$, so that $\w^d$ satisfies the following system: 
	\begin{equation}\label{8.26}
	\left\{
	\begin{aligned}
&	\frac{\d\w^d(t)}{\d t}+\mu  \A\w^d(t)+\beta (\B(\u^d(t))-\B(\u_{\infty})) + \alpha \w^d(t) +\beta (\mathcal{C}(\u^d(t))-\mathcal{C}(\u_{\infty}))=0, \\
&	\w^d(0)=\u_0-\u_{\infty}.
	\end{aligned}
	\right.
	\end{equation}
in $\V'+\wi\L^{\frac{r+1}{r}}$, for a.e. $t\in[0,T]$.	Taking the inner product with $\w(\cdot)$ to the first equation in $\eqref{8.26}_1$, we find 
	\begin{align}\label{8.27}
	&\frac{1}{2}\frac{\d}{\d t}\|\w^d(t)\|_{\H}^2+\mu \|\nabla\w^d(t)\|_{\H}^2 + \alpha \|\w^d(t)\|_{\H}^2 \nonumber\\&=-\beta\langle (\B(\u^d(t))-\B(\u_{\infty})),\w^d(t)\rangle -\beta\langle(\mathcal{C}(\u^d(t))-\mathcal{C}(\u_{\infty})),
	\w^d(t)\rangle\nonumber\\&\leq\frac{\mu }{2}\|\nabla\w^d(t)\|_{\H}^2+\varrho\|\w^d(t)\|_{\H}^2,
	\end{align}
	for $r>3$,	where we have used \eqref{2.27} and \eqref{2.30}. Thus, it is immediate that 
	\begin{align}
	& \frac{\d}{\d t}\|\w^d(t)\|_{\H}^2+\mu \|\nabla\w^d(t)\|_{\H}^2 + 2(\alpha-\varrho) \|\w^d(t)\|_{\H}^2 \leq 0. 
	\end{align}
	Hence, an application of variation of constants formula yields 
	\begin{align}
	\|\u^d(t)-\u_{\infty}\|_{\H}^2\leq e^{-\kappa t}\|\u_0-\u_{\infty}\|_{\H}^2, 
	\end{align}
	where $\kappa=(\alpha - \varrho)>0$, for $\alpha>\varrho$ and the exponential stability of $\u_{\infty}$ follows. For $r=3$ and $\mu\geq\frac{1}{2\beta}$ one can use the estimates \eqref{231} and \eqref{232} to get the required result. 
\end{proof} 
Now we discuss the exponential stability results in  the stochastic case. 
\begin{definition}
	A strong solution $\u(t)$ to the system (\ref{32}) converges to $\u_{\infty}\in\H$ is \emph{exponentially stable in mean square} if there exists a positive number $a > 0$, such that
	\begin{align*}
	\E\left[\|\u(t)-\u_{\infty}\|_{\H}^2\right]\leq \E\left[\|\u_0-\u_{\infty}\|_{\H}^2\right]e^{-at},\ t\geq 0.
	\end{align*}
	In particular, if $\u_{\infty}$ is a stationary solution to the system (\ref{3pp1}), then $\u_{\infty}$ is called \emph{exponentially stable in the mean square} provided that any strong solution to the system (\ref{32}) converges in $\H$ to $\u_{\infty}$ at the same exponential rate $a > 0$.
\end{definition}

\begin{definition}
	A strong solution $\u(t)$ to the system (\ref{32}) converges to $\u_{\infty} \in\H$ \emph{almost surely exponentially
		stable} if there exists $\upalpha >0$ such that
	$$\lim_{t\to+\infty}\frac{1}{t}\log\|\u(t)-\u_{\infty}\|_{\H}\leq -\upalpha,\ \mathbb{P}\text{-a.s.}$$
	In particular, if $\u_{\infty}$ is a stationary solution to the system (\ref{3pp1}), then $\u_{\infty}$ is called \emph{almost surely exponentially stable} provided that any strong solution to the system (\ref{32}) converges in $\H$ to $\u_{\infty}$ with the same constant $\upalpha> 0$.
\end{definition}

Let us now show the exponential stability of the stationary solutions to the system \eqref{3pp1} in the mean square as well as almost sure sense. The authors in \cite{LHGH2} obtained similar results, but with more regularity on the stationary solutions as well as the lower bound of $\mu$ depends on the stationary solutions. We relax those conditions in this work. The following results are true for all $r\geq  3$, and one has to take $2\beta\mu\geq  1$, for $r=3$. The system under our consideration is 
\begin{equation}\label{3.2}
\left\{
\begin{aligned}
\d\u(t)+\mu \A\u(t)+\B(\u(t))+ \alpha\u(t)+\beta\mathcal{C}(\u(t))&=\f+\Phi(t,\u(t))\d\W(t), \\
\u(0)&=\u_0,
\end{aligned}
\right.
\end{equation}
for a.e. $t\in[0,T]$, where $\f\in\V'$ and $\u_0\in\mathrm{L}^{4+\eta}(\Omega;\H)$ for some $\eta>0$.

\begin{theorem}\label{exp1}
	Let $\u_{\infty}$ be the unique stationary solution to the system  (\ref{3pp1}) and  $\Phi(t,\u_{\infty})=0,$ for all $t\geq 0$. Suppose that the conditions in Hypothesis \ref{hyp} are satisfied, then for $\theta=2\alpha-(2\varrho+L)>0,$ we have 
	\begin{align}\label{6.16}
&\E\left[\|\u(t)-\u_{\infty}\|_{\H}^2\right]\leq e^{-\theta t}\E\left[\|\u_0-\u_{\infty}\|_{\H}^2\right],
\end{align}
	provided\footnote{For a bounded domain $\mathcal{O}$, we can replace this condition with $\mu\lambda_1+ 2\alpha >2\varrho+L$, where $\lambda_1$ is the Poincar\'e constant.} 
	\begin{align}\label{6.3a}
	2\alpha > 2\varrho+L,
	\end{align}
	where $L$ is the constant appearing in Hypothesis \ref{hyp} (H.3) and $\varrho$ is defined in \eqref{215}.
\end{theorem}
\begin{proof}
	Let us define $\w:=\u-\u_{\infty}$ and $\theta=2\alpha-(2\varrho+L)>0$. Then $\w$ satisfies the following It\^o stochastic differential: 
	\begin{align}
	\left\{
	\begin{aligned}
	\d\w(t)+\mu \A\w(t)+(\B(\u(t))-\B(\u_{\infty}))&+ \alpha\w(t) +\beta(\mathcal{C}(\u(t))-\mathcal{C}(\u_{\infty}))\\&=(\Phi(t,\u(t))-\Phi(t,\u_{\infty}))\d\W(t), \\
	\w(0)&=\u_0-\u_{\infty},
	\end{aligned}
	\right.
	\end{align}
for a.e. $t\in[0,T]$,	since $\Phi(t,\u_{\infty})=0$, for all $t\in[0,T]$. Then $\w(\cdot)$ satisfies the following energy equality: 
	\begin{align}
e^{\theta t}	\|\w(t)\|_{\H}^2&=\|\w_0\|_{\H}^2-2\int_0^{t}e^{\theta s}\langle\B(\u(s))-\B(\u_{\infty}(s)),\w(s)\rangle\d s+\theta\int_0^{t}e^{\theta s}\|\w(s)\|_{\H}^2\d s
\nonumber\\ & \quad - 2\mu \int_0^{t}e^{\theta s}\|\nabla\w(s)\|_{\H}^2\d s - 2 \alpha\int_0^{t}e^{\theta s}\|\w(s)\|_{\H}^2\d s
\nonumber\\&\quad-2\int_0^{t}e^{\theta s}\langle\mathcal{C}(\u(s))-\mathcal{C}(\u_{\infty}(s)),\w(s)\rangle\d s+\int_0^{t}e^{\theta s}\|\wi\Phi(s)\|_{\mathcal{L}_{\Q}}^2\d s\nonumber\\&\quad+2\int_0^{t}e^{\theta s}(\widetilde{\Phi}(s)\d\W(s),\w(s)),
	\end{align}
where $\widetilde{\Phi}(\cdot)=\Phi(\cdot,\u(\cdot))-\Phi(\cdot,\u_{\infty})$. 	A calculation similar to \eqref{4.60} yields 
	\begin{align}\label{4.22}
e^{\theta t}\|\w(t)\|_{\H}^2&\leq\|\w_0\|_{\H}^2 +\left(\theta+2\varrho-2\alpha\right)\int_0^{t}e^{\theta s}\|\w(s)\|_{\H}^2\d s - \mu \int_0^{t}e^{\theta s}\|\nabla\w(s)\|_{\H}^2\d s 
\nonumber\\&\quad +\int_0^{t}e^{\theta s}\|\widetilde{\Phi}(s)\|^2_{\mathcal{L}_{\Q}}\d
s  +2\int_0^{t}e^{\theta s}(\widetilde{\Phi}(s)\d\W(s),\w(s)).
\end{align}
Taking expectation, and using Hypothesis \ref{hyp} (H.3) and the fact that the final term is a martingale, we find 
\begin{align}
	&e^{\theta t}\E\left[\|\w(t)\|_{\H}^2\right]\leq \E\left[\|\w_0\|_{\H}^2\right]+\left(\theta+2\varrho+L-2\alpha\right)\int_0^{t}e^{\theta s}\E\left[\|\w(s)\|_{\H}^2 \right]\d s.
\end{align}
Since $\mu$ satisfies \eqref{6.3a} implies $\theta=2\alpha-(2\varrho+L)>0$, we infer that  \eqref{6.16} is satisfied 	and hence $\u(t)$ converges to $\u_{\infty}$ exponentially in the mean square sense.
\end{proof}

\begin{theorem}\label{exp2}
	Let all conditions given in Theorem \ref{exp1} are satisfied and\footnote{For a bounded domain, this condition can be replaced with $\mu\lambda_1+2\alpha>2\varrho+7L$.} 
	\begin{align}
	2\alpha>2\varrho+7L.
	\end{align}
	 Then the strong solution $\u(\cdot)$ of the system (\ref{32}) converges to the stationary solution $\u_{\infty}$ of the system (\ref{3pp2}) almost surely exponentially stable.
\end{theorem}
\begin{proof}
	Let us take  $n=1,2,\ldots$ and $h>0$. Then the process $\|\u(\cdot)-\u_{\infty}\|_{\H}^2,$ for $t\geq nh$ satisfies:
	\begin{align}\label{6.17}
&	\|\u(t)-\u_{\infty}\|_{\H}^2+2\mu\int_{nh}^t\|\nabla(\u(s)-\u_{\infty})\|_{\H}^2\d s +2\alpha\int_{nh}^t\|\u(s)-\u_{\infty}\|_{\H}^2\d s \nonumber\\&=\|\u(nh)-\u_{\infty}\|_{\H}^2-2\int_{nh}^t\langle\B(\u(s))-\B(\u_{\infty}),\u(s)-\u_{\infty}\rangle\d s\nonumber\\&\quad-2\int_{nh}^t\langle\mathcal{C}(\u(s))-\mathcal{C}(\u_{\infty}),\u(s)-\u_{\infty}\rangle\d s\nonumber\\&\quad+2\int_{nh}^t\left((\Phi(s,\u(s))-\Phi(s,\u_{\infty}))\d\W(s),\u(s)-\u_{\infty}\right)\nonumber\\&\quad +\int_{nh}^t\|\Phi(s,\u(s))-\Phi(s,\u_{\infty})\|_{\mathcal{L}_{\Q}}^2\d s.
	\end{align}
	Taking the supremum from $nh$ to $(n+1)h$ and then taking expectation in (\ref{6.17}), we find 
	\begin{align}\label{6.18}
	&\E\left[\sup_{nh\leq t\leq  (n+1)h}\|\u(t)-\u_{\infty}\|_{\H}^2+\mu\int_{nh}^{(n+1)h}\|\nabla(\u(s)-\u_{\infty})\|_{\H}^2\d s  +2\alpha\int_{nh}^{(n+1)h}\|\u(s)-\u_{\infty}\|_{\H}^2\d s \right]\nonumber\\&\leq \E\left[\|\u(nh)-\u_{\infty}\|_{\H}^2\right] +2\varrho\E\left[\int_{nh}^{(n+1)h}\|\u(s)-\u_{\infty}\|_{\H}^2\d s\right]\nonumber\\&\quad +\E\left[\int_{nh}^{(n+1)h}\|\Phi(s,\u(s))-\Phi(s,\u_{\infty})\|_{\mathcal{L}_{\Q}}^2\d s\right]\nonumber\\&\quad +2\E\left[\sup_{nh\leq t\leq  (n+1)h}\left|\int_{nh}^t\left((\Phi(s,\u(s))-\Phi(s,\u_{\infty}))\d\W(s),\u(s)-\u_{\infty}\right)\right|\right],
	\end{align}
	where we have used (\ref{2.27}) and (\ref{2.30}). We estimate the final term in the right hand side of the inequality \eqref{6.18} using Burkholder-Davis-Gundy's, H\"older's and Young's inequalities as 
	\begin{align}\label{6.19}
	&2\E\left[\sup_{nh\leq t\leq  (n+1)h}\left|\int_{nh}^t\left((\Phi(s,\u(s))-\Phi(s,\u_{\infty}))\d\W(s),\u(s)-\u_{\infty}\right)\right|\right]\nonumber\\&\leq 2\sqrt{3}\E\left[\int_{nh}^{(n+1)h}\|\Phi(s,\u(s))-\Phi(s,\u_{\infty})\|_{\mathcal{L}_{\Q}}^2\|\u(s)-\u_{\infty}\|^2_{\H}\d s\right]^{1/2}\nonumber\\&\leq 2\sqrt{3}\E\left[\sup_{nh\leq s\leq  (n+1)h}\|\u(s)-\u_{\infty}\|_{\H}\left(\int_{nh}^{(n+1)h}\|\Phi(s,\u(s))-\Phi(s,\u_{\infty})\|_{\mathcal{L}_{\Q}}^2\d s\right)^{1/2}\right]\nonumber\\&\leq \frac{1}{2}\E\left[\sup_{nh\leq s\leq  (n+1)h}\|\u(s)-\u_{\infty}\|_{\H}^2\right]+6\E\left[\int_{nh}^{(n+1)h}\|\Phi(s,\u(s))-\Phi(s,\u_{\infty})\|_{\mathcal{L}_{\Q}}^2\d s\right].
	\end{align}
	Substituting (\ref{6.19}) in (\ref{6.18}), and then using Hypothesis \ref{hyp} (see (H.2)), we get 
	\begin{align}\label{6.21}
	&\E\left[\sup_{nh\leq t\leq  (n+1)h}\|\u(t)-\u_{\infty}\|_{\H}^2\right]+\vartheta\E\left[\int_{nh}^{(n+1)h}\|\u(s)-\u_{\infty}\|_{\H}^2\d s\right]\nonumber\\&\leq 2\E\left[\|\u(nh)-\u_{\infty}\|_{\H}^2\right],
	\end{align}
	where 
	$$\vartheta=2\left(2\alpha-(2\varrho+7L)\right)>0.$$
	Let us now use (\ref{6.16}) in (\ref{6.21}) to obtain 
	\begin{align}
	\E\left[\sup_{nh\leq t\leq  (n+1)h}\|\u(t)-\u_{\infty}\|_{\H}^2\right]\leq 2\E\left[\|\u_0-\u_{\infty}\|_{\H}^2\right]e^{-\theta nh}.
	\end{align}
	Using Chebychev's inequality, for $\epsilon\in(0,\theta)$, we also have 
	\begin{align}
	&\mathbb{P}\left\{\omega\in\Omega:\sup_{nh\leq t\leq  (n+1)h}\|\u(t)-\u_{\infty}\|_{\H}>e^{-\frac{1}{2}(\theta-\epsilon)nh}\right\}\nonumber\\&\leq e^{(\theta-\epsilon)nh}\E\left[\sup_{nh\leq t\leq  (n+1)h}\|\u(t)-\u_{\infty}\|_{\H}^2\right]\nonumber\\&\leq 2\E\left[\|\u_0-\u_{\infty}\|_{\H}^2\right]e^{-\epsilon nh},
	\end{align}
	and 
	\begin{align}
	&\sum_{n=1}^{\infty}\mathbb{P}\left\{\omega\in\Omega:\sup_{nh\leq t\leq  (n+1)h}\|\u(t)-\u_{\infty}\|_{\H}>e^{-\frac{1}{2}(\theta-\epsilon)nh}\right\}\nonumber\\&\quad\leq 2\E\left[\|\u_0-\u_{\infty}\|_{\H}^2\right]\frac{1}{e^{\epsilon h}-1}<+\infty.
	\end{align}
	Thus by using the Borel-Cantelli lemma, there is a finite integer $n_0(\omega)$ such that
	\begin{align}
	\sup_{nh\leq t\leq  (n+1)h}\|\u(t)-\u_{\infty}\|_{\H}\leq e^{-\frac{1}{2}(\theta-\epsilon)nh},\ \mathbb{P}\text{-a.s.,}
	\end{align}
	for all $n\geq n_0$, which completes the proof.
\end{proof}
\subsection{Stabilization by a multiplicative noise}
It is an interesting question  to ask  the exponential stability of  the stationary solution for small values of $\mu$ or $\alpha$. For 2D stochastic NSE, the authors in \cite{CTRJ1} obtained such a stabilization result with the system perturbed by an one dimensional Wiener process $\W(t)$ and for $\Phi(t,\u(t))=\sigma(\u(t)-\u_{\infty})$, where $\sigma$ is a real number. We apply a similar method to obtain  the stabilization of  CBF equations by using the same multipliative noise. Thus, we have the following stabilization result for SCBF by noise:
\begin{theorem}\label{thm4.7}
Let the system \eqref{32} be perturbed by an one dimensional Wiener process $\W(t)$  with $\Phi(t,\u(t))=\sigma(\u(t)-\u_{\infty})$, where $\sigma$ is a real number. Then, there exists $\Omega_0\subset\Omega$ with $\mathbb{P}(\Omega_0 )=0$, such that for all $\omega\not\in\Omega_0$, there exists $T(\omega) > 0$ such that any strong solution $\u(t)$ to the system (\ref{32}) satisfies
\begin{align}
\|\u(t)-\u_{\infty}\|_{\H}^2\leq \|\u_0-\u_{\infty}\|_{\H}^2e^{-\zeta t}, \ \text{ for any }\ t\geq T(\omega),
\end{align}
where\footnote{For a bounded domain, we can take $\zeta=\frac{1}{2}(\sigma^2+2\mu\lambda_1+ 2\alpha-2\varrho)>0$} $\zeta=\frac{1}{2}(\sigma^2+2\alpha-2\varrho)>0$. In particular, the exponential stability of sample paths  with probability one holds if $\zeta>0$.
\end{theorem}
\begin{proof}
	We know that the process $\u(\cdot)-\u_{\infty}$ satisfies the following energy equality:
	\begin{align}
	&\|\u(t)-\u_{\infty}\|_{\H}^2+2\mu\int_{0}^t\|\nabla(\u(s)-\u_{\infty})\|_{\H}^2\d s +2\alpha\int_{0}^t\|\u(s)-\u_{\infty}\|_{\H}^2\d s \nonumber\\&=\|\u_0-\u_{\infty}\|_{\H}^2-2\int_{0}^t\langle\B(\u(s))-\B(\u_{\infty}),\u(s)-\u_{\infty}\rangle\d s\nonumber\\&\quad-2\int_{0}^t\langle\mathcal{C}(\u(s))-\mathcal{C}(\u_{\infty}),\u(s)-\u_{\infty}\rangle\d s\nonumber\\&\quad+2\int_{0}^t\left((\Phi(s,\u(s)))\d\W(s),\u(s)-\u_{\infty}\right) +\int_{0}^t\|\Phi(s,\u(s))\|_{\mathcal{L}_{\Q}}^2\d s.
	\end{align}
	Applying It\^o's formula to the process $\log\|\u(\cdot)-\u_{\infty}\|_{\H}^2$, we find 
	\begin{align}\label{4.35}
\log\|\u(t)-\u_{\infty}\|_{\H}^2&=\log\|\u_0-\u_{\infty}\|_{\H}^2-2\mu\int_0^t\frac{\|\nabla(\u(s)-\u_{\infty})\|_{\H}^2}{\|\u(s)-\u_{\infty}\|_{\H}^2}\d s 
- 2\alpha \nonumber\\&\quad -2\int_{0}^t\frac{\langle\B(\u(s))-\B(\u_{\infty}),\u(s)-\u_{\infty}\rangle}{\|\u(s)-\u_{\infty}\|_{\H}^2}\d s\nonumber\\&\quad-2\int_{0}^t\frac{\langle\mathcal{C}(\u(s))-\mathcal{C}(\u_{\infty}),\u(s)-\u_{\infty}\rangle}{\|\u(s)-\u_{\infty}\|_{\H}^2}\d s\nonumber\\&\quad+2\int_{0}^t\frac{\sigma\|\u(s)-\u_{\infty}\|_{\H}^2}{\|\u(s)-\u_{\infty}\|_{\H}^2}\d\W(s) +\int_{0}^t\frac{\sigma^2\|\u(s)-\u_{\infty}\|_{\H}^2}{\|\u(s)-\u_{\infty}\|_{\H}^2}\d s\nonumber\\&\quad-\frac{1}{2}\int_0^t\frac{4\sigma^2\|\u(s)-\u_{\infty}\|_{\H}^4}{\|\u(s)-\u_{\infty}\|_{\H}^4}\d s\nonumber\\&\leq\log\|\u_0-\u_{\infty}\|_{\H}^2+(-2\alpha+2\varrho-\sigma^2)t+2\sigma\W(t),
	\end{align}
where we have used (\ref{2.27}) and (\ref{2.30}).	We know that $\lim\limits_{t\to\infty}\frac{\W(t)}{t}=0$, $\mathbb{P}$-a.s. Thus, one can assure the existence of  a set $\Omega_0\subset\Omega$ with $\mathbb{P}(\Omega_0)=0$ such
	that for every $\omega\not\in\Omega_0$, there exists $T(\omega)>0$ such that for all $t\geq T(\omega)$, we have
	$$\frac{2\sigma\W(t)}{t}\leq\frac{1}{2}(2\alpha-2\varrho+\sigma^2).$$
	Hence from (\ref{4.35}), we finally have 
	\begin{align}
	\log\|\u(t)-\u_{\infty}\|_{\H}^2&\leq\log\|\u_0-\u_{\infty}\|_{\H}^2-\frac{1}{2}(\sigma^2+2\alpha-2\varrho)t,
	\end{align}
for any $t\geq T(\omega)$,	which completes the proof.
\end{proof}

\section{Invariant Measures and Ergodicity}\label{se7}\setcounter{equation}{0} In this section, we discuss the existence and uniqueness of invariant measures and ergodicity results for the stochastic convective Brinkman-Forchheimer equations \eqref{32} on bounded domains. The reason for restricting ourselves to bounded domains is that we are using compact Sobolev embedding to prove the result of this section. 

  Let us first provide the definitions of invariant measures,  ergodic, strongly mixing and exponentially mixing invariant measures. Let $\X$ be a Polish space (compete separable metric space).  
\begin{definition}
	A probability measure $\upeta$ on	$(\X,\mathscr{B}(\X))$ is called \emph{an invariant		measure or a stationary measure} for a given transition probability	function $\mathrm{P}(t,\x,\d \y)$ if it satisfies
	$$\upeta(\A)=\int_{\X}\mathrm{P}(t,\x,\A)\d\upeta(\x),$$ for all $\A\in\mathscr{B}(\X)$ and
	$t>0$. Equivalently, if for all $\varphi\in \mathrm{C}_b(\X)$
	(the space of bounded continuous functions on $\X$), and all
	$t\geq 0$,
	$$\int_{\X}\varphi(\x)\d\upeta(\x)=\int_{\X}(\mathrm{P}_t\varphi)(\x)\d\upeta(\x),$$ where the Markov semigroup
	$(\mathrm{P}_t)_{t\geq 0}$ is defined by
	$$\mathrm{P}_t\varphi(\x)=\int_{\X}\varphi(\y)\mathrm{P}(t,\x,\d \y).$$
\end{definition}
\begin{definition}[{\cite[Theorem 3.2.4, Theorem 3.4.2]{GDJZ}, \cite{MSS}}]
	Let $\upeta$ be an invariant measure for $\left(\mathrm{P}_t\right)_{t\geq 0}.$  We say that the measure $\upeta$ is an \emph{ergodic measure,}  if for all $\varphi \in {\L}^2(\X;\upeta), $ we have  $$ \lim_{T\to +\infty}\frac{1}{T}\int_0^T (\mathrm{P}_t\varphi)(\x) \d t =\int_{\X}\varphi(\x) \d\upeta(\x) \ \text{ in } \ {\L}^2(\X;\upeta).$$ The invariant measure $\upeta$ for $\left(\mathrm{P}_t\right)_{t\geq 0}$ is called \emph{strongly mixing} if  for all $\varphi \in {\L}^2(\X;\upeta),$  we have $$\lim_{t\to+\infty}\mathrm{P}_t\varphi(\x) = \int_{\X}\varphi(\x) \d\upeta(\x)\ \text{ in }\ {\L}^2(\X;\upeta).$$ The invariant measure $\upeta$ for $\left(\mathrm{P}_t\right)_{t\geq 0}$ is called \emph{exponentially mixing}, if there exists a constant $k>0$ and a positive function $\Psi(\cdot)$ such that for any bounded Lipschitz function $\varphi$, all $t>0$ and all $\x\in\X$, $$\left|\mathrm{P}_t\varphi(\x)-\int_{\X}\varphi(\x)\d\upeta(\x)\right|\leq \Psi(\x)e^{-k t}\|\varphi\|_{\mathrm{Lip}},$$  where $\|\cdot\|_{\mathrm{Lip}}$ is the Lipschitz constant. 
\end{definition}

	Clearly exponentially mixing implies strongly mixing. \cite[Theorem 3.2.6]{GDJZ} states that if  $\upeta$ is the unique invariant measure for $(\mathrm{P}_t)_{t\geq 0}$, then  it is ergodic. The interested readers are referred to see \cite{GDJZ} for more details on the ergodicity for infinite dimensional systems. 

Let us now show that there exists a unique invariant measure for the Markovian transition probability associated to the  system (\ref{32}). Moreover, we establish that the invariant measure is ergodic and strongly mixing (in fact exponentially mixing). Let $\u(t,\u_0)$ denote the unique strong solution to the system (\ref{32}) with the initial condition $\u_0\in\H$. Let $(\mathrm{P}_t)_{t\geq 0}$ be the \emph{Markovian transition semigroup} in the space $\C_b(\H)$ associated to the system (\ref{32}) defined by
\begin{align}\label{mar}
\mathrm{P}_t\varphi(\u_0)=\E\left[\varphi(\u(t,\u_0))\right]=\int_{\H}\varphi(\y)\mathrm{P}(t,\u_0,\d
\y)=\int_{\H}\varphi(\y)\upeta_{t,\u_0}(\d \y),\;\varphi\in \C_b(\H),
\end{align}
where $\mathrm{P}(t,\u_0,\d \y)$ is the transition probability of $\u(t,\u_0)$ and $\upeta_{t,\u_0}$ is the law of $\u(t,\u_0)$. The semigroup $(\mathrm{P}_t)_{t\geq 0}$ is Feller, since the solution to \eqref{32} depends continuously on the initial data. From (\ref{mar}), we also have
\begin{align}\label{amr}
\mathrm{P}_t\varphi(\u_0)=\left<\varphi,\upeta_{t,\u_0}\right>=\left<\mathrm{P}_t\varphi,\upeta\right>,
\end{align}
where $\upeta$ is the law of the initial data $\u_0\in\H$. Thus from
(\ref{amr}), we have $\upeta_{t,\u_0}=\mathrm{P}_t^*\upeta$. We say that a
probability measure $\upeta$ on $\H$ is an \emph{invariant measure} if
\begin{align}
\mathrm{P}_t^*\upeta=\upeta,\textrm{ for all }\ t\geq 0.
\end{align}
That is, if a solution has law $\upeta$ at some time, then it has the same law for all later times. For such a solution, it can be shown by Markov property that for all $(t_1,\ldots,t_n)$ and $\tau>0$, $(\u(t_1+\tau,\u_0),\ldots,\u(t_n+\tau,\u_0))$ and $(\u(t_1,\u_0),\ldots,\u(t_n,\u_0))$ have the same law. Then, we say that the process $\u$ is \emph{stationary}. For more details, the interested readers are referred to see \cite{GDJZ,ADe}, etc.
\begin{theorem}\label{EIM}
	Let $\u_0\in\H$ be given. Then, under Hypothesis \ref{hyp}, for $\alpha>L$\footnote{One can obtain the same result for $\mu>\frac{L}{\lambda_1}$ also, see Remark \ref{IM-Remark} below.},	there exists an invariant measure for the system (\ref{32}) with support in $\V$.
\end{theorem}
\begin{proof}
	Let us use the energy equality obtained in \eqref{3.63} to find 
	\begin{align}\label{7.20}
	&	\|\u(t)\|_{\H}^2+2\mu \int_0^t\|\nabla\u(s)\|_{\H}^2\d s +2\alpha \int_0^t\|\u(s)\|_{\H}^2\d s +2\beta\int_0^t\|\u(s)\|_{\widetilde{\L}^{r+1}}^{r+1}\d s\nonumber\\&=\|{\u_0}\|_{\H}^2+\int_0^t\|\Phi(s,\u(s))\|_{\mathcal{L}_{\Q}}^2\d s+2\int_0^t(\Phi(s,\u(s))\d\W(s),\u(s)).
	\end{align}
In view of \eqref{Phi-Growth-condition} and \eqref{Phi-Lipschitz-condition}, we have 
\begin{align*}
	\int_0^t\|\Phi(s,\u(s))\|_{\mathcal{L}_{\Q}}^2\d s & \leq 2\int_0^t\|\Phi(s,\u(s)) - \Phi(s,\boldsymbol{0})  \|_{\mathcal{L}_{\Q}}^2\d s + 2 \int_0^t\|\Phi(s,\boldsymbol{0})  \|_{\mathcal{L}_{\Q}}^2\d s
	\nonumber\\ & \leq 2L\int_0^t\|\u(s)\|_{\H}^2\d s + 2 K t,
\end{align*}
which implies form \eqref{7.20} that
\begin{align}\label{7.21}
	&	\|\u(t)\|_{\H}^2+2\mu \int_0^t\|\nabla\u(s)\|_{\H}^2\d s + 2\left(\alpha-L\right) \int_0^t\|\u(s)\|_{\H}^2\d s+2\beta\int_0^t\|\u(s)\|_{\widetilde{\L}^{r+1}}^{r+1}\d s\nonumber\\& \leq \|{\u_0}\|_{\H}^2+ 2K t +2\int_0^t(\Phi(s,\u(s))\d\W(s),\u(s)).
\end{align}
Taking expectation in (\ref{7.21}) and using the fact that the final term is a martingale having zero expectation, we obtain  
	\begin{align}\label{5.4}
	&	\E\left\{	\|\u(t)\|_{\H}^2+2\mu \int_0^t\|\nabla\u(s)\|_{\H}^2\d s + 2\left(\alpha-L\right) \int_0^t\|\u(s)\|_{\H}^2\d s+2\beta\int_0^t\|\u(s)\|_{\widetilde{\L}^{r+1}}^{r+1}\d s \right\}
	\nonumber\\ & \leq
	\E\left[\|\u_0\|_{\H}^2\right] + 2K t.
	\end{align}
	Thus, for $\alpha>L$, we have 
	\begin{align}\label{5.6}
	\frac{2\min\{\mu,(\alpha-L)\}}{t}\E\left[\int_0^{t}\|\u(s)\|_{\V}^2\d s\right]\leq
	\frac{1}{T_0}\|\u_0\|_{\H}^2 +2K, \text{ for all }t>T_0.
	\end{align}
Using Markov's inequality, we get
	\begin{align}\label{5.7}
	\lim_{R\to\infty}\sup_{T>T_0}\left[\frac{1}{T}\int_0^T\mathbb{P}\Big\{\|\u(t)\|_{\V}>R\Big\}\d
	t\right]&\leq
	\lim_{R\to\infty}\sup_{T>T_0}\frac{1}{R^2}\E\left[\frac{1}{T}\int_0^T\|\u(t)\|_{\V}^2\d
	t\right]=0.
	\end{align}
	Hence along with the estimate in (\ref{5.7}),  using the compactness of $\V$ in $\H$, it is clear by a standard argument that the sequence of probability measures $$\upeta_{t,\u_0}(\cdot)=\frac{1}{t}\int_0^t\Pi_{s,\u_0}(\cdot)\d s,\ \text{ where }\ \Pi_{t,\u_0}(\Lambda)=\mathbb{P}\left(\u(t,\u_0)\in\Lambda\right), \ \Lambda\in\mathscr{B}(\H),$$ is tight, that is, for each $\updelta>0$, there is a compact subset $K\subset\H$  such that $\upeta_t(K^c)\leq \updelta$, for all $t>0$, and so by the Krylov-Bogoliubov theorem (or by a  result of Chow and Khasminskii see \cite{CHKH}) $\upeta_{t_n,\u_0}\to\upeta,$ weakly for $n\to\infty$, and $\upeta$ results to be an invariant measure for the transition semigroup $(\mathrm{P}_t)_{t\geq 0}$ defined by 	$$\mathrm{P}_t\varphi(\u_0)=\E\left[\varphi(\u(t,\u_0))\right],$$ for all $\varphi\in\C_b(\H)$, where $\u(\cdot)$ is the unique strong solution to the system (\ref{32}) with the initial condition $\u_0\in\H$.
\end{proof}

\begin{remark}\label{IM-Remark}
	One can obtain inequality \eqref{5.7} for $\mu>\frac{L}{\lambda_1}$ in the following way, where $\lambda_1$ is the constant appearing in \eqref{poin}. In view of \eqref{Phi-Growth-condition}, \eqref{Phi-Lipschitz-condition} and Poincar\'e inequality, we have 
	\begin{align*}
		\int_0^t\|\Phi(s,\u(s))\|_{\mathcal{L}_{\Q}}^2\d s &  \leq 2L\int_0^t\|\u(s)\|_{\H}^2\d s + 2 K t
			  \leq \frac{2L}{\lambda_1}\int_0^t\|\nabla\u(s)\|_{\H}^2\d s + 2 K t,
	\end{align*}
	which implies form \eqref{7.20} that
	\begin{align}\label{7.21-1}
		&	\|\u(t)\|_{\H}^2 + 2\left(\mu-\frac{L}{\lambda_1}\right) \int_0^t\|\nabla\u(s)\|_{\H}^2\d s    \leq \|{\u_0}\|_{\H}^2+ 2K t +2\int_0^t(\Phi(s,\u(s))\d\W(s),\u(s)).
	\end{align}
	Taking expectation in (\ref{7.21-1}) and using the fact that the final term is a martingale having zero expectation, we obtain 
	\begin{align}\label{5.4-1}
		&	\E\left\{	\|\u(t)\|_{\H}^2+2\left(\mu-\frac{L}{\lambda_1}\right)  \int_0^t\|\nabla\u(s)\|_{\H}^2\d s  \right\}
		  \leq
		\E\left[\|\u_0\|_{\H}^2\right] + 2K t.
	\end{align}
	Thus,  for $\mu>\frac{L}{\lambda_1}$, we have 
	\begin{align}\label{5.6-1}
		\frac{2}{t}\left(\mu-\frac{L}{\lambda_1}\right)\E\left[\int_0^{t}\|\nabla\u(s)\|_{\H}^2\d s\right]\leq
		\frac{1}{T_0}\|\u_0\|_{\H}^2 +2K, \text{ for all }t>T_0.
	\end{align}
	Using Markov's inequality, we get
	\begin{align}\label{5.7-1}
		\lim_{R\to\infty}\sup_{T>T_0}\left[\frac{1}{T}\int_0^T\mathbb{P}\Big\{\|\u(t)\|_{\V}>R\Big\}\d
		t\right]
		&\leq
	\left[1+\frac{1}{\lambda_1}\right]	\lim_{R\to\infty}\sup_{T>T_0}\frac{1}{R^2}\E\left[\frac{1}{T}\int_0^T\|\nabla\u(t)\|_{\H}^2\d
		t\right]=0.
	\end{align}
\end{remark}

Now we establish the uniqueness of invariant measure for the system (\ref{32}).  Similar results for 2D stochastic NSE is established in \cite{ADe} and for stochastic 2D Oldroyd models are obtained in \cite{MTM6}. The following result provides the exponential stability results for the system \eqref{32}. 

\begin{theorem}\label{exps1}
	Let $\u(\cdot)$ and $\v(\cdot)$ be two solutions of the system (\ref{32}) with $r>3$ and the initial data $\u_0,\v_0\in\H$, respectively. Then, for $\mu\lambda_1+ 2\alpha > 2\varrho+L$, we have 
	\begin{align}\label{513}
	\E\left[\|\u(t)-\v(t)\|_{\H}^2\right]\leq \|\u_0-\v_0\|_{\H}^2e^{-(\mu\lambda_1+2\alpha-(2\varrho+L))t},
	\end{align}
	where $\varrho$ is defined in \eqref{215}. 
\end{theorem}
\begin{proof}
	Let us define $\w(t)=\u(t)-\v(t)$. Then, $\w(\cdot)$ satisfies the following energy equality: 
	\begin{align}\label{514}
	\|\w(t)\|_{\H}^2&=\|\w_0\|_{\H}^2-2\mu\int_0^t\|\nabla\w(s)\|_{\H}^2\d s-2\alpha\int_0^t\|\w(s)\|_{\H}^2\d s
	\nonumber\\ & \quad -2\beta\int_0^t\langle
	\mathcal{C}(\u(s))-\mathcal{C}(\v(s)),\w(s)\rangle\d s -2\int_0^t\left<\B(\u(s))-\B(\v(s)),\w(s)\right>\d s
	\nonumber\\&\quad +\int_0^{t}\|\wi\Phi(s)\|_{\mathcal{L}_{\Q}}^2\d s +2\int_0^{t}(\widetilde{\Phi}(s)\d\W(s),\w(s)),
	\end{align}
	where $\widetilde{\Phi}(\cdot)=\Phi(\cdot,\u(\cdot))-\Phi(\cdot,\v(\cdot))$. Taking expectation in \eqref{514} and then using the Poincar\'e inequality \eqref{poin}, Hypothesis (H.3), \eqref{2.27} and (\ref{2.30}), one can easily see that 
	\begin{align}\label{515}
	\E\left[	\|\w(t)\|_{\H}^2\right]\leq\|\w_0\|_{\H}^2-(\mu\lambda_1+2\alpha-(2\varrho+L))\int_0^t\E\left[\|\w(s)\|_{\H}^2\right]\d s,
	\end{align}
	where $\varrho$ is defined in \eqref{215}. 	Thus, an application of the Gr\"onwall's inequality yields 
	\begin{align}
	\E\left[	\|\w(t)\|_{\H}^2\right]\leq \|\w_0\|_{\H}^2e^{-(\mu\lambda_1+2\alpha -(2\varrho+L))t},
	\end{align}
	and for $\mu\lambda_1+ 2\alpha > 2\varrho+L$, and we obtain the required result \eqref{513}. 
\end{proof}
\begin{remark}
For $2\beta\mu\geq  1$, the results obtained in Theorem \ref{exps1} can be established for $\mu>\frac{L}{\lambda_1}$, using the estimate \eqref{2.26}. Therefore, for $\mu>\max\left\{\frac{1}{2\beta},\frac{L}{\lambda_1}\right\}$ also, the exponential stability result given in Theorem \ref{exps1} holds. 
	\end{remark}
	Let us now establish the uniqueness of invariant measures for the system \eqref{32} obtained in Theorem \ref{EIM}. We prove the case of $r>3$ only and the case of $r=3$ follows similarly.

\begin{theorem}\label{UEIM}
	Let the conditions given in Theorem \ref{exps1} hold true and $\u_0\in\H$ be given. Then, for  $\mu\lambda_1+ 2\alpha > 2\varrho+L$, there is a unique invariant measure $\upeta$ to system (\ref{32}). The measure $\upeta$ is ergodic and strongly mixing, that is, 
	\begin{align}\label{6.9a}
		\lim_{t\to\infty}\mathrm{P}_t\varphi(\u_0)=\int_{\H}\varphi(\v_0)\d\upeta(\v_0), \ \upeta\text{-a.s., for all }\ \u_0\in\H\ \text{ and }\  \varphi\in\C_b(\H).
	\end{align} 
\end{theorem}
\begin{proof}
For $\varphi\in \text{Lip}(\H)$ (Lipschitz $\varphi$), since $\upeta$ is an invariant measure, we have 
	\begin{align}
	&	\left|\mathrm{P}_t\varphi(\u_0)-\int_{\H}\varphi(\v_0)\upeta(\d \v_0)\right|\nonumber\\&=	\left|\E[\varphi(\u(t,\u_0))]-\int_{\H}\mathrm{P}_t\varphi(\v_0)\upeta(\d \v_0)\right|\nonumber\\&=\left|\int_{\H}\E\left[\varphi(\u(t,\u_0))-\varphi(\u(t,\v_0))\right]\upeta(\d \v_0)\right|\nonumber\\&\leq L_{\varphi}\int_{\H}\E\left[\left\|\u(t,\u_0)-\u(t,\v_0)\right\|_{\H}\right]\upeta(\d \v_0)\nonumber\\&\leq  L_{\varphi} e^{-\frac{(\mu\lambda_1+2\alpha-(2\varrho+L))t}{2}}\int_{\H}\|\u_0-\v_0\|_{\H}\upeta(\d \v_0)\nonumber\\&\leq  L_{\varphi}e^{-\frac{(\mu\lambda_1+2\alpha-(2\varrho+L))t}{2}}\left(\|\u_0\|_{\H}+\int_{\H}\|\v_0\|_{\H}\upeta(\d \v_0)\right)\nonumber\\&\to 0\ \text{ as } \ t\to\infty,
	\end{align}
	since $\int_{\H}\|\v_0\|_{\H}\upeta(\d \v_0)<+\infty$. Hence, we deduce (\ref{6.9a}), for every $\varphi\in \C_b (\H)$, by the density of $\text{Lip}(\H)$ in $\C_b (\H)$. Note that, we have a stronger result that $\mathrm{P}_t\varphi(\u_0)$ converges exponentially fast to equilibrium, which is the exponential mixing property. This easily gives uniqueness of the invariant measure also. Indeed, if  $\wi\upeta$ is an another invariant measure, then
	\begin{align}
	&	\left|\int_{\H}\varphi(\u_0)\upeta(\d \u_0)-\int_{\H}\varphi(\v_0)\wi\upeta(\d \v_0)\right|\nonumber\\&= \left|\int_{\H}\mathrm{P}_t\varphi(\u_0)\upeta(\d \u_0)-\int_{\H}\mathrm{P}_t\varphi(\v_0)\wi\upeta(\d \v_0)\right|\nonumber\\&=\left|\int_{\H}\int_{\H}\left[\mathrm{P}_t\varphi(\u_0)-\mathrm{P}_t\varphi(\v_0)\right]\upeta(\d \u_0)\wi\upeta(\d \v_0)\right|\nonumber\\&\leq L_{\varphi}e^{-\frac{(\mu\lambda_1+2\alpha-(2\varrho+L))t}{2}}\int_{\H}\int_{\H}\|\u_0-\v_0\|_{\H}\upeta(\d \u_0)\wi\upeta(\d \v_0)\nonumber\\&\to 0\ \text{ as }\  t\to\infty.
	\end{align}
	By Theorem 3.2.6, \cite{GDJZ}, since $\upeta$ is the unique invariant measure for $(\mathrm{P}_t)_{t\geq 0}$, we know that it is ergodic. 
\end{proof}

\begin{remark}\label{rem-unbounded}
	Extending the work of \cite{Brzezniak+Ondrejat+Seidler_2017} on the $bw$-Feller property for stochastic nonlinear beam and wave equations, \cite[Theorem 6.5]{Brzezniak+Motyl+Ondrejat_2017} established the existence of an invariant measure for the stochastic 2D Navier-Stokes equations with multiplicative noise in unbounded domains. From \cite[Proposition 6.2]{Brzezniak+Motyl+Ondrejat_2017}, we infer that the semigroup $\mathrm{P}_t$ is $bw$-Feller, that is, if $\varphi : \H\to\mathbb{R}$  is a bounded sequentially weakly continuous function and $t > 0$, then $\mathrm{P}_t\varphi: \H\to\mathbb{R}$ is also a bounded sequentially weakly continuous function. In particular, if $\u_{0n}\xrightarrow{w}\u_0$ in $\H$,  then $\mathrm{P}_t\varphi(\u_{0n})\to\mathrm{P}_t\varphi(\u_0)$. Following \cite{Brzezniak+Motyl+Ondrejat_2017}, to prove the existence of an invariant measure for the system (\ref{32}) in unbounded domains, the montonicity method might not be helpful to prove $bw$-Feller property as montonicity method needs the strong convergence of the initial data. This will be addressed in future work.
\end{remark}

 \medskip\noindent
{\bf Data availability:}  Data sharing not applicable to this article as no datasets were generated or analysed during the current study.

 \medskip\noindent
{\bf Acknowledgments:} K. Kinra is funded by national funds through the FCT - Funda\c c\~ao para a Ci\^encia e a Tecnologia, I.P., under the scope of the project  UIDB/00297/2020 and UIDP/00297/2020 (Center for Mathematics and Applications). M. T. Mohan would  like to thank the Department of Science and Technology (DST) Science $\&$ Engineering Research Board (SERB), India for a MATRICS grant (MTR/2021/000066). The second author would also like to thank Prof. J. C. Robinson, University of Warwick for useful discussions.

\end{document}